\documentclass{amsart}

\usepackage{amsfonts,amscd,amssymb,mathtools}
\usepackage{tikz,tikz-cd}

\numberwithin{equation}{section}
\newtheorem{theorem}{Theorem}[section]
\newtheorem{Lemma}{Lemma}[section]
\newtheorem{Prop}{Propositon}[section]
\newtheorem{coro}{Corollary}[theorem]
\newtheorem{remark}{Remark}[section]

\newtheorem{definition}{Definition}[section]

\newcommand*{\LargerCdot}{\raisebox{-0.25ex}{\scalebox{1.3}{$\cdot$}}}

\def\Diff{\text{\rm Diff}}
\def\diff{\text{\rm diff}}
\def\SDiff{\text{\rm SDiff}}

\def\ap{\text{\rm ap}}

\def\id{\text{\rm id}}

\def\Div{\mathop{\rm div}}
\def\Exp{\text{\rm Exp}}

\def\m{{\rm m}}
\def\R{\mathbb{R}}
\def\Z{\mathbb{Z}}
\def\C{\mathbb{C}}

\DeclareMathOperator{\g}{\mathfrak{g}}

\begin{document}

\title[On the group of almost periodic diffeomorphisms]{On the group of almost periodic diffeomorphisms and its
exponential map}

\author{Xu Sun}
\address{Department of Mathematics, Northeastern University, Boston, MA 02115}
\email{sun.xu@husky.neu.edu}

\author{Peter Topalov}
\address{Department of Mathematics, Northeastern University, Boston, MA 02115}
\email{p.topalov@northeastern.edu}
\thanks{P.T. is partially supported by the Simons Foundation, Award \#526907}


\date{\today}



\begin{abstract}
We define the group of almost periodic diffeomorphisms on
$\mathbb{R}^n$ and on an arbitrary Lie group. We then study the
properties of its Riemannian and Lie group exponential maps and provide
applications to fluid equations. In particular, we show that there exists a geodesic
of a weak Riemannian metric on the group of almost periodic diffeomorphisms of the line
that consists entirely of conjugate points.
\end{abstract}

\maketitle


\section{Introduction}\label{Introduction}
The notion of an almost periodic function was first introduced by Bohr \cite{Bohr1,Bohr2}.
Bohr's idea was to generalize periodic functions of one variable. By definition, a continuous function
$f : \R\to\R$ is {\em almost periodic} (or, following Bohr's original terminology, {\em uniformly almost periodic})
if for any given $\varepsilon>0$ there exists $L_\varepsilon>0$ such that in any interval of length $L_\varepsilon$ there
exists $T\in\R$ so that for any $x\in\R$,
\[
|f(x+T)-f(x)|\le\varepsilon.
\]
The number $T$ is called an $\varepsilon$-{\em almost period}. Almost periodic functions were further studied by Besicovitch,
Bochner, von Neumann, Wiener, Weyl and many others. These functions feature many remarkable 
properties. For example, as in the periodic case, almost periodic functions have Fourier expansions but their Fourier exponents
are not necessarily multiples of a given number. Moreover, in contrast to the space of periodic functions on 
the line, the space of almost periodic functions is {\em not} separable. Almost periodic functions can be defined 
in a similar way on $\R^n$ (see Definition \ref{def:bohr'}) as well as on abstract groups (see \cite{vonNeumann}).
We refer to Section \ref{sec:spaces} for a detailed account on almost periodic functions.

Our interest in almost periodic functions stems from the study of non-linear partial differential equations in
spaces of functions on $\R^n$ with non-vanishing conditions at infinity. Such an investigation was initiated in
\cite{McOT1,McOT2,McOT3}, where the case of functions with asymptotic expansions at infinity was studied.
From this point of view, almost periodic functions are interesting since they are not $L^p$-summable ($0<p<\infty$) but 
display features of strict ergodicity that allow averaging over non-compact spaces. In this way, almost periodic quantities
may have finite averaged energy and conservation laws (see below). Many natural phenomena in physics, fluid and gas dynamics 
exhibit almost periodic behavior: we only mention the theory of disordered systems, almost periodic fluids, Vlasov systems, 
cellular automata (see e.g. \cite{Romerio,HPP,Spohn,JM,HK,KR}). In addition, almost periodic functions appear naturally in
completely integrable Hamiltonian systems. In fact, by Arnold-Liouville theorem, the solutions of a completely integrable
Hamiltonian system that lie on regular Liouville tori are almost periodic curves in the phase space. The same phenomenon 
appears in infinite dimensions (see e.g. \cite[Proposition 4.6]{KT1} for the almost periodic behavior of the solutions of the 
Korteweg-de Vries equation). Moreover, by the Kolmogorov-Arnold-Moser theory, small perturbations of Liouville tori with
Diophantine frequency vectors are preserved and the corresponding solutions have almost periodic behavior.
This persistence of almost periodic solutions under perturbations indicates that the phenomenon of almost periodicity is not isolated.

The main purpose of this paper is to define the groups of almost periodic diffeomorphisms and to study their main properties. 
We provide several applications to fluid equations.

For any real exponent $m\ge 0$ we define in Section \ref{sec:spaces} the space of $C^m$-almost periodic functions 
$C^m_\ap(\R^n,\R)$ that consists of functions $f$ in the H{\"o}lder space $C^m_b(\R^n,\R)$ such that
the set $\mathcal{S}_f=\{f_c(x)=f(x+c)\}_{c\in\R^n}$ is precompact in $C^m_b(\R^n,\R)$.
The space $C^m_\ap(\R^n,\R)$ is a Banach algebra with respect to pointwise multiplication of functions and
inherits many of the properties of the classical almost periodic functions. The reason we consider non-integer 
H{\"o}lderian exponents $m\ge 0$ stems from the properties of the Laplace operator on $\R^n$ which are crucial for
our applications to fluid flows. Inspired by Arnold's approach to fluid dynamics \cite{Arnold,EM}, we define in 
Section \ref{sec:groups_almost_periodic} the group $\Diff^m_{\ap}(\R^n)$ of $C^m$-almost periodic
diffeomorphisms of $\R^n$ with H{\"o}lderian exponent $m\ge 1$. Generally speaking,  $\Diff^m_{\ap}(\R^n)$ consists 
of $C^m$-diffeomorphisms of $\R^n$ whose infinitesimal generators are almost periodic vector fields of class 
$C^m_\ap(\R^n,\R^n)$ -- see Section \ref{sec:groups_almost_periodic} for detail.
The group $\Diff^m_{\ap}(\R^n)$ is a Banach manifold modeled on $C^m_\ap(\R^n,\R^n)$.
We prove that $\Diff^m_{\ap}(\R^n)$ is a topological group with respect to the composition of diffeomorphisms 
(Theorem \ref{th:group_regularity}). The composition in $\Diff^m_{\ap}(\R^n)$ has additional regularity properties 
that are similar to the regularity properties of  the groups of diffeomorphisms of Sobolev class $H^s$ (cf. \cite{EM,IKT}).
This allows us to define the Fr{\'e}chet-Lie group of $C^\infty$-almost periodic diffeomorphisms
$\Diff^\infty_\ap(\R^n)=\bigcap_{m\ge 1}\Diff^m_\ap(\R^n)$ equipped with the Fr{\'e}chet topology induced by 
the norms in $C^m_\ap(\R^n,\R^n)$. We prove that $\Diff^\infty_\ap(\R^n)$ has a well-defined 
$C^1_F$-smooth\footnote{The symbol $C^1_F$ refers to $C^1$-maps in Fr{\'e}chet calculus -- see Appendix C.}
{\em Lie group exponential map}
\begin{equation}\label{eq:exp_lg_introduction}
\Exp_{\rm LG} : C^\infty_\ap(\R^n,\R^n)\to\Diff^\infty_\ap(\R^n)
\end{equation}
such that $d_0\Exp_{\rm LG}=\id_{C^\infty_\ap}$ and for any vector field $c$ on $\R^n$ with constant components, the differential
$d_c\Exp_{\rm LG}$ of \eqref{eq:exp_lg_introduction} at $c$ has a non-trivial kernel in $C^\infty_\ap(\R^n,\R^n)$. 
This implies that the Lie group exponential map in $\Diff^\infty_\ap(\R^n)$ is {\em not} a $C^1_F$-smooth diffeomorphism
onto its image when restricted to any open neighborhood of zero in $C^\infty_\ap(\R^n,\R^n)$ 
(see Corollary \ref{coro:lie_group_exponential_map}). In contrast to the case of $C^\infty$-smooth diffeomorphisms 
on the circle considered in \cite[Counterexample 5.5.2]{Hamilton} the set of singular points of the map 
\eqref{eq:exp_lg_introduction} is {\em not} countable. A Riemannian counterpart of this example is given in 
Proposition \ref{eq:burgers_special_geodesics} (cf. Remark \ref{rem:burgers} below), where it is proved that 
{\em there exists a geodesic of a weak Riemannian metric on the group of almost periodic diffeomorphisms $\Diff^\infty_\ap(\R)$ that 
consists {\em entirely} of points conjugate to the identity}.

In Section \ref{sec:exponential_map} we consider the weak Riemannian metric $\nu_\alpha$ on $\Diff^\infty_\ap(\R^n)$
obtained by right translations of the scalar product defined on the tangent space to identity 
$T_\id\big( \Diff^\infty_\ap(\R^n)\big)\equiv C^\infty_\ap(\R^n,\R^n)$ by the averaged quantity
\begin{equation}\label{eq:scalar_product_introduction}
\langle u,v\rangle_\alpha:=\lim_{T\to\infty}\frac{1}{(2T)^n}\int_{[-T,T]^n}\big((I-\alpha^2\Delta)u,v\big)_{\R^n}\,dx,
\end{equation}
where $(\cdot,\cdot)_{\R^n}$ is the Euclidean scalar product in $\R^n$, $\Delta$ is the Euclidean Laplacian
$\Delta=\frac{\partial^2}{\partial x_1^2}+...+\frac{\partial^2}{\partial x_n^2}$ acting component-wise on vector fields on $\R^n$, 
$I$ is the $n\times n$ identity matrix and $\alpha\ge 0$ is a given constant. 
Note that the limit in \eqref{eq:scalar_product_introduction} exists since the integrand is an almost periodic function 
(see Section \ref{sec:spaces}). The least action principle applied to the energy functional
\[
\mathcal{E}(\varphi)=\frac{1}{2}\int_0^1\nu_\alpha(\dot\varphi,\dot\varphi)\,dt,\quad
\varphi\in C^2_F\big([0,1],\Diff^\infty_\ap(\R^n)\big),
\]
leads to a geodesic equation on $T\big(\Diff^\infty_\ap(\R^n)\big)\equiv\Diff^\infty_\ap(\R^n)\times C^\infty_\ap(\R^n,\R^n)$ which,
when written in Eulerian coordinates $u={\dot\varphi}\circ\varphi^{-1}$, takes the form
\begin{equation}\label{eq:ch_introduction}
\left\{
\begin{array}{l}
\m_t+(\nabla\m)\cdot u+\big((\Div u)\cdot I+(\nabla u)^T\big)\cdot\m=0,\quad\m := (1-\alpha^2\Delta)u\\
u|_{t=0}=u_0
\end{array}
\right. 
\end{equation}
where $\nabla u$ denotes the Jacobian matrix of $u\in C^\infty_\ap(\R^n,\R^n)$. Our main result in 
Section \ref{sec:exponential_map} is the following

\begin{theorem}\label{th:exp_local_diffeomorphism}
There exist an open neighborhood $V$ of zero in $C^\infty_\ap(\R^n,\R^n)$ and an open neighborhood $U$ of the identity
in $\Diff^\infty_\ap(\R^n)$ such that the {\em Riemannian exponential map} of $\nu_\alpha$ with $\alpha\ne 0$,
\[
\Exp : V\to U,
\] 
is well-defined and, in contrast to the Lie group exponential map \eqref{eq:exp_lg_introduction}, is a $C^1_F$-diffeomorphism.
\end{theorem}

\begin{remark}
In fact, $\Exp : V\to U$ is a $C^k_F$-smooth map for any $k\ge 1$.
\end{remark}

This Theorem generalizes the main results in \cite{KLT,CK,CKKT} to the case of almost periodic diffeomorphisms of $\R^n$.
As a side result we also obtain that for non-integer exponents $m\ge 3$, $m\notin\Z$, the geodesic equation \eqref{eq:ch_introduction} 
is locally well-posed in the space of $C^m$-almost periodic functions (see Theorem \ref{th:ch_m} in Section \ref{sec:exponential_map}). 
In particular, by taking $n=1$, we obtain from Theorem \ref{th:ch_m} and Remark \ref{rem:n=1} that 
the Camassa-Holm equation \cite{CH,CE,Con,Mis,McOT1} is locally well-posed in the space of almost periodic functions $C^m_\ap(\R,\R)$ 
for any real $m\ge 3$. More specifically, we have 

\begin{theorem}\label{th:ch}
For any $u_0\in C^m_\ap(\R,\R)$ with $m\ge 3$ there exists $T>0$ and a unique solution 
$u\in C\big([0,T],C^m_\ap(\R,\R)\big)\cap C^1\big([0,T],C^{m-1}_\ap(\R,\R)\big)$ of the Camassa-Holm equation that depends 
continuously on the initial data $u_0\in C^m_\ap(\R,\R)$ (cf. Theorem \ref{th:ch_m}).
\end{theorem}

An analog of this Theorem in $C_\ap^\infty(\R,\R)$ follows immediately from Theorem \ref{th:exp_local_diffeomorphism}. 
We have

\begin{theorem}\label{th:ch_frechet}
For any $u_0\in C_\ap^\infty(\R,\R)$ there exists $T>0$ and a unique solution $u\in C^1_F\big([0,T],C^\infty_\ap(\R,\R)\big)$
of the Camassa-Holm equation that depends continuously on the initial data $u_0\in C^\infty_\ap(\R,\R)$.
\end{theorem}

Note that in contrast to \cite{KLT} the method of proof of Theorem \ref{th:exp_local_diffeomorphism} is not based on elliptic
or tame estimates (\cite{Hamilton,Omori}).
As a direct consequence of Theorem \ref{th:exp_local_diffeomorphism} and the right invariance of the weak Riemannian metric
$\nu_\alpha$ we also obtain the following

\begin{coro}\label{coro:gauss_lemma}
There exists an open neighborhood $U$ of the identity in $\Diff^\infty_\ap(\R^n)$ such that if $\varphi\circ\psi^{-1}\in U$ then 
there exists a unique minimal among the curves lying in $U\circ\psi$ geodesic $\zeta$ of the weak Riemannian metric $\nu_\alpha$ with 
$\alpha\ne 0$ that connects the points $\varphi$ and $\psi$. Any other $C^1_F$-smooth curve minimal among the curves lying in 
$U\circ\psi$ and connecting these two points coincides with $\zeta$ when considered as a set in $\Diff^\infty_\ap(\R^n)$.
\end{coro}

Let $U$ be an open neighborhood in $\Diff^\infty_\ap(\R^n)$.
A $C^1_F$-smooth curve lying in $U$ and connecting two given points in $U$ is called {\em minimal
among the curves lying in $U$} if its length is less than or equal to the length of any other $C^1_F$-smooth curve lying in $U$ and 
connecting the points. Here we refer to \cite{BBHM} and \cite{EK} for examples of weak Riemannian metrics with smooth 
Riemannian exponential map but vanishing geodesic distance.

In Section \ref{sec:group_almost_periodic_general} we generalize the construction of the group of almost periodic diffeomorphisms 
from $\R^n$ to finite dimensional Lie groups. This generalization is very natural and requires a separate study. 
Note also that although almost periodic functions have been studied for almost a century, to our best knowledge the corresponding 
groups of almost periodic transformations have not been considered earlier.

\medskip

Several additional remarks on the case when $\alpha=0$ are in order.

\begin{remark}\label{rem:burgers}
If we set $\alpha=0$ and $n=1$ in equation \eqref{eq:ch_introduction} then we obtain the inviscid Burgers equation $($\cite{Burgers}$)$
$u_t+3 u_x u=0$, $u|_{t=0}=u_0$, where $u_0\in C^\infty_\ap(\R,\R)$. Then, the method of characteristics applies,
and one sees that
\begin{equation}\label{eq:burgers_solution}
u(t)=u_0\circ\psi(t)^{-1}
\end{equation}
where $\psi(t,x)=x+3 u_0(x) t$ and $\psi\in C^1_F\big([0,T],\Diff^\infty_\ap(\R)\big)$ for $T>0$ sufficiently 
small.\footnote{In fact, the maximal time of existence is equal to $1/\rho_0$ where $\rho_0=-\min\big\{0,\inf_\R(3 u_{0x})\big\}$.}
By Theorem \ref{th:ode} and Remark \ref{rem:ode_dependence_on_u}, the equation ${\dot\varphi}=u(t)\circ\varphi$, 
$\varphi|_{t=0}=\id$, has a unique solution $\varphi\in C^1_F\big([0,T],\Diff^\infty_\ap(\R)\big)$ 
such that $\varphi(T)$ depends $C^1_F$-smoothly on the initial data $u_0\in C^\infty_\ap(\R,\R)$. 
This implies that there exist an open neighborhood $U$ of zero in $C^\infty_\ap(\R,\R)$
and a well-defined $C^1_F$-smooth Burgers' Riemannian exponential map
\begin{equation}\label{eq:exp_burgers}
\Exp : U\to\Diff^\infty_\ap(\R)
\end{equation}
associated to the weak Riemannian metric $\nu_0$ on $\Diff^\infty_\ap(\R)$.
By arguing as in the proof of Theorem \ref{th:lie_group_exponential_map} $($cf. \cite{CK,CKKT}$)$, one then sees that 
$d_0\Exp=\id_{C^\infty_\ap}$ and for any $c\in\R$, $c\ne 0$, the differential $d_c\Exp$ of \eqref{eq:exp_burgers} at $c$
has a non-trivial kernel in $C^\infty_\ap(\R,\R)$ (see the discussion after Corollary \ref{coro:lie_group_exponential_map} 
and Proposition \ref{eq:burgers_special_geodesics} in Section \ref{sec:groups_almost_periodic} for details). 
This implies that for any $0<t\le T$ the point $\Exp(c t)=\id+c t$ is conjugate to $\id$. 
Hence, {\em the whole geodesic $\Exp(c t)$, $0<t\le T$, consists of points conjugate to identity}. 
This is in sharp contrast with the Burgers' Riemannian exponential map on the group of diffeomorphisms $\Diff^\infty(\mathbb{T})$ 
of the circle $($cf. \cite{CK,CKKT}$)$ where the set of conjugate points along the geodesic $\Exp(c t)$, $0< t\le T$, is countable.
Finally, note that it follows from \eqref{eq:burgers_solution}, Theorem \ref{th:group_regularity} and Lemma \ref{lem:general}
in Appendix A, that the solution of the Burgers equation belongs to 
$C\big([0,T],C^m_\ap(\R,\R)\big)\cap C^1\big([0,T],C^{m-1}_\ap(\R,\R)\big)$
and depends continuously on $u_0\in C^m_\ap(\R,\R)$ for any real $m\ge 2$.
\end{remark}

\begin{remark}\label{rem:euler}
For any real $m\ge 1$, $n\ge 2$, define the group of {\em  volume preserving almost periodic diffeomorphisms} 
\[
\SDiff^m_\ap(\R^n):=\big\{\varphi\in\Diff^m_\ap(\R^n)\,\big|\,\forall x\in\R^n,\,\det[d_x\varphi]=1\big\}.
\]
Then, the least action principle applied to the restriction of the weak Riemannian metric $\nu_\alpha$ with $\alpha=0$ 
to $\SDiff^m_\ap(\R^n)$ leads, in Eulerian coordinates, to the {\em incompressible Euler equation} on the space of $
C^m$-almost periodic vector fields $C^m_\ap(\R^n,\R^n)$. Many of the results proved in this paper hold also in this case. 
Note also that the solutions of the Euler equation have bounded averaged energy
\[
E(u):=\lim_{T\to\infty}\frac{1}{(2T)^n}\int_{[-T,T]^n}\big(u,u\big)_{\R^n}\,dx
\]
that is independent of $t$ in the domain of definition of $u$. We will treat the Euler case in a separate work
$($cf. \cite{Giga} for related results for the Navier-Stokes equation$)$.
\end{remark}

The results above can be readily translated to the group $\diff^m_b(\R^n)$, $m\ge 1$, that consists of $C^m$-diffeomorphisms
of $\R^n$ whose infinitesimal generators are vector fields of class $c^m_b(\R^n,\R^n)$ where $c^m_b(\R^n,\R^n)$
denotes the closure of $C^\infty_b(\R^n,\R^n)$ in $C^m_b(\R^n,\R^n)$ (see Section \ref{sec:groups_almost_periodic} and Appendix B 
for details). Note, however, that the averaged scalar product \eqref{eq:scalar_product_introduction} in this case is not well-defined, 
and hence the corresponding equations are {\em not} necessarily geodesic equations of a weak Riemannian metric on $\diff^m_b(\R^n)$.

\medskip

\begin{remark}
Several groups of diffeomorphisms of $\R^n$ were considered in the literature (see e.g. \cite{Cantor,Con,IKT,McOT2,MichorMamford} and 
the references therein). However, for the most part these groups of diffeomoephisms of $\R^n$ are generated by vector fields that decay at infinity 
and, consequently, do not contain shifts. Although shifts, and even more general asymptotic terms, can be included (see for example the group of 
asymptotic diffeomorphisms defined in \cite{McOT1,McOT2,McOT3}), these groups do not have infinitesimal generators with oscillatory behavior 
at infinity. In contrast, the groups of almost periodic diffeomorphisms $\Diff^m_\ap(\R^n)$ and $\Diff^\infty_\ap(\R^n)$ 
considered in this paper have periodic infinitesimal generators for any a priory given period and direction. 
This property allows the existence of  weak Riemannian metrics and geodesics that consist entirely of points conjugate to the identity. 
We refer to Remark \ref{rem:quantization_condition} and Proposition \ref{eq:burgers_special_geodesics} in Section \ref{sec:groups_almost_periodic} for 
detailed discussion of this phenomenon. We would like also to point out that H{\"o}lder regularity with non-integer exponents is especially well suited 
for solving PDEs involving the Euclidean Laplacian on $\R^n$. An important feature of the groups of almost periodic diffeomorphisms is that their 
infinitesimal generators can be averaged over $\R^n$ and hence energy functionals and other conserved quantities can be defined. Finally, we mention 
the recent paper \cite{NenRainer} where (different) groups of diffeomorphisms of H{\"o}lder regularity were studied.
\end{remark}

\medskip

For the convenience of the reader, we conclude this work with three Appendices where we collect several technical 
results used in the main body of the paper. 

\medskip

\noindent{\em Acknowledgments:} The authors of this paper are thankful to the anonymous referees for many valuable comments and 
suggestions that significantly improved the quality of the paper.

\section{Spaces of almost periodic functions with H{\"o}lderian exponents}\label{sec:spaces}
In this Section we collect several classical facts about almost periodic functions and define the spaces of
almost periodic functions used in this paper. As mentioned in the Introduction, the following definition is given by Bohr
\cite{Bohr1,Bohr2}.

\begin{definition}\label{def:bohr}
A continuous function $f\in C(\R,\R)$ is called {\em almost periodic} (in the sense of Bohr) if 
for any $\varepsilon>0$ there exists $L=L_\varepsilon>0$ such that in any interval in
$\R$ of length $L$ there exists $T\equiv T_\varepsilon\in\R$ (called an $\varepsilon$-almost period of $f$) such that
\[
\big|f(x+T)-f(x)\big|\le\varepsilon.
\]
for any $x\in\R$.
\end{definition}
It is not difficult to see from this definition that any almost periodic function is in fact bounded and uniformly continuous on 
$\R$ (see e.g. \cite{Levitan,levitan1982almost}).
Denote by $C_b(\R,\R)$ the Banach space of bounded continuous functions $f : \R\to\R$ equipped with the
the {\em $L^\infty$-norm} $|f|_\infty:=\sup_{x\in\R}|f(x)|$. The following important characterization of 
almost periodic functions was given by Bochner.

\begin{theorem}[Bochner]\label{th:bochner_characterization}
A continuous function $f\in C(\R,\R)$ is almost periodic if and only if the set of functions
\[
\mathcal{S}_f:=\big\{f_c\,\big|\,f_c(x):=f(x+c)\big\}_{c\in\R}
\]
is precompact in $C_b(\R,\R)$.
\end{theorem}
An alternative characterization of almost periodic functions was given by Wiener.
Recall that a trigonometric polynomial is an expression of the form
\begin{equation}\label{eq:trigonometric_polynomial}
P(x)=\sum_{k=1}^N c_k e^{i \lambda_k x}
\end{equation}
where $c_1,...,c_N$ are complex numbers and $\lambda_1,...,\lambda_N$ are real constants.
One has the following

\begin{theorem}[Wiener]\label{th:wiener_characterization}
A function $f : \R\to\R$ is almost periodic if and only if there exists a sequence of trigonometric polynomials
$(P_k)_{k\ge 1}$ such that $P_k\to f$ in $C_b(\R,\R)$.
\end{theorem}

In this way the space of almost periodic functions $C_\ap(\R,\R)$ is the closure of the linear space of
trigonometric polynomials in $C_b(\R,\R)$.
The notion of almost periodic functions can be generalized to functions on $\R^n$ with $n\ge 2$.
(We refer to Section \ref{sec:group_almost_periodic_general} where almost periodic functions on 
finite dimensional Lie group are considered.)

\begin{definition}\label{def:bohr'}
A continuous function $f\in C(\R^n,\R)$ is called almost periodic (in the sense of Bohr) 
if for any $\varepsilon>0$ there exists an $n$-dimensional closed cube $K\equiv K_\varepsilon$ in $\R^n$ such that 
for any $x_0\in\mathbb{R}^n$ there exists $T\equiv T_\varepsilon\in x_0+K$ such that 
\[
\big|f(x+T)-f(x)\big|\le\varepsilon
\]
for any $x\in\R^n$.
\end{definition}

As in the one-dimensional case, any almost periodic function is {\em bounded and uniformly continuous} in $\R^n$.
Moreover, there are analogs of Bochner's and Wiener's characterizations that we
summarize in the following

\begin{theorem}\label{th:almost_periodic_functions_equivalence}
Let $f$ be a bounded continuous function in $\R^n$. Then the following three statements are
equivalent:
\begin{enumerate}
\item[(i)] $f$ is almost periodic;
\item[(ii)] The set $\mathcal{S}_f:=\big\{f_c\,\big|\,f_c(x):=f(x+c)\big\}_{c\in\R^n}$ is precompact in $C_b(\mathbb{R}^n,\R)$;
\item[(iii)] $f$ is a limit of a sequence of trigonometric polynomials in $C_b(\mathbb{R}^n,\R)$;
\end{enumerate}
Moreover, if $f\in C(\R^n,\R)$ is an almost periodic function then it has a well-defined {\em mean value},
\[
{\bar f}:=\lim_{T\to\infty}\frac{1}{(2T)^n}\int_{[-T,T]^n}f(x)\,dx.
\]
\end{theorem}
For the proof we refer e.g. to \cite[Theorem 5.13A.2]{simon2010szegHo}).
A trigonometric polynomial on $\R^n$ is a function that can be written as $P(x)=\sum_{k=1}^N c_k e^{i (\Lambda_k,x)_{\R^n}}$
where $c_1,...,c_N$ are complex numbers, $\Lambda_1,...,\Lambda_N\in\R^n$ are given vectors, and where 
$(\cdot,\cdot)_{\R^n}$ is the Euclidean scalar product in $\R^n$.
Denote by $C_{\ap}(\mathbb{R}^n,\R)$ the space of almost periodic functions on $\mathbb{R}^n$.
In view of Theorem \ref{th:almost_periodic_functions_equivalence} the space $C_{\ap}(\mathbb{R}^n,\R)$ is a closed 
subspace of the Banach space of bounded continuous functions $C_b(\R^n,\R)$ supplied with the maximum norm.

\medskip

Take a real exponent $m\ge 0$. For $m$ {\em integer}, denote by $C^m_b(\mathbb{R}^n,\R)$ 
the space of bounded continuously differentiable functions on $\mathbb{R}^n$ whose partial derivatives of order $\le m$ 
are continuous and bounded in $\mathbb{R}^n$. The space $C^m_b(\mathbb{R}^n,\R)$ is equipped with the norm
\begin{equation}\label{eq:C^m-norm}
|f|_m:=\max_{0\le|\beta|\le m}|\partial^\beta f|_\infty
\end{equation}
where $\beta\in\Z_{\ge 0}^n$ is a multi-index. 
In the case when the exponent $m\ge 0$ is {\em not} integer, 
$m=k+\gamma$ with $k\ge 0$ integer and $0<\gamma<1$, 
$C^m_b(\mathbb{R}^n,\R)$ will denote the H{\"o}lder space of functions $f\in C^k_b(\R^n,\R)$ whose derivatives 
of order $k$ are H{\"o}lder continuous with exponent $\gamma$, equipped with the standard H{\"o}lder norm
\begin{equation}\label{eq:holder_norm}
|f|_m:=|f|_k+\max_{|\beta|=k}[\partial^\beta f]_\gamma\quad\text{where}\quad
[g]_\gamma:=\sup_{x\ne y}\frac{\big|g(x)-g(y)\big|}{|x-y|^\gamma}\,.
\end{equation}
It is a well-known fact that for any real $m\ge 0$ the space $C^m_b(\mathbb{R}^n,\R)$ is a Banach algebra with respect to 
pointwise multiplication of functions (see e.g. \cite[Proposition 2.1.1]{AlinhacGerard}). 

\begin{remark}
The definition of a Banach algebra we adopt in this paper does {\em not} require the absolute constant $C>0$ in 
the inequality $|fg|_m\le C|f|_m |g|_m$ to be equal to one.
\end{remark}

In what follows we will call the norm in 
$C^m_b(\mathbb{R}^n,\R)$ a {\em $C^m$-norm} independently of whether the exponent $m\ge 0$ is integer or not.
Denote by $[m]$ the integer part of $m$, i.e. the greatest integer $k$ such that $k\le m$.

\medskip

For a given $f\in C^m_b(\mathbb{R}^n,\R)$ with $m\ge 0$ real consider the family of functions in
$C^m_b(\mathbb{R}^n,\R)$,
\begin{equation}\label{eq:S_f}
\mathcal{S}_f:=\big\{f_c\in C^m_b(\R^n,\R)\,\big|\,f_c(x):=f(x+c)\big\}_{c\in\R^n}.
\end{equation}
Consider the set of {\em $C^m$-almost periodic} functions
\begin{equation}\label{eq:C^m_ap}
C^m_{\ap}(\mathbb{R}^n,\R):=\big\{f\in C^m_b(\mathbb{R}^n,\R)\,
\big|\,\mathcal{S}_f\,\text{is precompact in}\,\,C^m_b(\R^n,\R)\big\}.
\end{equation}
We will equip $C^m_{\ap}(\mathbb{R}^n,\R)$ with the $C^m$-norm topology inherited from $C^m_b(\R^n,\R)$.
Clearly, $C^{m_2}_{\ap}(\mathbb{R}^n,\R)\subseteq C^{m_1}_{\ap}(\mathbb{R}^n,\R)$
for $0\le m_1\le m_2$ and the inclusion is continuous. By Theorem \ref{th:almost_periodic_functions_equivalence}, 
$C^0_\ap(\mathbb{R}^n,\R)$ coincides with the space $C_\ap(\R^n,\R)$ of almost periodic functions in the sense of Bohr. 
The main result of this Section is Theorem \ref{th:almost_periodic_functions_equivalence_m} below which generalizes 
Theorem \ref{th:almost_periodic_functions_equivalence}.

Note that in contrast to the space $C^m_b(\mathbb{R}^n,\R)$ with integer exponent, the elements of the space
$C^m_b(\mathbb{R}^n,\R)$ with non-integer H{\"o}lderian exponent $m\ge 0$ are {\em not} necessarily approximated in 
the $C^m$-norm by elements from $C^\infty_b(\R^n,\R)=\bigcap_{k\ge 1} C^k_b(\R^n,\R)$.
More specifically, consider the {\em little H{\"o}lder space} $c^m_b(\R^n,\R)$ that is the closure of $C^\infty_b(\R^n,\R)$ in
$C^m_b(\R^n,\R)$. For $m\ge 0$ non-integer $c^m_b(\R^n,\R)$ is a proper closed subspace in 
$C^m_b(\R^n,\R)$ (see Appendix B for details). 
Since $C^m_b(\R^n,\R)$ is a Banach algebra one concludes that $c^m_b(\R^n,\R)$ is also a Banach algebra.
Note that for $m\ge 0$ integer $c^m_b(\R^n,\R)\equiv C^m_b(\R^n,\R)$ (cf. Remark \ref{rem:integer_exponents} in Appendix B). 

\begin{remark}\label{rem:the_lipschitz_case}
For $k\ge 0$ integer we can also consider the space $C_b^{k,1}(\R^n,\R)$ of functions $f$ in $C^k_b(\R^n,\R)$ whose derivatives
$\partial^\beta f$, $|\beta|\le k$, are Lipschitz continuous. When equipped with the norm \eqref{eq:holder_norm} with $\gamma=1$
the space $C_b^{k,1}(\R^n,\R)$ becomes a Banach algebra. The theory of almost periodic functions and diffeomorphisms 
developed in Section \ref{sec:spaces}, Section \ref{sec:groups_almost_periodic}, and Appendix B extends without changes also to this case. 
(Note, however, that the characterization of the corresponding little H{\"o}lder space in Lemma \ref{lem:characterization_little_holder*} in 
Appendix B fails.) Unfortunately, the scale of spaces $C_b^{k,1}(\R^n,\R)$ is not well suited for our main applications since the important result of 
Muhamadiev \cite[Theorem 1]{Muham1} used in Section \ref{sec:exponential_map} as well as analogous results for the ``pure'' Laplace 
operator on $\R^n$ do not hold in this case. For this reason we do not include the case of $C_b^{k,1}(\R^n,\R)$ in the present paper.
\end{remark}

First, we prove the following important

\begin{Prop}\label{prop:approximation}
If $f$ belongs to $C^m_\ap(\mathbb{R}^n,\R)$ with $m\ge 0$ real then there exists a sequence of functions $(f_k)_{k\ge 1}$ in 
$C^\infty_b(\R^n,\R)$ such that $f_k\stackrel{C^m_b}{\to}f$ as $k\to\infty$. In other words, $C_\ap^m(\R^n,\R)\subseteq c_b^m(\R^n,\R)$.
\end{Prop}

\begin{proof}[Proof of Proposition \ref{prop:approximation}]
The case when $m$ is a non-negative integer is trivial since $c^m_b(\R^n,\R)=C^m_b(\R^n,\R)$.
Assume that $m\notin\Z$ and $m\ge 0$.
Take $f\in C^m_b(\mathbb{R}^n,\R)$ and assume that $\mathcal{S}_f$ is precompact in
$C^m_b(\mathbb{R}^n,\R)$. Then $f\in C^{[m]}_b(\R^n,\R)$ and for any multi-index $|\beta|=[m]$
we have $\partial^\beta f\in C^\gamma_b(\R^n,\R)$ where $\gamma:=m-[m]>0$.
Then, in view of Lemma \ref{lem:charachterization_little_holder} in Appendix B, the Proposition will follow once we prove that 
\begin{equation}\label{eq:partial_f}
\partial^\beta f\in c^\gamma_b(\R^n,\R)
\end{equation}
for any $|\beta|=[m]$. To this end, we first note that for any multi-index $|\beta|=[m]$ one has that
$\mathcal{S}_{\partial^\beta f}=\partial^\beta\big(\mathcal{S}_f\big)$.
This, together with the continuity of the map $\partial^\beta : C^m_b(\R^n,\R)\to C^\gamma_b(\R^n,\R)$,
then implies that $\mathcal{S}_{\partial^\beta f}$ is precompact in $C^\gamma_b(\R^n,\R)$.
Hence, \eqref{eq:partial_f} will follow once we prove the statement of the Proposition for $m=\gamma$ where $0<\gamma<1$.
Let us prove it. Assume that $f\in C^\gamma_b(\mathbb{R}^n,\R)$ and that $\mathcal{S}_f$ is precompact in 
$C^\gamma_b(\mathbb{R}^n,\R)$. We will first prove that the map
\begin{equation}\label{eq:the_shift_curve}
c\mapsto f_c,\quad\R^n\to C^\gamma_b(\R^n,\R)
\end{equation}
is continuous. It is enough to prove the claimed continuity at $c=0$.
Assume that the curve above is not continuous at $c=0$. Then there exist $\varepsilon_0>0$ and a sequence 
$(c_j)_{j\ge 1}$ of real numbers such that $c_j\to 0$ as $j\to\infty$ and
\begin{equation}\label{eq:away}
|f_{c_j}-f|_\gamma>\varepsilon_0>0.
\end{equation}
Since $\mathcal{S}_f$ is precompact in $C^\gamma_b(\R^n,\R)$ there exist ${\tilde f}\in C^\gamma_b(\R^n,\R)$ and 
a subsequence $(c_j')_{j\ge 1}$ of $(c_j)_{j\ge 1}$ such that $f_{c_j'}\stackrel{C^\gamma_b}{\to}{\tilde f}$ as $j\to\infty$.
On the other side, the pointwise continuity of $f$ implies that pointwise $f_{c_j'}\to f$ as $j\to\infty$.
This implies that $f\equiv{\tilde f}$ and hence,
\[
f_{c_j'}\stackrel{C^\gamma_b}{\to}f\quad\text{as}\quad j\to\infty.
\]
Since this contradicts \eqref{eq:away}, we conclude that the map \eqref{eq:the_shift_curve} is continuous.
In particular, $[f_c-f]_\gamma\to 0$ as $c\to 0$ or equivalently,
\begin{equation}\label{eq:c-limit}
\sup\limits_{x\in\R^n,h\ne 0}\frac{\big|\big(f(x+c+h)-f(x+c)\big)-\big(f(x+h)-f(x)\big)\big|}{|h|^\gamma}\to 0
\quad\text{as}\quad c\to 0.
\end{equation}
Let $\chi\in C^\infty_c(\R^n,\R)$ be a cut-off function such that $\chi\ge 0$, 
$\mathop{\rm supp}\chi\subseteq\big\{x\in\R^n\,\big|\,|x|\le 1 \big\}$, and $\int_{\R^n}\chi(x)\,dx=1$. 
For any $\varepsilon>0$ denote $\chi_\varepsilon(x):=\chi(x/\varepsilon)/\varepsilon^n$.
We have (see \eqref{eq:holder_norm}),
\[
\begin{array}{l}
[f*\chi_\varepsilon-f]_\gamma=\Big[\int_{\R^n}\big[f(x+y)-f(x)\big]\chi_\varepsilon(-y)\,dy\Big]_\gamma\\
=\sup\limits_{x\in\R^n,h\ne 0}\left|\int_{\R^n}\frac{\big(f(x+y+h)-f(x+y)\big)-\big(f(x+h)-f(x)\big)}{|h|^\gamma}
\,\chi_\varepsilon(-y)\,dy\right|\\
\le\sup\limits_{x\in\R^n,h\ne 0,|y|<\varepsilon}
\frac{\left|\big(f(x+y+h)-f(x+y)\big)-\big(f(x+h)-f(x)\big)\right|}{|h|^\gamma}.
\end{array}
\]
By comparing this with \eqref{eq:c-limit} we see that $[f*\chi_\varepsilon-f]_\gamma\to 0$ as $\varepsilon\to 0+$.
In addition, by the properties of the mollifiers and the fact that $f$ is uniformly continuous (note that $\gamma>0$), 
$f*\chi_\varepsilon\stackrel{C_b}{\to} f$ as $\varepsilon\to 0+$. Hence
\[
f*\chi_\varepsilon\stackrel{C^\gamma_b}{\to}f\quad\text{as}\quad\varepsilon\to 0+.
\]
Since $f*\chi_\varepsilon\in C^\infty_b(\R^n,\R)$ for any $\varepsilon>0$, we conclude the proof
of Proposition \ref{prop:approximation}.
\end{proof}

\begin{remark}
Proposition \ref{prop:approximation} is used in a crucial way in the proof of the continuity of the composition
in Theorem \ref{th:group_regularity} -- see Section \ref{sec:groups_almost_periodic}.
\end{remark}

\begin{remark}
For {\em integer} exponents $m\ge 1$ spaces equivalent to $C^m_\ap(\R^n)$ were studied in the literature 
-- see e.g. \cite{shubin:survey1978}. 
\end{remark}

\begin{remark}\label{rem:approximation_by_mollifiers}
It follows from Proposition \ref{prop:approximation} and Lemma \ref{lem:charachterization_little_holder} in Appendix B that 
for any $f\in C^m_\ap(\R^n,\R)$ we have that $f*\chi_\varepsilon\stackrel{C_b^m}{\to}f$ as $\varepsilon\to 0+$ where 
$\chi_\varepsilon$ is the cut-off function from the proof of Proposition \ref{prop:approximation} and 
$f*\chi_\varepsilon\in C_b^\infty(\R^n,\R)$.  (For integer exponents $m\ge 0$ this follows directly form the properties of 
the mollifiers and the uniform continuity of (classical) almost periodic functions in $C_\ap(\R^n,\R)$.) 
By Remark \ref{rem:approximation_by_mollifiers*} below, 
$f*\chi_\varepsilon\in\bigcap\limits_{k\ge 1, k-\text{integer}}C_\ap^k(\R^n,\R)$.
\end{remark}

We have

\begin{Prop}\label{prop:banach_algebra}
The set of $C^m$-almost periodic functions $C^m_{\ap}(\mathbb{R}^n,\R)$ is a normed linear space with the following 
additional properties:
\begin{itemize} 
\item[(i)] For any real $m\ge 0$ and for any multi-index $\beta\in\Z_{\ge 0}$ such that $0\le|\beta|\le m$ the map 
\[
\partial^{\beta}: C^m_{\ap}(\mathbb{R}^n,\R)\to C^{m-|\beta|}_{\ap}(\mathbb{R}^n,\R),\quad 
f\mapsto\partial^{\beta}f,
\]
is well-defined and continuous.

\item[(ii)] For any real exponents $m_1, m_2\ge 0$ the pointwise multiplication
\[
C^{m_1}_{\ap}(\mathbb{R}^n,\R)\times C^{m_2}_{\ap}(\mathbb{R}^n,\R)\to
C^{\min(m_1, m_2)}_{\ap}(\mathbb{R}^n,\R),\quad (f, g)\mapsto f g,
\]
is a continuous map. 
\end{itemize}
\end{Prop}

\begin{remark}\label{rem:banach_algebra_C^n_b}
Items (i) and (ii) obviously hold for the larger spaces $c^m_b(\R^n,\R)$ and $C^m_b(\R^n,\R)$
-- see e.g. Proposition 2.1.1 in \cite{AlinhacGerard}. 
\end{remark}

\begin{proof}[Proof of Proposition \ref{prop:banach_algebra}]
First, recall that the image of a precompact set under a continuous map is also precompact.
With this in mind, the statement that $C^m_\ap(\R^n,\R)$ is a linear space can be proven as follows:
First, note that for any $f,g\in C^m_\ap(\R^n,\R)$ we have
\begin{equation}\label{eq:S-relation_sum}
\mathcal{S}_{f+g}\subseteq\mathcal{S}_f+\mathcal{S}_g.
\end{equation}
Since $\mathcal{S}_f+\mathcal{S}_g$ is the image of the precompact set $\mathcal{S}_f\times\mathcal{S}_g$ in $
C^m_b(\R^n,\R)\times C^m_b(\R^n,\R)$ under the continuous map
\[
C^m_b(\R^n,\R)\times C^m_b(\R^n,\R)\to C^m_b(\R^n,\R),\quad (f,g)\mapsto f+g,
\] 
we conclude from \eqref{eq:S-relation_sum} that $\mathcal{S}_{f+g}$ is contained in a precompact set in 
$C^m_b(\R^n,\R)$. Hence $\mathcal{S}_{f+g}$ is precompact in $C^m_b(\R^n,\R)$. 
This shows that $f+g\in C^m_\ap(\R^n,\R)$. By arguing in a similar way, one also sees that $C^m_\ap(\R^n,\R)$ is 
invariant under multiplication by scalars. Thus, $C^m_\ap(\R^n,\R)$ is a linear space.
The remaining part of the Proposition can be proven by using similar arguments:
By Remark \ref{rem:banach_algebra_C^n_b}, items (i) and (ii) hold for the larger space $C^m_b(\R^n,\R)$.
Since the space $C^m_{\ap}(\R^n,\R)$ is continuously embedded in $C^m_b(\R^n,\R)$, items (i) and (ii) of 
Proposition \ref{prop:banach_algebra} will follow once we show that the image of the map in (i) consists of $
C^{m-|\beta|}$-almost periodic functions and the image of the map in (ii) consists of $C^{\min(m_1,m_2)}$-almost 
periodic functions. The proofs of these two statements are similar and we will prove only the second one of them.
Take $f\in C^{m_1}_{\ap}(\R^n,\R)$ and $g\in C^{m_2}_{\ap}(\R^n,\R)$ with $m_1, m_2\ge 0$.
Then, 
\begin{equation}\label{eq:S-relation_product}
\mathcal{S}_{f g}\subseteq S_f\cdot S_g:=\big\{F G\,\big|\,F\in\mathcal{S}_f,\,G\in\mathcal{S}_g\big\}.
\end{equation}
The set $S_f\cdot S_g$ is precompact in $C^{\min(m_1,m_2)}_b(\R^n,\R)$ since it is
the image of the precompact set $\mathcal{S}_f\times \mathcal{S}_g$ in 
$C^{m_1}_b(\mathbb{R}^n,\R)\times C^{m_2}_b(\mathbb{R}^n,\R)$ 
under the continuous map
\[
C^{m_1}_b(\mathbb{R}^n,\R)\times C^{m_2}_b(\mathbb{R}^n,\R)\to
C^{\min(m_1, m_2)}_b(\mathbb{R}^n,\R),\quad (F, G)\mapsto F G.
\]
This together with \eqref{eq:S-relation_product} then implies that $\mathcal{S}_{f g}$ is precompact in $C^m_b(\R^n,\R)$,
which completes the proof that $fg\in C^{\min(m_1,m_2)}_{\ap}(\R^n,\R)$.
\end{proof}

\begin{coro}\label{coro:trigonometric_polynomials}
For any real $m\ge 0$ the set of trigonometric polynomials is contained in $C^m_\ap(\R^n,\R)$.
\end{coro}

\begin{proof}[Proof of Corollary \ref{coro:trigonometric_polynomials}]
Since, by Proposition \ref{prop:banach_algebra}, $C^m_\ap(\R^n,\R)$ is a linear space, it is enough to show that 
for any given vector $\Lambda\in\R^n$ the function $x\mapsto e^{i(\Lambda,x)_{\R^n}}$ belongs to $C^m_\ap(\R^n,\C)$. 
In order to prove this we first note that the image of the map $\R^n\to\C$, $c\mapsto e^{i(\Lambda,c)_{\R^n}}$,
is contained inside the unit circle in $\C$, and hence it is precompact in $\C$.
On the other side, the set of constants $\C$ is continuously embedded in $C^m_b(\R^n,\C)$. This implies that
the image of the map $\R^n\to C^m_b(\R^n,\C)$, $c\mapsto e^{i(\Lambda,c)_{\R^n}}$,
is precompact in $C^m_b(\R^n,\C)$. Since the pointwise multiplication of functions in $C^m_b(\R^n,\C)$ is 
continuous, we then conclude from
\[
e^{i(\Lambda,x+c)_{\R^n}}=e^{i(\Lambda,c)_{\R^n}}\cdot e^{i(\Lambda,x)_{\R^n}}
\]
that the image of the map
\[
\R^n\to C^m_b(\R^n,\C),\quad c\mapsto\big[x\mapsto e^{i (\Lambda,x+c)_{\R^n}}\big],
\]
is precompact in $C^m_b(\R^n,\C)$. This, together with the fact that trigonometric polynomials are elements of 
$C^\infty_b(\R^n,\R)$ then implies that the function $x\mapsto e^{i(\Lambda,x)_{\R^n}}$ belongs to $C^m_\ap(\R^n,\C)$.
\end{proof}

We have the following generalization of Theorem \ref{th:almost_periodic_functions_equivalence}.

\begin{theorem}\label{th:almost_periodic_functions_equivalence_m}
For any $n\ge 1$ and for any real exponent $m\ge 0$ the space of $C^m$-almost periodic functions 
$C^m_{\ap}(\R^n,\R)$ is a closed subspace in the Banach space $C^m_b(\R^n,\R)$ such that
$C^m_\ap(\R^n,\R)\subseteq c^m_b(\R^n,\R)\subseteq C^m_b(\R^n,\R)$. Moreover, 
for a given $f\in C^m_b(\R^n,\R)$ the following three statements are equivalent:
\begin{itemize}
\item[(i)] $f\in C^m_{\ap}(\mathbb{R}^n,\R)$;
\item[(ii)]  for any multi-index $\beta\in\mathbb{Z}_{\ge 0}^n$ such that 
$0\le |\beta|\le m$ one has that $\partial^{\beta}f\in C_{\ap}^\gamma(\mathbb{R}^n,\R)$
where $\gamma:=m-[m]$;
\item[(iii)] $f$ is a limit in $C^m_b(\R^n,\R)$ of a sequence of trigonometric polynomials.
\end{itemize}
\end{theorem}

As a consequence of Theorem \ref{th:almost_periodic_functions_equivalence_m} and Proposition \ref{prop:banach_algebra} 
we obtain

\begin{coro}
The space $C^m_{\ap}(\mathbb{R}^n,\R)$ is a Banach algebra.
\end{coro}

\begin{proof}[Proof of Theorem \ref{th:almost_periodic_functions_equivalence_m}]
We will first prove that $C^m_{\ap}(\R^n,\R)$ is a closed subspace in the Banach space $C^m_b(\R^n,\R)$.
Take a sequence $(f_k)_{k\ge 1}$ in $C^m_{\ap}(\R^n,\R)$ such that
\begin{equation}\label{eq:approximation_f}
f_k\stackrel{C^m_b}{\to} f\quad\text{as}\quad k\to\infty
\end{equation}
for some $f\in C^m_b(\R^n,\R)$. We will prove that $f\in C^m_\ap(\R^n,\R)$. 
To this end, take a sequence $(\alpha_j)_{j\ge 1}$ of real numbers.
It follows from the definition \eqref{eq:C^m_ap} that for any $k\ge 1$ there exists a subsequence 
$(\alpha^k_j)_{j\ge 1}$ of $(\alpha_j)_{j\ge 1}$ and a function ${\tilde f}_k\in C^m_b(\R^n,\R)$ such that
\begin{equation}\label{eq:sequence_c}
(f_k)_{\alpha^k_j}\stackrel{C^m_b}{\to}{\tilde f}_k\quad\text{as}\quad j\to\infty
\end{equation}
with the property that for any $k\ge 1$ the sequence $(\alpha^{k+1}_j)_{j\ge 1}$ is a subsequence of $(\alpha^k_j)_{j\ge 1}$
and $(\alpha^1_j)_{j\ge 1}$ is a subsequence of $(\alpha_j)_{j\ge 1}$.\footnote{Within this proof we will write $(f)_c$ instead of
$f_c$ (cf. \eqref{eq:S_f} for the definition of $f_c$).}
Since the $C^m$-norm is translation invariant (cf. \eqref{eq:C^m-norm} and \eqref{eq:holder_norm})
we conclude from \eqref{eq:sequence_c} that for any given $k, l\ge 1$ such that  $k\ge l$ we have
\begin{equation}\label{eq:isometry}
|{\tilde f}_k-{\tilde f}_l|_m=\lim_{j\to\infty}\big|(f_k)_{\alpha^k_j}-(f_l)_{\alpha^k_j}\big|_m=|f_k-f_l|_m.
\end{equation}
Using that $(f_k)_{k\ge 1}$ is a Cauchy sequence in $C^m_b(\R^n,\R)$ we obtain from \eqref{eq:isometry} that 
there exists ${\tilde f}\in C^m_b(\R^n,\R)$ such that
\begin{equation}\label{eq:sequence_tilde}
{\tilde f}_k\stackrel{C^m_b}{\to}{\tilde f}\quad\text{as}\quad k\to\infty.
\end{equation}
Our last observation is that for any $k, j\ge 1$ we have $|(f_k)_{\alpha^k_j}-(f)_{\alpha^k_j}|_m=|f_k-f|_m$,
which implies that {\em uniformly} in $j\ge 1$,
\begin{equation}\label{eq:sequence_norms}
|(f_k)_{\alpha^k_j}-(f)_{\alpha^k_j}|_m\to 0\quad\text{as}\quad k\to\infty.
\end{equation}
For any given $k, j\ge 1$ we have
\[
|{\tilde f}-(f)_{\alpha^k_j}|_m\le|{\tilde f}-{\tilde f}_k|_m+|{\tilde f}_k-(f_k)_{\alpha^k_j}|_m+
|(f_k)_{\alpha^k_j}-(f)_{\alpha^k_j}|_m.
\]
This, together with \eqref{eq:sequence_c}, \eqref{eq:sequence_tilde}, and
\eqref{eq:sequence_norms}, then implies that for any given integer $p\ge 1$ there exist $k_p, j_p\ge 1$ such that
$|{\tilde f}-(f)_{\alpha^{k_p}_{j_p}}|_m\le 1/p$ with $k_{p+1}>k_p$ and $j_{p+1}>j_p$. 
This shows that the function $\tilde f$ can be approximated in  $C^m_b(\R^n,\R)$ by a sequence of the form 
$(f_{\tilde{\alpha}_j})_{j\ge 1}$ with $\tilde{\alpha}_j\in\R$ where $(\tilde{\alpha}_j)_{j\ge 1}$ is a subsequence of 
$(\alpha_j)_{j\ge 1}$. Hence, the set ${\mathcal S}_f$ is precompact in $C^m_b(\R^n,\R)$, and therefore $f\in C^m_\ap(\R^n,\R)$. 
This completes the proof of that $C^m_\ap(\R^n,\R)$ is a closed subspace in $C^m_b(\R^n,\R)$.

Let us now prove that statements (i) and (ii) are equivalent. Consider the map
\begin{equation}\label{eq:i}
\imath : C^m_b(\R^n,\R)\to\big[C_b^\gamma(\R^n,\R)\big]^{N_0},\quad
f\mapsto\big(\partial^\beta f\big)_{0\le|\beta|\le m},
\end{equation}
where $N_0$ is the number of different multi-indexes $\beta\in\Z_{\ge 0}^n$ with $0\le|\beta|\le m$. 
In view of the definition of the $C^m$-norm, the map \eqref{eq:i} is a linear isomorphism of normed spaces onto its 
image equipped with the norm coming from $[C_b^\gamma(\R^n,\R)\big]^{N_0}$.
In particular, this implies that the image of \eqref{eq:i} is a closed linear subspace in $\big[C_b^\gamma(\R^n,\R)\big]^{N_0}$.
In addition to the map \eqref{eq:i}, for any multi-index $\beta$ with $0\le|\beta|\le m$ consider the map
\begin{equation}\label{eq:i_beta}
\imath_\beta : C^m_b(\R^n,\R)\to C_b^\gamma(\R^n,\R),\quad
f\mapsto\partial^\beta f.
\end{equation}
Now, assume that $f\in C^m_{\ap}(\R^n,\R)$. Then, by the definition \eqref{eq:C^m_ap}, 
the set $\mathcal{S}_f$ is precompact in $C^m_b(\R^n,\R)$. 
Since for any given multi-index $\beta$ the map \eqref{eq:i_beta} is continuous,
the set $\imath_\beta\big(\mathcal{S}_f\big)$ is precompact in $C_b^\gamma(\R^n,\R)$. On the other side, 
one easily sees that
\begin{equation}\label{eq:relation1}
\imath_\beta\big(\mathcal{S}_f\big)=\mathcal{S}_{\partial^\beta f}.
\end{equation}
This implies that $\mathcal{S}_{\partial^\beta f}$ is precompact in $C_b^\gamma(\R^n,\R)$, and hence
$\partial^\beta f\in C_{\ap}^\gamma(\R^n,\R)$ by the definition of the space $C_{\ap}^\gamma(\R^n,\R)$.
This shows that (i) implies (ii).
Now assume that item (ii) holds. Then, for any multi-index $\beta$ such that $0\le|\beta|\le m$ the sets 
$\mathcal{S}_{\partial^\beta f}$ are precompact in $C_b^\gamma(\R^n,\R)$. 
Consider the image $\imath\big(\mathcal{S}_f\big)$ of the set 
$\mathcal{S}_f$ under the map \eqref{eq:i}. It follows from \eqref{eq:relation1} that
\begin{equation}\label{eq:inclusion1}
\imath\big(\mathcal{S}_f\big)\subseteq\bigtimes_{0\le|\beta|\le m}\mathcal{S}_{\partial^\beta f}.
\end{equation}
Since the set on the right hand side of \eqref{eq:inclusion1} is precompact in $\big[C_b^\gamma(\R^n,\R)\big]^{N_0}$ 
(as it is a direct product of precompact sets), we conclude from \eqref{eq:inclusion1} that $\imath\big(\mathcal{S}_f\big)$ is 
precompact in $\big[C_b^\gamma(\R^n,\R)\big]^{N_0}$.
Since \eqref{eq:i} is an isomorphism onto its closed image, we see that $\mathcal{S}_f$ is precompact in
$C^m_b(\R^n,\R)$. Hence, the function $f$ is $C^m$-almost periodic, 
$f\in C^m_{\ap}(\R^n,\R)$. This proves the equivalence of (i) and (ii).

\begin{remark}\label{rem:approximation_by_mollifiers*}
Since for any $m\ge 0$ real the space $C^m_\ap(\R^n,\R)$ is a closed subspace in $C^m_b(\R^n,\R)$ and since for any 
$f\in C^m_\ap(\R^n,\R)$ the map $c\mapsto f\big((\cdot)+c\big)$, $\R^n\to C^m_\ap(\R^n,\R)$, is continuous 
(see the proof of Proposition \ref{prop:approximation}), we conclude that for any multi-index 
$\beta\in\Z_{\ge 0}^n$ the mollifier
\[
\partial^\beta(f*\chi_\varepsilon)=f*(\partial^\beta\chi_\varepsilon)=
\int_{\R^n}f\big((\cdot)-y\big)(\partial^\beta\chi_\varepsilon)(y)\,dy
\]
belongs to $C_\ap^m(\R^n,\R)$ in view of the convergence in $C_\ap^m(\R^n,\R)$ of the Riemann sums in the integral. 
By the equivalence of (i) and (ii) we obtain that
\[
f*\chi_\varepsilon\in C^\infty_\ap(\R^n,\R^n)\equiv\bigcap_{k\ge 1, k\text{-integer}} C^k_\ap(\R^n,\R).
\] 
In particular, we see that the space $C^\infty_\ap(\R^n,\R^n)$ is dense in $C_\ap^m(\R^n,\R)$.
\end{remark}

Finally, we will prove that items (i) and (iii) are equivalent.
The implication $(iii)\Rightarrow(i)$ follows from the closedness of $C^m_\ap(\R^n,\R)$ and 
Corollary \ref{coro:trigonometric_polynomials}. Let us prove that $(i)$ implies $(iii)$.
Take $f\in C^m_\ap(\R^n,\R)$. If $m\ge 0$ is integer, then (iii) follows from \cite[Theorem 2.2]{shubin:survey1978}.
If $m\ge 0$ is not integer we argue as follows: By Remark \ref{rem:approximation_by_mollifiers*} $f$ can be approximated in 
$C^m_\ap(\R^n,\R)$ by a sequence of functions in $C^\infty_\ap(\R^n,\R)$. On the other side, since any element of 
$C^\infty_\ap(\R^n,\R)$ can be approximated in $C^{[m]+1}_\ap(\R^n,\R)$ by a sequence of trigonometric polynomials and since 
the inclusion  $C^{[m]+1}_\ap(\R^n,\R)\subseteq C^m_\ap(\R^n,\R)$ is bounded, we see that any element of 
$C^\infty_\ap(\R^n,\R)$ can be approximated in $C^m_\ap(\R^n,\R)$ by a sequence of trigonometric polynomials.
This shows that the set of trigonometric polynomials is dense in $C^m_\ap(\R^n,\R)$.
Conversely, assume that $f$ is a limit in $C^m_b(\R^n,\R)$ of a sequence of trigonometric polynomials.
Then, $f\in C^m_\ap(\R^n,\R)$ since  $C^m_\ap(\R^n,\R)$ is a closed subspace in $C^m_b(\R^n,\R)$ and
since, by Corollary \ref{coro:trigonometric_polynomials}, the set of trigonometric polynomials lies in $C^m_\ap(\R^n,\R)$.
\end{proof}

Our last result in this Section is the following

\begin{Lemma}\label{lem:division}
Assume that for $f\in C^m_{\ap}(\mathbb{R}^n,\R)$ there exists $\varepsilon>0$ so that 
$|f(x)|>\varepsilon$ for any $x\in\R^n$. Then $1/f\in C^m_{\ap}(\R^n,\R)$ and there exists an open 
neighborhood $U$ of zero in $C^m_{\ap}(\R^n,\R)$ so that for any $g\in U$ one has that $|f+g|>\varepsilon/2$ and the map
\begin{equation}\label{eq:division}
U\to C^m_{\ap}(\R^n,\R),\quad g\mapsto1/(f+g),
\end{equation}
is real analytic.
\end{Lemma}

\begin{remark}
The proof of Lemma \ref{lem:division} carries over without essential changes to $c^m_b(\R^n,\R)$ and $C^m_b(\R^n,\R)$.
Note that Lemma \ref{lem:division} does {\em not} hold without the assumption that $|f(x)|>\varepsilon>0$ for any $x\in\R^n$.
\end{remark}

\begin{proof}[Proof of Lemma \ref{lem:division}]
Assume that $f\in C^m_{\ap}(\mathbb{R}^n,\R)$ and that there exists $\varepsilon>0$ so that 
$|f(x)|>\varepsilon$ for any $x\in\R^n$. We will first prove that $1/f\in C^m_{\ap}(\R^n,\R)$. 
To this end consider the open neighborhood of $f$ in $C^m_b(\R^n,\R)$, 
\[
V(f):=\big\{F\in C^m_b(\R^n,\R)\,\big|\,|F(x)|>\varepsilon/2\big\},
\]
and note the the Banach algebra property of $C^m_b(\R^n,\R)$ and the product rule easily imply that the map
\[
\mu : V(f)\to C^m_b(\R^n,\R),\quad F\mapsto 1/F,
\]
is continuous. One sees from \eqref{eq:S_f} that the set $\mathcal{S}_f$ and its (compact) closure in 
$C^m_b(\R^n,\R)$ are contained in $V(f)$. This shows that the image $\mu\big(\mathcal{S}_f\big)$ of
$\mathcal{S}_f$ under the continuous map $\mu$ above is precompact in $C^m_b(\R^n,\R)$. On the other side, since
\[
\mathcal{S}_{1/f}=\mu\big(\mathcal{S}_f\big),
\]
we conclude that $\mathcal{S}_{1/f}$ is precompact in $C^m_b(\R^n,\R)$. This proves that
$1/f\in C^m_{\ap}(\R^n,\R)$.
In order to prove the second statement of the Lemma consider the open neighborhood of zero 
in $C^m_{\ap}(\R^n,\R)$
\[
U:=\big\{g\in C^m_{\ap}(\R^n,\R)\,\big|\,|g/f|_m<1/2,\,|f+g|>\varepsilon/2\big\}
\] 
and note that
\[
\frac{1}{f+g}=(1/f)\,\frac{1}{1+(g/f)}=(1/f)\sum_{k=0}^\infty(-1)^k(g/f)^k
\]
where the series converges in $C^m_{\ap}(\R^n,\R)$ uniformly in $g\in U$.
Since by the Banach algebra property of $C^m_\ap(\R^n,\R)$ the terms in the series above are polynomial maps 
\[
C^m_{\ap}(\R^n,\R)\to C^m_{\ap}(\R^n,\R),\quad g\mapsto (-1)^k(1/f)(g/f)^k,
\] 
we conclude that the map \eqref{eq:division} is real analytic.
\end{proof}

\medskip  

Recall that
\begin{equation}\label{eq:C^infty_ap}
C^\infty_\ap(\R^n,\R):=\bigcap_{k\ge 1, k\text{-integer}} C^k_\ap(\R^n,\R).
\end{equation}
We will equip $C^\infty_\ap(\R^n,\R)$ with the Fr\'echet topology induced by the norms $|\cdot|_k$ with $k\ge 1$ integer
(see Appendix C where we collect basic facts from the calculus in Fr\'echet spaces). 
We will also need the larger Fr\'echet space 
\begin{equation}\label{eq:C^infty_b}
C^\infty_b(\R^n,\R):=\bigcap_{k\ge 1, k\text{-integer}} C^k_b(\R^n,\R).
\end{equation}

\begin{remark}\label{rem:equivalent_systems_of_norms}
Note that we will obtain the same Fr\'echet space if we replace the space $C^k_\ap(\R^n,\R)$ in formula \eqref{eq:C^infty_ap} with 
$C^{k+\gamma}_\ap(\R^n,\R)$ for some given $0<\gamma<1$ chosen independent of $k\ge 1$. 
The reason is that the system of norms $\big\{|\cdot|_{k+\gamma}\big\}_{k\ge 1}$ induces the same Fr\'echet topology
on $C^\infty_\ap(\R^n,\R)$. Similarly, the replacement of $C^k_b(\R^n,\R)$ in \eqref{eq:C^infty_b} by 
$C^{k+\gamma}_b(\R^n,\R)$ (or $c^{k+\gamma}_b(\R^n,\R)$) leads to the same Fr\'echet space.
\end{remark}

One has the following characterization of $C^\infty_\ap(\R^n,\R)$.

\begin{Lemma}\label{lem:characterization_smooth}
A function $f\in C^\infty_b(\R^n,\R)$ belongs to $f\in C^\infty_\ap(\R^n,\R)$ if and only if the set $\mathcal{S}_f$ is 
precompact in $C^\infty_b(\R^n,\R)$.
\end{Lemma}

\begin{proof}[Proof of Lemma \ref{lem:characterization_smooth}.]
The proof of this Lemma follows directly from the definition \eqref{eq:C^m_ap} and a standard diagonal argument.
\end{proof}

\section{Groups of almost periodic diffeomorphisms}\label{sec:groups_almost_periodic}
For $m\ge 1$ real, denote by $C^m_{\ap}(\mathbb{R}^n,\mathbb{R}^n)$ the space of $C^m$-almost periodic vector fields on 
$\mathbb{R}^n$ whose components lie in the space $C^m_{\ap}(\R^n,\R)$. 
In what follows we often write $C^m_{\ap}$ regardless of whether we consider spaces of functions or tensor fields on $\R^n$. 
Consider the following set of maps
\begin{equation}\label{eq:Diff^m_ap}
{\Diff}^m_{\ap}(\mathbb{R}^n):=\Big\{\varphi(x)=x+f(x)\,\Big|\,
f\in C^m_\ap\,\,\text{s.t.}\,\det\big(I+[d_xf]\big)>\varepsilon\,\,\text{for some}\,\,\varepsilon>0\Big\}
\end{equation}
where $I$ is the $n\times n$ identity matrix, $[d_xf]$ is the Jacobian matrix of $f$ at $x\in\mathbb{R}^n$, and
the inequality in \eqref{eq:Diff^m_ap} is satisfied uniformly in $x\in\R^n$.
In particular, by Hadamard-Levy's theorem (see e.g. \cite[Supplement 2.5D]{abraham2012manifolds}), the set 
$\Diff^m_{\ap}(\R^n)$ consists of orientation preserving diffeomorphisms.
In view of \eqref{eq:Diff^m_ap}, Lemma \ref{lem:division}, and the Banach algebra property of $C^m_\ap(\R^n,\R)$, 
for a given $\varepsilon>0$ the inequality $\det\big(I+[d_xf]\big)>\varepsilon$ $\forall x\in\R^n$ is an open condition on 
$C^m_\ap(\R^n,\R^n)$. This implies that $\Diff^m_{\ap}(\R^n)$ can be identified with an open set in
$C^m_{\ap}(\R^n,\R^n)$ via the map
\[
{\Diff}^m_{\ap}(\mathbb{R}^n)\to C^m_{\ap}(\mathbb{R}^n,\mathbb{R}^n),\quad
\varphi\mapsto f:=\varphi-\id,
\]
which is bijective onto its image. In this way, the set $\Diff^m_{\ap}(\R^n)$ is a $C^\infty$-smooth Banach manifold 
modeled on $C^m_{\ap}(\R^n,\R^n)$. We will also consider the larger set of maps
\begin{equation}\label{eq:Diff^m}
\diff^m_b(\mathbb{R}^n)\coloneqq\Big\{\varphi(x)=x+f(x)\,\Big|\, f\in c^m_b\,\,\text{s.t.}\,
\det\big(I+[d_xf]\big)>\varepsilon\,\,\text{for some}\,\,\varepsilon>0\Big\}
\end{equation}
where the inequality in \eqref{eq:Diff^m} is satisfied uniformly in $x\in\R^n$.
By the same reasoning as above, $\diff^m_b(\mathbb{R}^n)$ consists of orientation preserving diffeomorphisms of 
$\mathbb{R}^n$ and can be identified with an open set in $c^m_b$ via the map $\diff^m_b(\R^n)\to c^m_b$, 
$\varphi\mapsto f:=\varphi-\id$. Hence, $\diff^m_b(\mathbb{R}^n)$ is a $C^{\infty}$-smooth Banach manifold modeled on 
$c^m_b$. We will need the following generalization of \cite[Lemma 2.2]{MisYon}.

\begin{theorem}\label{th:diff^m-topological_group}
For any real $m\ge 1$ the set $\diff^m_b(\mathbb{R}^n)$ is a topological group with respect to the composition of maps.
Moreover, for any integer $r\ge 0$ the maps
\[
\circ: \diff^{m+r}_b(\mathbb{R}^n)\times\diff^m_b(\mathbb{R}^n)\to\diff^m_b(\mathbb{R}^n),\quad
(\varphi, \psi)\mapsto\varphi\circ\psi,
\]
and
\[
\imath: \diff^{m+r}_b(\mathbb{R}^n)\to\diff^m_b(\mathbb{R}^n),\quad
\varphi\mapsto\varphi^{-1},
\]
are $C^r$-smooth.
\end{theorem}

For the proof of this Theorem we refer to Appendix B.

\begin{remark}\label{rem:NB}
Note that if we replace for $m\ge 1$ and $m\notin\Z$ the little H{\"o}lder space $c^m_b$ in definition \eqref{eq:Diff^m} 
by the H{\"o}lder space $C^m_b$, then the composition in Theorem \ref{th:diff^m-topological_group} will {\em not} be 
continuous. In order to see this, take a function $f$ that is $C^2$-smooth everywhere except at $x=0$ where its derivative is 
equal to $\sqrt{|x|}$ in some open neighborhood of zero. Then, one sees from Lemma \ref{lem:charachterization_little_holder} and
Lemma \ref{lem:characterization_little_holder*} in Appendix B that $f\in C^{3/2}_b(\R,\R)$ but $f\notin c^{3/2}_b(\R,\R)$. 
A direct computation involving the definition of the norm in $C^{3/2}_b$ then shows that $|f\circ\tau_c-f|_{3/2}\ge 1$ for any 
$c\notin\Z$ where $\tau_c : x\mapsto x+c$ is translation.
\end{remark}

Since by Theorem \ref{th:almost_periodic_functions_equivalence_m} the space $C^m_{\ap}$ is a closed subspace in 
$c^m_b$, the group $\Diff^m_{\ap}(\R^n)$ is a submanifold in $\diff^m_b(\R^n)$.
First, we prove the following

\begin{Lemma}\label{lem:criteria}
Assume that $\varphi=\id+f\in\diff^m_b(\mathbb{R}^n)$ with $m\ge 1$. 
Then $\varphi\in\Diff^m_{\ap}(\mathbb{R}^n)$ if and only if the set
\[
\mathcal{S}_{\varphi}:=\big\{\varphi_c\in\diff^m_b(\R^n)\,\big|\,\varphi_c(x)\coloneqq x+f(x+c)\big\}_{c\in\mathbb{R}^n}
\]
is precompact in $\diff^m_b(\mathbb{R}^n)$. In this case, the set $\mathcal{S}_{\varphi}$ and its closure in $\diff^m_b(\R^n)$
are contained in $\Diff^m_{\ap}(\R^n)$.
\end{Lemma}

\begin{remark}\label{rem:c}
It is easy to see that $\varphi\in\diff^m_b(\R^n)$ implies that $\varphi_c\in\diff^m_b(\R^n)$ for any $c\in\R^n$.
\end{remark}

\begin{remark}
Note that the notation $(\cdot)_c$ has different meanings depending on whether we apply it to functions in $C^m_b$ or to
diffeomorphisms. This shall not lead to confusion since the type of the objects will always be clear from the context.
\end{remark}

\begin{proof}[Proof of Lemma \ref{lem:criteria}]
Take $\varphi\in\diff^m_b(\R^n)$. Then $\varphi(x)=x+f(x)$ where $f\in c^m_b$ and there exists
$\varepsilon>0$ so that $\det\big(I+[d_xf]\big)>\varepsilon$ for any $x\in\R^n$.
Clearly, $\mathcal{S}_f\subseteq c^m_b$.
Since for any $c\in\R^n$, $[d_xf_c]=[d_{x+c}f]$ where $f_c(x)=f(x+c)$, we conclude that
\[
\det\big(I+[d_xf_c]\big)>\varepsilon\quad\forall x, c\in\R^n.
\]
This implies that for any $g$ in the closure $\overline{\mathcal{S}_f}$ of the set 
$\mathcal{S}_f$ in $C^m_b(\R^n)$,
\[
\det\big(I+[d_xg]\big)>\varepsilon/2\quad\forall x\in\R^n.
\]
Since $c^m_b$ is a closed subspace in $C^m_b$ we conclude that $\id+\overline{\mathcal{S}_f}\subseteq\diff^m_b(\R^n)$.
By the definition of the Banach manifold structure on $\diff^m_b(\R^n)$,
\begin{equation}\label{eq:shifted_closure}
\overline{\mathcal{S}_\varphi}=\id+\overline{\mathcal{S}_f}\subseteq\diff^m_b(\R^n),
\end{equation}
where $\overline{\mathcal{S}_\varphi}$ is the closure of the set $\mathcal{S}_\varphi$ in $\diff^m_b(\R^n)$.
The first statement of Lemma \ref{lem:criteria} then immediately follows from \eqref{eq:shifted_closure}.
In order to prove the last statement of the Lemma note that the definition of $\mathcal{S}_f$ and 
the fact that $C^m_{\ap}$ is a closed subset in $C^m_b$ imply that $f\in C^m_{\ap}$ if and only if 
$\overline{\mathcal{S}_f}\subseteq C^m_{\ap}(\R^n)$. Combining this with \eqref{eq:shifted_closure} we conclude that 
$\overline{\mathcal{S}_\varphi}\subseteq\Diff^m_{\ap}(\R^n)$.
\end{proof}

The main result of this section is
\begin{theorem}\label{th:group_regularity}
For any real $m\ge 1$ the set ${\Diff}^m_{\ap}(\mathbb{R}^n)$ is a topological subgroup of 
$\diff^m_b(\mathbb{R}^n)$. Moreover, for any integer $r\ge 0$ the maps
\[
\circ: {\Diff}^{m+r}_{\ap}(\mathbb{R}^n)\times{\Diff}^m_{\ap}(\mathbb{R}^n)\to{\Diff}^m_{\ap}(\mathbb{R}^n),\quad
(\varphi, \psi)\mapsto\varphi\circ\psi,
\]
and
\[
\imath: {\Diff}^{m+r}_{\ap}(\mathbb{R}^n)\to{\Diff}^m_{\ap}(\mathbb{R}^n),\quad
\varphi\mapsto\varphi^{-1},
\]
are $C^r$-smooth.
\end{theorem}

\begin{proof}[Proof of Theorem \ref{th:group_regularity}]
We will first prove that $\Diff^m_{\ap}(\mathbb{R}^n)$ is closed under composition and inversion of diffeomorphisms.
Take $\varphi,\psi\in\Diff^m_{\ap}(\R^n)$ and assume that  $\varphi(x)=x+f(x)$ and $\psi(x)=x+g(x)$ for some
$f,g\in C^m_{\ap}$. Then, for any $c, x\in\R^n$ we have
\begin{equation}\label{eq:shift1}
\begin{split}
\big(\varphi_c\circ\psi_c\big)(x)&=x+g(x+c)+f\bigl(x+c+g(x+c)\bigr)\\
&=\Bigl(x+g(x)+f\bigl(x+g(x)\bigr)\Bigr)_c\\
&=\big(\varphi\circ\psi\big)_c(x).
\end{split}
\end{equation}
This shows that 
\begin{equation}\label{eq:shift2}
\mathcal{S}_{\varphi\circ\psi}=\big\{(\varphi\circ\psi)_c\big\}_{c\in\R^n}=
\big\{\varphi_c\circ\psi_c\big\}_{c\in\R^n}\subseteq
\mathcal{S}_\varphi\circ\mathcal{S}_\psi.
\end{equation}
On the other side, 
\begin{equation}\label{eq:shift3}
\mathcal{S}_\varphi\circ\mathcal{S}_\psi=\circ\big(\mathcal{S}_\varphi\times\mathcal{S}_\psi\big)\subseteq\diff^m_b(\R^n)
\end{equation}
where $\circ\big(\mathcal{S}_\varphi\times\mathcal{S}_\psi\big)$ denotes the image of the set 
$\mathcal{S}_\varphi\times\mathcal{S}_\psi$ under the composition map
\begin{equation}\label{eq:composition2}
\circ: \diff^m_b(\mathbb{R}^n)\times\diff^m_b(\mathbb{R}^n)\to\diff^m_b(\mathbb{R}^n),\quad
(\varphi, \psi)\mapsto\varphi\circ\psi.
\end{equation}
Since $\varphi,\psi\in\Diff^m_{\ap}(\R^n)$, we see from Lemma \ref{lem:criteria} that
$\mathcal{S}_\varphi$ and $\mathcal{S}_\psi$ are precompact sets in $\diff^m_b(\R^n)$.
Hence, the direct product $\mathcal{S}_\varphi\times\mathcal{S}_\psi$ is a precompact set in 
$\diff^m_b(\mathbb{R}^n)\times\diff^m_b(\mathbb{R}^n)$. 
This, together with the continuity of the composition map \eqref{eq:composition2} and relation \eqref{eq:shift3}, then
implies that $\mathcal{S}_\varphi\circ\mathcal{S}_\psi$ is precompact in $\diff^m_b(\R^n)$.
Equation \eqref{eq:shift2} then shows that $\mathcal{S}_{\varphi\circ\psi}$ is a precompact set in $\diff^m_b(\R^n)$. 
By Lemma \ref{lem:criteria},  we then conclude that
\[
\varphi\circ\psi\in\Diff^m_{\ap}(\R^n).
\]
Now, take $\varphi\in\Diff^m_{\ap}(\R^n)$ and assume that $\varphi(x)=x+f(x)$ where $f\in C^m_\ap$.
It follows from Theorem \ref{th:diff^m-topological_group} that $\varphi^{-1}\in\diff^m_b(\R^n)$ and hence
$\psi:=\varphi^{-1}(x)=x+g(x)$ for some $g\in c^m_b$. Since $\varphi\circ\psi=\id$, we see from \eqref{eq:shift1} 
that $\psi_c\circ\varphi_c=\id$ and $\psi_c,\varphi_c\in\diff^m_b(\R^n)$ (Remark \ref{rem:c}) for any $c\in\R^n$. 
In particular, $\psi_c=(\varphi_c)^{-1}$ and hence
\begin{equation*}
\imath(\mathcal{S}_{\varphi})=\mathcal{S}_{\varphi^{-1}},
\end{equation*} 
where $\imath: \diff^m_b(\mathbb{R}^n)\to\diff^m_b(\mathbb{R}^n)$ is the inversion map.
Since $\varphi\in\Diff^m_{\ap}(\R^n)$, we conclude from Lemma \ref{lem:criteria} that the set 
$\mathcal{S}_{\varphi}$ is precompact in $\diff^m_b(\mathbb{R}^n)$. 
The continuity of the inversion map $\imath: \diff^m_b(\R^n)\to\diff^m_b(\R^n)$ from 
Theorem \ref{th:diff^m-topological_group} then implies that
$\mathcal{S}_{\varphi^{-1}}=\imath(\mathcal{S}_{\varphi})$ is precompact in $\diff^m_b(\mathbb{R}^n)$. 
Then, by Lemma \ref{lem:criteria},
\[
\varphi^{-1}\in\Diff^m_{\ap}(\mathbb{R}^n).
\]
In this way, we have proved that $\Diff^m_{\ap}(\R^n)$ is a group with respect to the composition of diffeomorphisms.
The continuity and the $C^r$-regularity of the composition and the inversion map on $\Diff^m_\ap(\R^n)$ 
then follow from Theorem \ref{th:diff^m-topological_group}. 
\end{proof}

\begin{remark}\label{rem:composition}
As a consequence, one sees that for any real $m\ge 1$ and for any integer $r\ge 0$ the map
\[
\circ : C^{m+r}_{\ap}(\R^n,\R)\times\Diff^m_{\ap}(\R^n)\to C^m_{\ap}(\R^n,\R),\quad
(f,\varphi)\mapsto f\circ\varphi,
\]
is well-defined and $C^r$-smooth. The same result also holds for the group $\diff^m_b(\R^n)$
and the space $c^m_b(\R^n,\R)$.
\end{remark}

\medskip

Finally, consider the topological group
\begin{equation}\label{eq:Diff^infty_ap}
\Diff^\infty_{\ap}(\R^n):=\bigcap_{k\ge 1, k\text{-integer}}\Diff^k_\ap(\R^n)
\end{equation}
supplied with the Fr\'echet topology coming from $C^\infty_\ap(\R^n,\R)$ as well as the larger Fr\'echet group
\begin{equation}\label{eq:Diff^infty_b}
\Diff^\infty_b(\R^n):=\bigcap_{k\ge 1, k\text{-integer}}\diff^k_b(\R^n).
\end{equation}
It follows from Theorem \ref{th:group_regularity} and Theorem \ref{th:diff^m-topological_group} that the composition and 
the inversion map in these groups are $C^\infty_F$-smooth. 
In this sense, the groups $\Diff^\infty_{\ap}(\R^n)$ and $\Diff^\infty_b(\R^n)$ are {\em Fr\'echet-Lie groups}.

\medskip

\noindent{\em The Lie group exponential map on $\Diff^\infty_\ap(\R^n)$:}
Moreover, it follows from Theorem \ref{th:ode} in Appendix A that for any $u\in C\big([0,T],C^\infty_{\ap}\big)$ there exists 
a unique solution $\varphi\in C^1_F\big([0,T],\Diff^\infty_\ap(\mathbb{R}^n)\big)$ of 
the equation
\begin{equation}\label{eq:ode_frechet}
\left\{
\begin{array}{l}
\dot{\varphi}=u\circ\varphi,\\
\varphi|_{t=0}=\id.
\end{array}
\right. 
\end{equation}
By taking $u\in C^\infty_\ap$ independent of $t$ we get a global solution 
\[
\varphi\in C^1_F\big([0,\infty),\Diff^\infty_\ap(\mathbb{R}^n)\big)
\]
which, in view of Remark \ref{rem:ode_dependence_on_u} in Appendix A, depends $C^1$-smoothly on the
choice of the vector field $u\in C^\infty_\ap$ in the sense that the map
\[
[0,T]\times C^\infty_\ap\to\Diff^\infty_\ap(\R^n),\quad(t,u)\mapsto\varphi(t;u),
\]
is $C^1_F$-smooth. This allows us to define the {\em Lie group exponential map},
\begin{equation}\label{eq:exp_lg}
\Exp_{\rm LG} : C^\infty_\ap\to\Diff^\infty_\ap(\R^n),\quad u\mapsto\varphi(1,u),
\end{equation}
which is $C^1_F$-smooth. One easily sees that for any $t\ge 0$ and for any $u\in C^\infty_\ap$,
\[
\Exp_{\rm LG}(t u)=\varphi(1,t u)=\varphi(t,u)
\]
which implies that
\begin{equation*}
d_0\Exp_{\rm LG}=\id_{C^\infty_\ap}.
\end{equation*}
Now, take a non-zero constant vector $c\in\R^n$ and let $u_0=c\ne 0$. Then $\varphi(t,u_0)=\id+t u_0$ is the solution 
of \eqref{eq:ode_frechet} with $u=u_0$. Take $\delta u\in C^\infty_\ap$ and for any given $t\ge 0$ consider the directional 
derivative in $\Diff^\infty_\ap(\R^n)$,
\[
w(t)\equiv\big(D_2\varphi\big)(t,u_0)\cdot\delta u:=\frac{d}{ds}\Big|_{s=0}\varphi\big(t,u_0+s\,\delta u\big)\in C^\infty_\ap.
\]
It follows directly from \eqref{eq:ode_frechet} that $w\in C^1_F\big([0,\infty),C^\infty_\ap\big)$ satisfies the equation
\[
{\dot w}(t,x)=(\delta u)(x+c t),\quad w|_{t=0}=0,
\]
for any given $x\in\R^n$. This implies that $w(t)=\int_0^t(\delta u)(x+c s)\,ds$ and hence
\begin{equation}\label{eq:d_cExp}
\big(d_c\Exp_{\rm LG}\big)(\delta u)=\big(D_2\varphi\big)(1,u_0)\cdot\delta u=\int_0^1(\delta u)(x+c s)\,ds.
\end{equation}
Now, take 
\begin{equation}\label{eq:delta_u}
(\delta u)(x):=\sum_{k\ne 0}{\hat\nu}_k e^{i \left(\Lambda_k,x\right)_{\R^n}},\quad \Lambda_k:=2\pi k c/|c|^2,
\end{equation}
so that ${\hat\nu}_{-k}=\overline{{\hat\nu}_k}$ and the sequence $({\hat\nu}_k)_{k\in\Z\setminus\{0\}}$ in $\C^n$ is 
chosen so that the series in \eqref{eq:delta_u} converges in $C^\infty_\ap$. For example, one can choose
${\hat\nu}_k=O\big(e^{-|k|}\big)$. Then, in view of \eqref{eq:d_cExp}, we obtain that for $\delta u\in C^\infty_\ap$
given by \eqref{eq:delta_u} we have
\[
\big(d_c\Exp_{\rm LG}\big)(\delta u) = 0.
\]
In this way we proved

\begin{theorem}\label{th:lie_group_exponential_map}
The group $\Diff^\infty_\ap(\R^n)$ has a well-defined Lie group exponential map
$\Exp_{\rm LG} : C^\infty_\ap\to\Diff^\infty_\ap(\R^n)$ such that $d_0\Exp_{\rm LG}=\id_{C^\infty_\ap}$ and 
for any given non-zero vector $c\in\R^n$ there exists a non-zero $\delta u\in C^\infty_\ap$ such that 
$\big(d_c\Exp_{\rm LG}\big)(\delta u) = 0$. The kernel of $d_c\Exp_{\rm LG}$ is infinite dimensional.
\end{theorem}

Note that in contrast to the Lie group exponential map on the group of $C^\infty$-smooth diffeomorphisms of the circle
(see e.g. \cite[Counterexample 5.5.2]{Hamilton}), we have {\em no} restrictions on the choice of the non-zero constant 
vector $c\in\R^n$. In this way we see that for any given $c\in\R^n\setminus\{0\}$ the one-parameter family 
$\big\{\Exp_{\rm LG}(t\,c)\big\}_{t>0}$ of diffeomorphisms in $\Diff^\infty_\ap(\R^n)$ consists of critical points of \eqref{eq:exp_lg}.
As a direct consequence of Theorem \ref{th:lie_group_exponential_map} and Remark \ref{rem:non-degeneracy} in Appendix C 
we obtain

\begin{coro}\label{coro:lie_group_exponential_map}
There is {\em no} open neighborhood $V$ of zero in $C^\infty_\ap$ such that the restriction
$\Exp_{\rm LG}\big|_V : V\to\Diff^\infty_\ap(\R^n)$ is a $C^1_F$-diffeomorphism onto its image.
\end{coro}

\medskip

\noindent{\em The Burgers' Riemannian exponential map on $\Diff^\infty_\ap(\R)$:}
A Riemannian counterpart of Corollary \ref{coro:lie_group_exponential_map} was discussed in Remark \ref{rem:burgers}.
The geodesic equation corresponding to the scalar product \eqref{eq:scalar_product_introduction} with $\alpha=0$ and $n=1$
is the inviscid Burgers equation
\begin{equation}\label{eq:burgers}
u_t+3u_x u=0,\quad u|_{t=0}=u_0\in C^\infty_\ap(\R,\R).
\end{equation}
By Remark \ref{rem:burgers}, there exists an open neighborhood $U$ of zero in $C^\infty_\ap(\R,\R)$ and $T>0$ such that 
\eqref{eq:burgers} has a unique solution $u\in C^1_F\big([0,T],C^\infty_\ap(\R,\R)\big)$ that depends $C^1_F$-smoothly 
on the initial data $u_0\in U$. Then, the corresponding geodesic $\varphi\in C^1_F\big([0,T],\Diff^\infty_\ap(\R)\big)$ is a solution
(see Theorem \ref{th:ode} and Remark \ref{rem:ode_dependence_on_u} in Appendix A) of the equation 
\begin{equation}\label{eq:ode'}
\left\{
\begin{array}{l}
\dot{\varphi}=u\circ\varphi,\\
\varphi|_{t=0}=\id.
\end{array}
\right. 
\end{equation}
Note that the initial data $u_0=c$ where $c\ne 0$ is a real constant gives the solution $u(t)=c$ of \eqref{eq:burgers} which, 
by \eqref{eq:ode'}, gives the geodesic $\Exp(ct)=\id+ct\in\Diff^\infty_\ap(\R)$, $t\in\R$.
Now, take a variation $\delta u_0\in C^\infty_\ap(\R,\R)$ of the initial data $u_0$ and denote by 
$u(s)\in C^1_F([0,T],C^\infty_\ap(\R,\R))$ the solution of the Burgers equation with initial data $u(s)|_{t=0}=u_0+s\delta u_0$ 
for small values of $s\in\R$. Then, we obtain from \eqref{eq:burgers} that
\begin{equation*}
(\delta u)_t+3c\,(\delta u)_x=0,\quad(\delta u)|_{t=0}=\delta u_0,
\end{equation*}
where $\delta u:=\frac{d}{ds}\big|_{s=0} u(s)$.
This equation has a unique solution
\begin{equation}\label{eq:variation_burgers}
(\delta u)(t,x)=(\delta u_0)(x-3c t).
\end{equation}
By substituting $u(s)$ for $u$ in equation \eqref{eq:ode'} and then taking the derivative $\frac{d}{ds}\big|_{s=0}$ we 
obtain the following equation for the variation $\delta\varphi:=\frac{d}{ds}\big|_{s=0}\varphi(s)$,
\[
(\delta\varphi)_t=(\delta u)(t, x+ct),\quad(\delta\varphi)(t,x)=\int_0^t(\delta u)(s, x+cs)\,ds
\]
Combining this with \eqref{eq:variation_burgers}, we finally obtain
\begin{equation}\label{eq:d_cExp_burgers}
(d_c\Exp)(\delta u_0)=(\delta\varphi)(1,x)=\int_0^1(\delta u_0)(x-2cs)\,ds=
\frac{1}{2c}\int_{x-2c}^x(\delta u_0)(y)\,dy
\end{equation}
which is essentially the same formula as the one in \eqref{eq:d_cExp}.

Now, assume that 
\begin{equation}\label{eq:conjugate_point}
(d_c\Exp)(\delta u_0)=0.
\end{equation} 
This together with \eqref{eq:d_cExp_burgers} implies that $\int_{x-2c}^x(\delta u_0)(y)\,dy=0$ for any $x\in\R$.
By differentiating this identity with respect to $x$ we obtain that
\[
(\delta u_0)(x-2c)=(\delta u_0)(x)
\]
for any $x\in\R$. These considerations easily imply the following

\begin{Lemma}\label{lem:quantization_condition}
$(d_c\Exp)(\delta u_0)=0$ if and only if $\delta u_0\in C^\infty_\ap(\R,\R)$ is a periodic function with
period $2c$ and vanishing mean-value $\frac{1}{2c}\int_0^{2c}(\delta u_0)(x)\,dx=0$.
\end{Lemma}

In the almost periodic case we can construct a solution $0\ne\delta u_0\in C^\infty_\ap(\R,\R)$ of \eqref{eq:conjugate_point}
for any value of the real constant $c\ne 0$. For example, we can take
\[
(\delta u_0)(x):=\sum_{k\ne 0}{\hat\nu_k} e^{2\pi i k x/2c}
\]
where ${\hat\nu}_{-k}=\overline{{\hat\nu}_k}$ and $\hat\nu_k:=O\big(e^{-|k|}\big)$, $k\ne 0$.
In this way, we have

\begin{Prop}\label{eq:burgers_special_geodesics}
For any given $c\in\R$, $c\ne 0$, the geodesic $\Exp(ct)=\id+ct\in\Diff^\infty_\ap(\R)$, $t\in\R$, $t\ne 0$, corresponding to the
Burgers' Riemannian exponential map consists entirely of points conjugate to the identity.
\end{Prop}

\begin{remark}\label{rem:quantization_condition}
Note that if one requires $\delta u_0\in C^\infty_\ap(\R,\R)$ to be 1-periodic then, by Lemma \ref{lem:quantization_condition},
the equation $(d_c\Exp)(\delta u_0)=0$ will have a non-trivial solution if and only if the periods $2c$ and $1$ are rationally dependent or, equivalently,
when $c\in\mathbb{Q}$. This is the reason Proposition \ref{eq:burgers_special_geodesics} does {\em not} hold for the group of
diffeomorphisms of the circle $\Diff^\infty(\mathbb{T})$. Moreover, the Proposition does {\em not} hold for any group of diffeomorphisms of $\R$ with 
infinitesimal generators vanishing at infinity. The inclusion of the constant vector fields as infinitesimal generators of the group of
diffeomorphisms is not sufficient. In fact, a necessary condition for a group of diffeomorphisms of $\R$ to satisfy 
Proposition \ref{eq:burgers_special_geodesics} is that it must have a $c$-periodic infinitesimal generator for any $c\in\R$, $c\ne 0$. 
The analog of this condition on $\R^n$ is that there exists a direction  $c\in\R^n$, $|c|=1$, such that the group of diffeomorphisms of $\R^n$ 
has an $\alpha c$-periodic infinitesimal generator for any real $\alpha>0$.
\end{remark}

\medskip

In the next Section we show that the Riemannian exponential map of an appropriately chosen weak Riemannian metric
on $\Diff^\infty_\ap(\R^n)$ is a $C^1_F$-diffeomorphism onto its image when restricted to a sufficiently small 
open neighborhood of zero in $C^\infty_\ap(\R^n)$.

\section{A Riemannian exponential map}\label{sec:exponential_map}
In this section we study the properties of the Riemannian exponential map on
$\Diff^{\infty}_{\ap}(\mathbb{R}^n)$ corresponding to the right-invariant (weak) Riemannian metric 
on $\Diff^{\infty}_{\ap}(\mathbb{R}^n)$ given at the tangent space at identity 
$T_{\id}\big(\Diff^{\infty}_{\ap}(\R^n)\big)\equiv C^\infty_\ap(\R^n,\R^n)$
by the scalar product $\langle\cdot,\cdot\rangle_\alpha : C^\infty_\ap\times C^\infty_\ap\to\R$,
\begin{equation}\label{eq:scalar_product}
\langle u,v\rangle_\alpha:=\lim_{T\to\infty}\frac{1}{(2T)^n}\int_{[-T,T]^n}\big((I-\alpha^2\Delta)u,v\big)_{\R^n}\,dx,
\end{equation}
where $(\cdot,\cdot)_{\R^n}$ is the Euclidean scalar product in $\R^n$, $\Delta$ is the Euclidean Laplacian
$\Delta=\frac{\partial^2}{\partial x_1^2}+...+\frac{\partial^2}{\partial x_n^2}$ acting component-wise on vector fields on $\R^n$, 
and $\alpha>0$ is a given {\em positive} constant.
Since by Proposition \ref{prop:banach_algebra} the integrand in \eqref{eq:scalar_product}
is an almost periodic function, the limit in \eqref{eq:scalar_product} exists by Theorem \ref{th:almost_periodic_functions_equivalence}. 
By using the product formula 
\begin{equation}\label{eq:product_formula}
\big(\Delta u,v\big)_{\R^n}=-\big(\nabla u,\nabla v\big)_{\R^n}+\Div\big((\nabla u)^Tv\big)
\end{equation}
where $\nabla u:=[du]\equiv\big(\frac{\partial u_k}{\partial x_l}\big)_{1\le k,l\le n}$ is the Jacobian $n\times n$ matrix,
$(\cdot)^T$ is the transpose of a matrix, and 
$\big(\nabla u,\nabla v\big)_{\R^n}:=\sum_{k=1}^n\big(\nabla u_k,\nabla v_k\big)_{\R^n}$, 
we obtain from \eqref{eq:scalar_product}, \eqref{eq:product_formula}, and from the divergence theorem that
\begin{eqnarray}
\langle u,v\rangle_\alpha&=&\lim_{T\to\infty}\frac{1}{(2T)^n}\int_{[-T,T]^n}
\Big((u,v)_{\R^n}+\alpha^2(\nabla u,\nabla v)_{\R^n}\Big)\,dx\nonumber\\
&-&\alpha^2\lim_{T\to\infty}\frac{1}{(2T)^n}\int_{\partial([-T,T]^n)}\Big(v,\frac{\partial u}{\partial\nu}\Big)_{\R^n}\,d\sigma\nonumber\\
&=&\lim_{T\to\infty}\frac{1}{(2T)^n}\int_{[-T,T]^n}\Big( (u,v)_{\R^n}+\alpha^2(\nabla u,\nabla v)_{\R^n}\Big)\,dx,
\label{eq:scalar_product'}
\end{eqnarray}
where $\nu$ is the outward unit normal to the boundary of the cube $[-T,T]^n$, $d\sigma$ is the volume form of the boundary, 
and where we used that the integrand $\Big(v,\frac{\partial u}{\partial\nu}\Big)_{\R^n}$ is uniformly bounded on $\R^n$, to conclude that 
the integral over $\partial([-T,T]^n)$ is of order $O(T^{n-1})$, and hence the corresponding limit vanishes. 
Formula \eqref{eq:scalar_product'} shows that the scalar product \eqref{eq:scalar_product} extends to a positive definite bounded bilinear map 
$C^1_\ap\times C^1_\ap\to\R$ and hence defines a weak scalar product on the tangent space
$T_{\id}\big(\Diff^{\infty}_{\ap}(\R^n)\big)\equiv C^\infty_\ap$.
Recall that a scalar product on a topological vector space $X$ being {\em weak} means that the topology induced on 
the vector space $X$ by this scalar product is weaker than the original topology of $X$.
The scalar product \eqref{eq:scalar_product} defines a (weak) Riemannian metric on the group 
${\Diff}^{\infty}_{\ap}(\R^n)$ by right translations: for any $\varphi\in\Diff^{\infty}_{\ap}(\R^n)$ and for any 
$\xi,\eta\in T_{\varphi}\big(\Diff^{\infty}_{\ap}(\R^n)\big)\equiv C^\infty_\ap$, define
\begin{equation}\label{eq:riemannian_metric}
\nu_\alpha(\xi,\eta):=\big\langle R_{\varphi^{-1}}\xi,R_{\varphi^{-1}}\eta\big\rangle_\alpha
=\lim_{T\to\infty}\frac{1}{(2T)^n}\int\limits_{[-T,T]^n}
\Big(\big(A_\alpha\circ R_{\varphi^{-1}}\big)\xi,R_{\varphi^{-1}}\eta\Big)_{\R^n}\,dy
\end{equation}
where $R_{\varphi^{-1}\xi}:=\xi\circ\varphi^{-1}$, $R_{\varphi^{-1}}\eta:=\eta\circ\varphi^{-1}$,
and $A_\alpha:= I-\alpha^2\Delta$.  By changing the variables in the integral above we obtain
\begin{eqnarray}
\nu_\alpha(\xi,\eta)&=&\lim_{T\to\infty}\frac{1}{(2T)^n}\!\!\!\!\!\int\limits_{\varphi^{-1}([-T,T]^n)}\!\!\!\!
\Big(\big(R_\varphi\circ A_\alpha\circ R_{\varphi^{-1}}\big)\xi,\eta\Big)_{\R^n}\!\!\!
\det[d_x\varphi]\,dx\nonumber\\
&=&\lim_{T\to\infty}\frac{1}{(2T)^n}\int\limits_{[-T,T]^n}
\Big(\big(R_\varphi\circ A_\alpha\circ R_{\varphi^{-1}}\big)\xi,\eta\Big)_{\R^n}\!\!\!
\det[d_x\varphi]\,dx.\label{eq:riemannian_metric'}
\end{eqnarray}
The last equality readily follows from the inclusion of sets
\[
\big[-(T-B),T-B\big]^n\subseteq\varphi^{-1}\big([-T,T]^n\big)\subseteq\big[-(T+B),T+B\big]^n
\]
which holds for any $T>B$ where $B\ge |g|_\infty$ is a constant independent of $T$ and $
\varphi^{-1}=\id+g$ with $g\in C^\infty_\ap$.
The least action principle applied to \eqref{eq:riemannian_metric'}, together with the arguments used in the periodic case
\cite[Appendix C]{KLT}, show that the Euler-Lagrange equation corresponding to the energy functional
\begin{equation}\label{eq:energy_functional}
\mathcal{E}(\varphi):=\frac{1}{2}\int_0^1\nu_\alpha(\dot\varphi,\dot\varphi)\,dt,\quad\quad
\varphi\in C^2_F\big([0,1],\Diff^\infty_\ap(\R^n)\big),
\end{equation}
when written in Eulerian coordinates $u={\dot\varphi}\circ\varphi^{-1}$, takes the form
\begin{equation}\label{eq:ch}
\left\{
\begin{array}{l}
\m_t+(\nabla\m)\cdot u+\big((\Div u)\cdot I+(\nabla u)^T\big)\cdot\m=0,\quad\m := (1-\alpha^2\Delta)u\\
u|_{t=0}=u_0
\end{array}
\right. .
\end{equation}

\begin{remark}
If $n=1$ equation \eqref{eq:ch} is the Camassa-Holm equation (see \cite{CH}).
\end{remark}
In Lagrangian coordinates $(\varphi,v)\in T\big(\Diff^{\infty}_{\ap}(\R^n)\big)\equiv \Diff^{\infty}_{\ap}(\R^n)\times C^\infty_\ap$ 
where $v:=u\circ \varphi$ and
\begin{equation}
\left\{
\begin{array}{l}
\dot{\varphi}=u\circ\varphi\\
\varphi|_{t=0}=\id
\end{array}
\right. 
\end{equation}
one has
\begin{equation}\label{eq:ch'}
(\dot{\varphi},\dot{v})=\big(v,R_{\varphi}\circ A_\alpha^{-1}\circ B_\alpha\circ R_{\varphi^{-1}}(v)\big)\\
\end{equation}
where $(\varphi, v)|_{t=0}=\big(\id, u_0\big)$,
\[
B_\alpha(u):= A_\alpha\big[(\nabla u)\cdot u\big]-\nabla[A_\alpha u]\cdot u-
\big(\Div u\cdot I+(\nabla u)^{T}\big) A_\alpha u,
\] 
and $A^{-1}_{\alpha} : C^\infty_{\ap}\to C^\infty_{\ap}$ is the inverse of of the linear isomorphism of Fr\'echet
spaces $A_\alpha : C^\infty_\ap\to C^\infty_\ap$. Note that for any {\em non-integer} $m\ge 0$ the differential
expression $A_\alpha$ defines a linear isomorphism of Banach spaces
\begin{equation}\label{eq:A_alpha}
A_\alpha : C^{m+2}_\ap\to C^m_\ap
\end{equation}
with a (bounded) inverse $A^{-1}_{\alpha} : C^m_{\ap}\to C^{m+2}_{\ap}$
(see Corollary \ref{coro:A_on_ap_isomorphism} in Appendix A).

\medskip

Denote by $B_{C^m_{\ap}}(\rho)$ the open ball in $C^m_{\ap}$ of radius $\rho>0$ centered at zero. 
One has the following

\begin{theorem}\label{th:ch_m}
Let $m\ge 3$, $m\notin\Z$, and $\rho>0$. Then there exists $T\equiv T_{m,\rho}>0$ such that for
any $u_0\in B_{C^m_{\ap}}(\rho)$, there exists a unique solution 
$u\in C\big([0,T],C^m_{\ap}\big)\cap C^1\big([0,T],C^{m-1}_{\ap}\big)$ of (\ref{eq:ch}) that depends
continuously on the initial data $u_0$ in the sense that the data-to-solution map 
\[
B_{C^m_{\ap}}(\rho)\to C\big([0,T], C^m_{\ap}\big)\cap C^1\big([0,T], C^{m-1}_{\ap}\big),\quad u_0\mapsto u,
\]
is continuous. Moreover, the right hand side of \eqref{eq:ch'} is a $C^1$-smooth vector field on 
$\Diff^m_{\ap}(\R^n)\times C^m_\ap$ and, in particular, there exists a unique solution
\[
(\varphi,v)\in C^1\big([0,T],{\Diff}^m_{\ap}(\R^n)\times C^m_{\ap}\big)
\]
of (\ref{eq:ch'}) such that the data-to-solution map
\begin{equation*}
[0,T]\times B_{C^m_{\ap}}(\rho)\to{\Diff}^m_{\ap}(\R^n)\times C^m_{\ap},\quad(t,u_0)\mapsto\big(\varphi(t,u_0),v(t,u_0)\big),
\end{equation*}
is $C^1$-smooth.
\end{theorem}
The proof of this theorem is based on the fact that the right-hand side of \eqref{eq:ch'} is a $C^1$-smooth vector field on
$\Diff^m_{\ap}(\R^n)\times C^m_\ap$. Details of the proof are given in Appendix A.

\begin{remark}\label{rem:n=1}
If $n=1$ then the condition $m\notin\Z$ in Theorem \ref{th:ch_m} is {\em not} needed.
\end{remark}

\begin{remark}\label{rem:higher_smoothness}
In fact, we indicate in Appendix A that the right-hand side of \eqref{eq:ch'} is an analytic vector field on 
$\Diff^m_{\ap}(\R^n)\times C^m_\ap$. The higher regularity of the right-hand side of \eqref{eq:ch'} is used in the proof of
Proposition \ref{prop:onto} and Proposition \ref{prop:exact_regularity} below.
\end{remark}

\begin{remark}\label{rem:independence_of_m}
Proposition \ref{prop:ch_extension} below implies that the time of existence $T>0$ in Theorem \ref{th:ch_m} can
be chosen independent of the regularity exponent $m\ge 3$, $m\notin\Z$.
This Proposition is a variant of the no-loss-no-gain result of Ebin and  Marsden in \cite[Theorem 12.1]{EM}.
\end{remark}

\medskip

Note that if for some $T>0$ the curve $(\varphi,v)\in C^1\big([0,T],{\Diff}^m_{\ap}(\R^n)\times C^m_{\ap}\big)$ is a solution of
\eqref{eq:ch'}, then for any $0<\mu\le T/2$ the curve 
\begin{equation}\label{eq:symmetry}
t\mapsto\big(\varphi(\mu t),\mu\,v(\mu t)\big),\quad [0,2]\to{\Diff}^m_{\ap}(\R^n)\times C^m_{\ap}, 
\end{equation}
is a $C^1$-smooth solution of \eqref{eq:ch'} with initial data $\big(\id,\mu\,v(0)\big)$. Combining this with the second statement of 
Theorem \ref{th:ch_m}, one concludes that for any $m\ge 3$, $m\notin\Z$, there exists an open neighborhood $V_m$ of zero in 
$C^m_\ap$ such that for any $u_0\in V_m$ there exists a unique solution 
\begin{equation}\label{eq:initial_data_m}
(\varphi,v)\in C^1\big([0,2],{\Diff}^m_{\ap}(\R^n)\times C^m_{\ap}\big),\quad(\varphi,v)|_{t=0}=(\id,u_0),
\end{equation}
of (\ref{eq:ch'}) such that the solution map
\begin{equation}\label{eq:ch_solution_[0,2]}
[0,2]\times V_m\to{\Diff}^m_{\ap}(\R^n)\times C^m_{\ap},\quad(t,u_0)\mapsto\big(\varphi(t,u_0),v(t,u_0)\big),
\end{equation}
is $C^1$-smooth. This allows us for any $m\ge 3$, $m\notin\Z$, to define the {\em exponential map}
\begin{equation}\label{eq:Exp_m}
\Exp_m : V_m\to{\Diff}^m_{\ap}(\R^n),\quad u_0\mapsto\varphi(1,u_0),
\end{equation}
which is a $C^1$-map. It follows from the \eqref{eq:symmetry} and the uniqueness of the solutions of \eqref{eq:ch'}
that for any $u_0\in V_m$, $t\in[0,2]$ and $0<\mu\le 1$, we have $\varphi(\mu t,u_0)=\varphi(t,\mu u_0)$ and hence,
$\varphi(\mu,u_0)=\varphi(1,\mu u_0)$ for any $0<\mu\le 1$.
This, together with the continuity of the map \eqref{eq:initial_data_m} implies that for any $u_0\in V_m$ and 
for any $0\le t\le 2$,
\[
\Exp_m(t u_0)=\varphi(1,t u_0)=\varphi(t,u_0)
\]
which implies that
\begin{equation}\label{eq:d_0Exp}
d_0\Exp_m=\id_{C^m_\ap}.
\end{equation}

Now, for a given $0<\gamma<1$, consider the set of H{\"o}lderian exponents
\[
M_\gamma:=\big\{k+\gamma\,\big|k\ge 3, k\in\Z\big\}.
\]
It will follow from Proposition \ref{prop:ch_extension} below that for any $m\in M_\gamma$ the neighborhoods $V_m$ 
can be chosen so that
\begin{equation}\label{eq:V_m}
V_m:=V_{m_3}\cap C^m_{\ap},\quad m_3:=3+\gamma.
\end{equation}
This allows us to define the associated to the Riemannian metric \eqref{eq:riemannian_metric} on ${\Diff}^\infty_{\ap}(\R^n)$
{\em Riemannian exponential map}
\begin{equation}\label{eq:Exp_infinity}
\Exp_\infty:=\Exp_{m_3}|_{C^\infty_\ap} :\,V_{m_3}\cap C^\infty_\ap\to\Diff^\infty_{\ap}(\R^n).
\end{equation}
Note that for any $m\in M_\gamma$ it follows from \eqref{eq:d_0Exp} and the inverse function theorem in Banach spaces that there exists 
an open neighborhood ${\widetilde V}_m\subseteq V_m$ of zero in $C^m_\ap$ such that the restriction of the exponential map \eqref{eq:Exp_m} to this 
neighborhood is a diffeomorphism onto an open neighborhood of the identity in ${\Diff}^m_{\ap}(\R^n)$. 
Of course, it could happen that
\[
\bigcap_{m\in M_\gamma}{\widetilde V}_m=\emptyset.
\]
We will show, however, that this does not happen, and moreover, we have the following

\begin{theorem}\label{th:Exp_infty}
There exist an open neighborhood $V$ of zero in $C^{\infty}_{\ap}$ and an open neighborhood $U$ of $\id$ in 
$\Diff^{\infty}_{\ap}(\R^n)$ such that 
\[
\Exp_\infty :\,V\to U
\]
is a $C^1_F$-diffeomorphism.
\end{theorem}

As mentioned above, the symbol $C^1_F$ denotes the class of {\em Fr\'echet (continuously) differentiable} maps (see Appendix C).
Note also that the non-degeneracy condition \eqref{eq:d_0Exp} is not sufficient for claiming Theorem \ref{th:Exp_infty}. 
The proof of Theorem \ref{th:Exp_infty} is based on a variant of the inverse function theorem 
in Fr\'echet spaces -- see Theorem \ref{th:inverse_function_theorem} in Appendix C.

\begin{remark}\label{rem:Exp_infty}
In fact, it follows from Remark \ref{rem:higher_smoothness} that $\Exp_\infty : V\to U$ is a $C^k_F$-map for any $k\ge 1$.
\end{remark}

\medskip

We first prove the following

\begin{Prop}\label{prop:ch_extension}
Assume that $m\in M_\gamma$. Then for any $u_0\in V_{m_3}\,\cap\,C^m_{\ap}$ there exists a unique solution $(\varphi,v)$
of (\ref{eq:ch'}) that lies in $C^1\big([0,2),\Diff^m_{\ap}(\R^n)\times C^m_{\ap}\big)$ and
depends $C^1$-smoothly on the initial data $u_0\in V_{m_3}\,\cap\,C^m_{\ap}$ in the sense that the data-to-solution map
\begin{equation}\label{eq:data-to-solution_map}
(t,u_0)\mapsto\big(\varphi(t,u_0),v(t,u_0)\big),\quad 
[0,2)\times\big(V_{m_3}\,\cap\,C^m_{\ap}\big)\mapsto\Diff^m_{\ap}(\R^n)\times C^m_{\ap},
\end{equation}
is $C^1$-smooth.
\end{Prop}

Since equation \eqref{eq:ch'} describes the geodesics of the right-invariant Riemannian metric \eqref{eq:riemannian_metric} on
$\Diff^\infty_\ap(\R^n)$ it is natural to suppose that it is invariant with respect to right translations. 
In fact, we have

\begin{Lemma}\label{lem:symmetry}
Let $m\in M_\gamma$, $T>0$, and $\psi\in\Diff^m_\ap(\R^n)$. Then, if 
the curve 
\[
(\varphi,v)\in C^1\big([0,T],\Diff^m_\ap(\R^n)\times C^m_\ap\big)
\] 
is a solution of \eqref{eq:ch'} then the right-translated by $\psi$ curve
\begin{equation}\label{eq:translated_curve}
(\varphi\circ\psi,v\circ\psi)\in C^1\big([0,T],\Diff^m_\ap(\R^n)\times C^m_\ap\big)
\end{equation}
is again a solution of \eqref{eq:ch'}.
\end{Lemma}

The Lemma follows by a direct substitution of the curve \eqref{eq:translated_curve} into \eqref{eq:ch'}. 
We will omit the proof. 

\begin{proof}[Proof of Proposition \ref{prop:ch_extension}]
We will prove Proposition \ref{prop:ch_extension} by induction on the regularity exponent 
$m\in M_\gamma$. If $m=m_3$ the statement of the Proposition follows from the choice of the neighborhood $V_{m_3}$.
Further, assume that $m>m_3$, $m\in M_\gamma$, and that there exists a unique solution 
\[
(\varphi,v)\in C^1\big([0,2),\Diff^{m-1}_{\ap}(\R^n)\times C^{m-1}_{\ap}\big).
\]
of equation \eqref{eq:ch'}. This, together with the fact that, by Theorem \ref{th:ch_m}, the right hand side 
of \eqref{eq:ch'} is a $C^1$-smooth vector field on $\Diff^{m-1}_\ap(\R^n)\times C^{m-1}_\ap$, implies that 
there exists an open neighborhood $\mathcal{O}_{m-1}$ of $(\varphi,v)|_{t=0}=(\id,u_0)$ in 
$\Diff^{m-1}_\ap(\R^n)\times C^{m-1}_\ap$ such that for any initial data $(\varphi_0,v_0)\in\mathcal{O}_{m-1}$
there exists a unique solution of \eqref{eq:ch'} such that the data-to-solution map
\begin{equation}\label{eq:local_solution}
S : [0,2)\times\mathcal{O}_{m-1}\to\Diff^{m-1}_\ap(\R^n)\times C^{m-1}_\ap,\quad
(t,\varphi_0,u_0)\mapsto\big(\varphi(t,\varphi_0,u_0),v(t,\varphi_0,u_0)\big),
\end{equation}
is $C^1$-smooth. For any $1\le j\le n$ denote by $\tau_j : \R\to\Diff^\infty_\ap(\R^n)$ the one-parameter group of 
translations on $\R^n$,
\begin{equation}\label{eq:tau_k}
\tau_j(s)(x):=x+s e_j,\quad s\in\R,
\end{equation}
where $e_j$ is the standard $j$-th basis vector in $\R^n$. Since $u_0\in C^m_\ap$ it follows from 
Theorem \ref{th:group_regularity} and Remark \ref{rem:composition} that for any $1\le j\le n$ and for $\varepsilon>0$ 
sufficiently close to zero, the curve 
\begin{equation}\label{eq:gamma_k}
\gamma_j : (-\varepsilon,\varepsilon)\to\mathcal{O}_{m-1},\quad s\mapsto\big(\tau_j(s),u_0\circ\tau_j(s)\big),
\end{equation}
is $C^1$-smooth. It follows from Lemma \ref{lem:symmetry} and the uniqueness of the solutions of \eqref{eq:ch'} that 
for any $1\le j\le n$, $t\in[0,2)$, and $s\in(-\varepsilon, \varepsilon)$ we have
\[
S\big(t,\gamma_j(s)\big)=\big(\varphi(t)\circ\tau_j(s),v(t)\circ\tau_j(s)\big).
\]
Since the map \eqref{eq:local_solution} and the curve \eqref{eq:gamma_k} are $C^1$-smooth we conclude that
for any given $t\in[0,2)$ the curve 
\[
(-\varepsilon,\varepsilon)\to\Diff^{m-1}_\ap\times C^{m-1}_\ap,\quad
s\mapsto S\big(t,\gamma_j(s)\big)=\big(\varphi(t)\circ\tau_j(s),v(t)\circ\tau_j(s)\big)
\]
is $C^1$-smooth. Taking the derivative of this curve at $s=0$ we see that for any $1\le j\le n$,
\[
\frac{\partial\varphi(t)}{\partial x_j},\,\,\frac{\partial v(t)}{\partial x_j}\in C^{m-1}_\ap,
\]
which, by Theorem \ref{th:almost_periodic_functions_equivalence_m}, implies that $\big(\varphi(t),v(t)\big)$ in 
$\Diff^m_\ap(\R^n)\times C^m_\ap$. This, together with the fact that \eqref{eq:ch'} is a $C^1$-smooth dynamical system on 
$\Diff^m_\ap(\R)\times C^m_\ap$, then shows that
\[
(\varphi,v)\in C^1\big([0,2),\Diff^m_\ap(\R)\times C^m_\ap  \big). 
\]
The fact that the data-to-solution map \eqref{eq:data-to-solution_map} is a $C^1$-map follows from
the theorem on the dependence on parameters of the solutions of ODEs in Banach spaces,
applied to the $C^1$-smooth dynamical system \eqref{eq:ch'}.
\end{proof}

As mentioned above, it follows from the inverse function theorem and the non-degeneracy condition 
\eqref{eq:d_0Exp} with $m=m_3$ that the open neighborhood $V_{m_3}$ of zero in $C^{m_3}_\ap$ can be chosen 
so that there exists an open neighborhood $U_{m_3}$ of $\id$ in $\Diff^{m_3}_\ap(\R^n)$ such that the map
\[
\Exp\equiv\Exp_{m_3} : V_{m_3}\to U_{m_3}
\]
is a $C^1$-diffeomorphism. In what follows we will assume that such choice is made.
In addition to \eqref{eq:V_m} for any $m\in M_\gamma$ denote
\begin{equation}\label{eq:U_m}
U_m:=U_{m_3}\cap\Diff^m_\ap(\R^n).
\end{equation}
As mentioned above, Proposition \ref{prop:ch_extension} implies that for any $m\in M_\gamma$
the exponential map \eqref{eq:Exp_m} is well defined on $V_m$ and
\begin{equation}\label{eq:Exp_m'}
\Exp_m\equiv\Exp\big|_{C^m_\ap} : V_m\to U_m
\end{equation}
is a $C^1$-map. Summarizing the above, we obtain

\begin{Prop}\label{prop:C^1-map} There exist an open neighborhood $V_{m_3}$ of zero in $C^{m_3}_\ap$ and an open
neighborhood $U_{m_3}$ of identity in $\Diff^{m_3}_\ap(\R^n)$ such that the (Riemannian) exponential map \eqref{eq:Exp_m},
\[
\Exp\equiv\Exp_{m_3} : V_{m_3}\to U_{m_3},
\] 
is a $C^1$-diffeomorphism. Moreover, for any $m\in M_\gamma$ the map \eqref{eq:Exp_m'} where 
$V_m=V_{m_3}\cap C^m_\ap$ and $U_m=V_{m_3}\cap C^m_\ap$ is a $C^1$-map.
\end{Prop}

The following two statements are needed for the proof of Theorem \ref{th:Exp_infty}.

\begin{Prop}\label{prop:onto}
For any $m\in M_\gamma$ the map \eqref{eq:Exp_m'} is onto.
\end{Prop}

\begin{Prop}\label{prop:exact_regularity}
For any $m\in M_\gamma$ and for any $u_0\in V_{m_3}\cap C^{m+1}_\ap$ one has 
\[
d_{u_0}{\Exp}\big(C^m_{\ap}\setminus C^{m+1}_{\ap}\big)\subseteq
C^m_{\ap}\setminus C^{m+1}_{\ap}.
\]
\end{Prop}

Before proceeding with the proofs of these two statements we need to define the auxiliary map $\widetilde{\mathcal{F}}$
(see \eqref{eq:F-tilde} below). It follows from Theorem \ref{th:ch_m}, Remark \ref{rem:higher_smoothness}, and 
Proposition \ref{prop:ch_extension},
that for any $u_0\in V_{m_3}$ there exists an open neighborhood $\mathcal{O}_{m_3}$ of $(\id,u_0)$ in 
$\Diff^{m_3}_\ap(\R^n)\times C^{m_3}_\ap$ such that for any initial data $(\psi,w)\in\mathcal{O}_{m_3}$ there exists a unique 
solution of \eqref{eq:ch'} such that for any $\ell\ge 1$ the data-to-solution map
\[
[0,2)\times\mathcal{O}_{m_3}\to\Diff^{m_3}_\ap(\R^n)\times C^{m_3}_\ap,\quad
(t,\psi,w)\mapsto\big(\varphi(t,\psi,w),v(t,\psi,w)\big),
\]
is a $C^\ell$-map. This allows us to define
\[
E : \mathcal{O}_{m_3}\to\Diff^{m_3}_\ap(\R^n),\quad(\psi,w)\mapsto\varphi(1,\psi,w),
\]
which is a $C^\ell$-map such that $E(\id, w)=\Exp(w)$ for any $(\id, w)\in\mathcal{O}_{m_3}$. Now, consider the map
\[
\mathcal{F} : \mathcal{O}_{m_3}\to\Diff^{m_3}_\ap(\R^n)\times\Diff^{m_3}_\ap(\R^n),\quad 
(\psi,w)\mapsto\big(\psi,E(\psi,w)\big).
\]
The differential of $\mathcal{F}$ at the point $(\id,u_0)$ is
\begin{equation}\label{eq:matrix}
d_{(\id,u_0)}\mathcal{F} =
\begin{pmatrix}
\id_{C^{m_3}_\ap} & 0\\
* & d_{u_0}\Exp
\end{pmatrix}.
\end{equation}
Since by construction $\Exp: V_{m_3}\to U_{m_3}$ is a $C^1$-diffeomorphism, the differential 
$d_{u_0}\Exp : C^{m_3}_\ap\to C^{m_3}_\ap$ is 
an isomorphism. This together with \eqref{eq:matrix} implies that 
\[
d_{(\id, u_0)}\mathcal{F} : C^{m_3}_\ap\times C^{m_3}_\ap\to C^{m_3}_\ap\times C^{m_3}_\ap
\] 
is an isomorphism. Then, by the inverse function theorem in Banach spaces, we can shrink the neighborhood $\mathcal{O}_{m_3}$
of $(\id,u_0)$ in $\Diff^{m_3}_\ap(\R^n)\times C^{m_3}_\ap$ if necessary so that the map
\begin{equation}\label{eq:F}
\mathcal{F} : \mathcal{O}_{m_3}\to\mathcal{W}^{m_3},
\end{equation}
where $\mathcal{W}^{m_3}$ is an open neighborhood of $\big(\id,\Exp(u_0)\big)$ in 
$\Diff^{m_3}_\ap(\R^n)\times\Diff^{m_3}_\ap(\R^n)$,
is a $C^\ell$-diffeomorphism. In particular, we see that the map
\begin{equation}\label{eq:F-tilde}
\widetilde{\mathcal{F}}:=\pi_2\circ\mathcal{F}^{-1} : \mathcal{W}^{m_3}\to C^{m_3}_\ap
\end{equation}
where $\pi_2 : \Diff^{m_3}_\ap(\R^n)\times C^{m_3}_\ap\to C^{m_3}_\ap$ is the projection onto the second component,
is $C^\ell$-smooth and
\[
\widetilde{\mathcal{F}}\big(\id,\Exp(u_0)\big)=u_0.
\]
It follows from Lemma \ref{lem:symmetry} and the uniqueness of the solutions of \eqref{eq:ch'} that
for any $(\varphi_1,\varphi_2)\in\mathcal{W}^{m_3}$ and for any $\tau\in\Diff^{m_3}_\ap(\R^n)$ sufficiently close to 
$\id$ in $\Diff^{m_3}_\ap(\R^n)$ we have 
\begin{equation}\label{eq:F_symmetry}
\widetilde{\mathcal{F}}(\varphi_1\circ\tau,\varphi_2\circ\tau)=\widetilde{\mathcal{F}}(\varphi_1,\varphi_2)\circ\tau.
\end{equation}

\begin{remark}\label{rem:geometry}
The map $\widetilde{\mathcal{F}}$ has a simple geometrical meaning: For a given $u_0\in V_{m_3}$, $\widetilde{\mathcal{F}}$ 
assigns to a pair of diffeomorphisms $\varphi_1,\varphi_2\in U_{m_3}\subseteq\Diff^{m_3}_\ap(\R^n)$ such that $\varphi_1$ is 
sufficiently close to $\id$ and $\varphi_2$ is sufficiently close to $\Exp(u_0)$ in $\Diff^{m_3}_\ap(\R^n)$ a vector 
$w\in C^{m_3}_\ap$ such that the geodesic issuing at $\varphi_1$ with velocity $w$ reaches $\varphi_2$ at time $t=1$. 
The main point is that $\widetilde{\mathcal{F}}$ is a $C^\ell$-map for any given 
$\ell\ge 1$. Note also that by construction $\widetilde{\mathcal{F}}\big(\id,\Exp(u_0)\big)=u_0$.
\end{remark}

Now we are ready to proof Proposition \ref{prop:onto} and Proposition \ref{prop:exact_regularity}.

\begin{proof}[Proof of Proposition \ref{prop:onto}]
If $m=m_3$ then Proposition \ref{prop:onto} holds by the construction of the neighborhoods $V_{m_3}$ and $U_{m_3}$. 
Now, assume that $m> m_3$, $m\in M_\gamma$, and take $\varphi\in U_{m_3}\cap C^m_\ap$. Since by construction 
$\Exp : V_{m_3}\to U_{m_3}$ is a $C^1$-diffeomorphism, there exists $u_0\in V_{m_3}$ such that $\Exp(u_0)=\varphi$. 
Hence, in view of \eqref{eq:F_symmetry}, for any $\tau\in\Diff^{m_3}_\ap(\R^n)$ that is sufficiently close to $\id$ in 
$\Diff^{m_3}_\ap(\R^n)$ we have
\begin{equation}\label{F_symmetry'}
\widetilde{\mathcal{F}}(\tau,\varphi\circ\tau)=\widetilde{\mathcal{F}}(\id,\varphi)\circ\tau=u_0\circ\tau.
\end{equation}
Take an arbitrary vector $\xi\in\R^n$ and consider the  one-parameter group of translations $\tau_\xi : \R\to\Diff^\infty_\ap(\R^n)$,
\begin{equation}\label{eq:tau_xi}
\tau_\xi(s)(x):=x+s\,\xi.
\end{equation}
For $\varepsilon>0$ sufficiently close to zero consider the $C^\ell$-smooth curve (see Theorem \ref{th:group_regularity})
\[
\gamma_\xi : (-\varepsilon,\varepsilon)\to\mathcal{W}_{m_3},\quad s\mapsto\big(\tau_\xi(s),\varphi\circ\tau_\xi(s) \big),
\]
where $\ell:=[m]-3\ge 1$. Since the map \eqref{eq:F-tilde} is $C^\infty$-smooth we conclude that the curve
\[
(-\varepsilon,\varepsilon)\to C^{m_3}_\ap,\quad 
s\mapsto\widetilde{\mathcal{F}}\big(\tau_\xi(s),\varphi\circ\tau_\xi(s)\big)=u_0\circ\tau_\xi(s),
\]
is a $C^\ell$-map. The equality above follows from \eqref{F_symmetry'}.
By taking the $j$-th derivative of this curve at $s=0$ successively for $1\le j\le\ell$, we see that 
for any $1\le j\le\ell$ and for any $\xi\in\R^n$,
\[
[d^j u_0](\xi)\in C^{m_3}_\ap,
\]
where $[d^j u_0]$ denotes the $j$-th differential of $u_0$. 
This together with Theorem \ref{th:almost_periodic_functions_equivalence_m} (ii) then implies that $u_0\in C^m_\ap$.
Then, by Proposition \ref{prop:C^1-map}, $\varphi=\Exp(u_0)\in\Diff^m_\ap(\R^n)$.
This completes the proof of Proposition \ref{prop:onto}.
\end{proof}

\begin{proof}[Proof of Proposition \ref{prop:exact_regularity}]
Assume that $m\in M_\gamma$ and $u_0\in V_{m_3}\cap C^{m+1}_\ap$. 
Since $\Exp : V_{m_3}\to U_{m_3}$ is a $C^1$-diffeomorphism and $u_0\in V_{m_3}$ we conclude that
\begin{equation}\label{eq:dExp}
d_{u_0}\Exp : C^{m_3}_\ap\to C^{m_3}_\ap
\end{equation}
is an isomorphism. Moreover, it follows from Proposition \ref{prop:C^1-map} and $u_0\in V_{m_3}\cap C^{m+1}_\ap$ that 
\[
\varphi:=\Exp(u_0)\in\Diff^{m+1}_\ap(\R^n)
\]
and
\[
d_{u_0}\Exp : C^m_\ap\to C^m_\ap.
\]
Hence, Proposition \ref{prop:exact_regularity} will follow if we show that the inverse of \eqref{eq:dExp} maps $C^{m+1}_\ap$
into $C^{m+1}_\ap$. Let us prove this: Take a variation vector $\delta\varphi\in C^{m+1}_\ap$. Then for $\varepsilon>0$ 
sufficiently small and for any given $\xi\in\R^n$ consider the $C^1$-map
\[
(-\varepsilon,\varepsilon)\times(-\varepsilon,\varepsilon)\to\mathcal{W}_{m_3},\quad 
(t,s)\mapsto\Big(\tau_\xi(s),\big(\varphi+t\delta\varphi\big)\circ\tau_\xi(s)\Big),
\]
where $\tau_\xi(s)\in\Diff^\infty_\ap(\R^n)$ is given by \eqref{eq:tau_xi} and $\mathcal{W}_{m_3}$ is the domain of definition 
of the $C^\infty$-map \eqref{eq:F-tilde}. It follows from \eqref{eq:F_symmetry} that for any $t,s\in(-\varepsilon,\varepsilon)$,
\[
\widetilde{\mathcal{F}}\Big(\tau_\xi(s),\big(\varphi+t\delta\varphi\big)\circ\tau_\xi(s)\Big)=
\widetilde{\mathcal{F}}\big(\id,\varphi+t\delta\varphi\big)\circ\tau_\xi(s).
\]
By taking the partial derivative $\frac{\partial}{\partial t}\big|_{t=0}$ of this map we see that for any
$s\in(-\varepsilon,\varepsilon)$,
\begin{equation}\label{eq:symmetry''}
D_2\widetilde{\mathcal{F}}\big|_{\big(\tau_\xi(s),\varphi\circ\tau_\xi(s)\big)}\big((\delta\varphi)\circ\tau_\xi(s)\big)=
(\delta u_0)\circ\tau_\xi(s),
\end{equation}
where $\delta u_0\in C^{m_3}_\ap$ denotes the preimage of the variation vector $\delta\varphi\in C^{m_3}_\ap$ with respect to the
isomorphism \eqref{eq:dExp}. Since the map \eqref{eq:F-tilde} is $C^\infty$-smooth and since, by 
Theorem \ref{th:group_regularity} and the fact that $\varphi\in\Diff^{m+1}_\ap(\R^n)$ and $\delta\varphi\in C^{m+1}_\ap$, 
the curves
\[
(-\varepsilon,\varepsilon)\to\Diff^{m_3}_\ap(\R^n)\times\Diff^{m_3}_\ap(\R^n),\quad
s\mapsto\big(\tau_\xi(s),\varphi\circ\tau_\xi(s)\big)
\]
and
\[
(-\varepsilon,\varepsilon)\to C^{m_3}_\ap,\quad s\mapsto(\delta\varphi)\circ\tau_\xi(s)
\]
are $C^\ell$-smooth with $\ell:=[m]-2$, we conclude from \eqref{eq:symmetry''} that 
\[
(-\varepsilon,\varepsilon)\mapsto(\delta u_0)\circ\tau_\xi(s)
\]
is a $C^\ell$-map. By taking the $j$-th derivative of this curve at $s=0$ successively for $1\le j\le\ell$ we conclude
that for any $1\le j\le\ell$ and for any $\xi\in\R^n$,
\[
\big[d^j(\delta u_0)\big](\xi)\in C^{m_3}_\ap,
\]
where, as above, $\big[d^j(\delta u_0)\big]$ denotes the $j$-th differential of $\delta u_0$.
This together with Theorem \ref{th:almost_periodic_functions_equivalence_m} (ii) implies that
$\delta u_0\in C^{m+1}_\ap$.  This completes the proof of Proposition \ref{prop:exact_regularity}.
\end{proof}

\medskip

Now, we are ready to prove Theorem \ref{th:Exp_infty} and Corollary \ref{coro:gauss_lemma}.

\begin{proof}[Proof of Theorem \ref{th:Exp_infty}]
Theorem \ref{th:Exp_infty} follows directly from Proposition \ref{prop:C^1-map}, Proposition \ref{prop:onto},
Proposition \ref{prop:exact_regularity} proved above, and Theorem \ref{th:inverse_function_theorem} in Appendix C.
\end{proof}

\begin{proof}[Proof of Corollary \ref{coro:gauss_lemma}]
Take $u,\delta u\in C^\infty_\ap$ and for a given $\varepsilon>0$
consider the one-parameter family of geodesics on $\Diff^\infty_\ap(\R^n)$,
\begin{equation*}
\zeta : (-\varepsilon,\varepsilon)\times[0,1]\to\Diff^\infty_\ap(\R^n),\quad
\zeta_s(t)\equiv\zeta(s,t):=\Exp\big(t(u+s\,\delta u)\big),
\end{equation*}
where for simplicity of notation we set $\Exp\equiv\Exp_\infty$. 
Note that these geodesics start at the identity, $\zeta(s,0)=\id$ for all $s\in(-\varepsilon,\varepsilon)$, but
have not necessarily fixed endpoints $\zeta(s,1)$, $s\in(-\varepsilon,\varepsilon)$. 
We set $\zeta_0:=\zeta(0,t)$, $\dot\zeta_0:=\frac{\partial\zeta}{\partial t}(0,t)$, and 
$\delta\zeta_0:=\frac{\partial\zeta}{\partial s}(0,t)$.
Then, we obtain for the variation of the energy functional \eqref{eq:energy_functional}, 
\begin{equation}\label{eq:variation1}
\frac{d}{ds}\Big|_{s=0}\mathcal{E}(\zeta_s)=\int_0^1\text{\rm EL}(\zeta_0,\dot\zeta_0)(\delta\zeta_0)\,dt
+\nu_\alpha(\dot\zeta_0,\delta\zeta_0)\big|_{t=1}-\nu_\alpha(\dot\zeta_0,\delta\zeta_0)\big|_{t=0},
\end{equation}
where for any given $\xi\in C^\infty_\ap$ and for any $t\in[0,1]$,
\[
\text{\rm EL}\big(\zeta_0(t),\dot\zeta_0(t)\big)(\xi):=(D_1L)\big(\zeta_0(t),\dot\zeta_0(t)\big)(\xi)-
\big[(D_2L)\big(\zeta_0(t),\dot\zeta_0(t)\big)(\xi)\big]^{\LargerCdot},
\]
and $L(\varphi,v):=\frac{1}{2}\nu_\alpha(\varphi)(v,v)$, $(\varphi,v)\in\Diff^\infty_\ap(\R^n)\times C^\infty_\ap$, is the 
Lagrangian.\footnote{Here $D_kL$, $k=1,2$, denotes the partial derivative with respect to the $k$-th argument.}
Since, in our framework, the geodesics  are obtained from the least action principle, they satisfies the variational 
Euler-Lagrange equation, i.e. 
${\rm EL}(\zeta_0,\dot\zeta_0)(\xi)=0$ for any $\xi\in C^\infty_\ap$ and for any $t\in[0,1]$.
Hence, by \eqref{eq:variation1},
\begin{eqnarray}
\frac{d}{ds}\Big|_{s=0}\mathcal{E}(\zeta_s)&=&\nu_\alpha\big(\dot\zeta_0(1),\delta\zeta_0(1)\big)
-\nu_\alpha\big(\dot\zeta_0(0),\delta\zeta_0(0)\big)\label{eq:variation2}\\
&=&\nu_\alpha\big(d_u\Exp(u),d_u\Exp(\delta u)\big)\label{eq:variation3},
\end{eqnarray}
where we used that $\delta\zeta_0(0)=0$.
\begin{remark}[Conservation of Energy]\label{rem:conservation_of_energy}
The variational formula \eqref{eq:variation2} holds, with $\zeta$ replaced by $\varphi$, for any $C^2_F$-variation 
$\varphi : (-\varepsilon,\varepsilon)\times[0,1]\to\Diff^\infty_\ap(\R^n)$ of a given geodesic 
$\varphi_0 : [0,1]\to\Diff^\infty_\ap(\R^n)$.\footnote{By Theorem \ref{th:Exp_infty}, the geodesic 
$\varphi_0$ is defined in some open interval of $\R$ containing $[0,1]$.}
By taking $\varphi(s,t):=\varphi_0(t+s)$ we obtain
\[
\mathcal{E}(\varphi_s)=\frac{1}{2}\int_0^1\nu_\alpha(\varphi_0(t+s))\big(\dot\varphi_0(t+s),\dot\varphi_0(t+s)\big)\,dt
=\frac{1}{2}\int_s^{1+s}\!\!\!\!\!\!\!\nu_\alpha(\varphi_0(t))\big(\dot\varphi_0(t),\dot\varphi_0(t)\big)\,dt
\]
which implies that 
\[
\frac{d}{ds}\Big|_{s=0}\mathcal{E}(\varphi_s)=
\frac{1}{2}\Big(\nu_\alpha\big(\dot\varphi_0(1),\dot\varphi_0(1)\big)
-\nu_\alpha\big(\dot\varphi_0(0),\dot\varphi_0(0)\big)\Big).
\]
By comparing this with \eqref{eq:variation2} we conclude that
\[
\nu_\alpha\big(\dot\varphi_0(1),\dot\varphi_0(1)\big)=\nu_\alpha\big(\dot\varphi_0(0),\dot\varphi_0(0)\big).
\]
Then, by rescaling the initial vector of the geodesic $\varphi_0$, we see that the quantity 
$\nu_\alpha\big(\dot\varphi_0(t),\dot\varphi_0(t)\big)$ is independent of $t$ on the domain of definition of $\varphi_0$.
\end{remark}

Since by Remark \ref{rem:conservation_of_energy}, 
$\nu_\alpha\big(\dot\zeta_s(t),\dot\zeta_s(t)\big)=\langle u+s\,\delta u,u+s\,\delta u\rangle_\alpha$ 
for any $s\in(-\varepsilon,\varepsilon)$ and $t\in[0,1]$, we see that 
$\mathcal{E}(\zeta_s)=\frac{1}{2}\langle u+s\,\delta u,u+s\,\delta u\rangle_\alpha$.
This, together with \eqref{eq:variation3}, then implies that
\begin{equation}\label{eq:gauss_lemma}
\nu_\alpha\big(d_u\Exp(u),d_u\Exp(\delta u)\big)=\langle u,\delta u\rangle_\alpha,
\end{equation}
which is nothing else but an analog of the classical Gauss Lemma in finite dimensional Riemannian geometry.
Standard arguments as in the proof of Lemma 10.6 and Theorem 10.4 in \cite{Milnor} then imply 
that any diffeomorphism $\varphi$ from the neighborhood $U$ of $\id$ in $\Diff^\infty_\ap(\R^n)$
given by Theorem \ref{th:Exp_infty} is connected with the identity $\id$ by a unique minimal among the curves lying 
in $U$ geodesic. The Corollary then follows from the right invariance of the weak Riemannian metric $\nu_\alpha$.
\end{proof}

\section{Almost periodic diffeomorphisms on Lie groups }\label{sec:group_almost_periodic_general}
In this Section we generalize the group of almost periodic diffeomorphisms from $\R^n$ to any 
finite dimensional Lie group $G$, $\dim G=n$. Note that our aim is only to indicate such a generalization by giving the main
definitions and by proving only the very basic results from Section \ref{sec:spaces} and Section \ref{sec:groups_almost_periodic}.
We do not attempt to generalize each of the statements proved in these sections. In particular, we do not treat the
case of H{\"o}lderian exponents. A full generalization of the results proved in this paper to an arbitrary Lie group $G$ 
will require a separate study.

\medskip

Almost periodic functions on an abstract group $G$ were defined by von Neumann in \cite{vonNeumann}.
By definition, a bounded function $f : G\to\C$ is called {\em almost periodic} (in the sense of von Neumann) if
the family of functions 
\[
\mathcal{S}_f:=\big\{f_c : G\to\C\,\big|\,f_c(x):=f(xc),\,\forall x, c\in G\big\}, 
\]
is precompact with respect to the supremum norm $|f|_\infty=\sup_{x\in G}|f(x)|$. Note that this definition
does not require any topology on $G$. Since our interest to almost periodic functions stems from the analysis 
of almost periodic solutions of PDEs on non-compact manifolds, in what follows we will assume that $G$ 
is a finite dimensional Lie group with Lie algebra $\g$. Although not specifically assumed, we will be mostly
interested in the situation when $G$ is a {\em non-compact} Lie group.
As usual, we will identify the Lie algebra $\g$ with the tangent space at the identity $T_eG$. 
For a given $x\in G$ consider the right translation $R_x : G\to G$, $R_x: g\mapsto g x$. 
Then, the differential of $R_x$ at the identity induces an isomorphism $d_e R_x: \g\to T_xG$ that 
allows us to trivialize the tangent bundle $TG$,
\begin{equation}\label{eq:trivialization}
\tau : G\times\g\to TG,\quad (x,\xi)\stackrel{\tau}{\mapsto}\big(x,(d_eR_x)(\xi)\big).
\end{equation}
We will also fix a positive scalar product $\mu_e : \g\times\g\to\R$ on $\g$ and denote by $\mu$ the corresponding 
right-invariant Riemannian metric on $TG$. Denote by $\nabla$ the Levi-Civita connection of $\mu$.
Let $m\ge 0$, $m\in\Z$. By definition, a $C^m$-smooth vector field $X$ on $G$ belongs to the class $C^m_b(TG)$ if
\begin{equation}\label{eq:boundedness_condition}
\|\nabla^j X\|_\mu<\infty,\quad 0\le j\le m,
\end{equation}
where $\nabla^j$ is the $j$-th power of $\nabla$ and $\|\cdot\|_\mu$ is the induced by the Riemannian metric
supremum norm on the space of tensor fields. We will equip $C^m_b(TG)$ with the norm
\begin{equation}\label{eq:C^m-norm_of_X}
\|X\|_m:=\max_{0\le j\le m}\|\nabla^j X\|_\mu.
\end{equation}
Similarly, a $\g$-valued function $F\in C^m(G,\g)$ belongs to the class 
$C^m_b(G,\g)$ if the vector field
\begin{equation}\label{eq:X_F}
X_F:=\big(\tau^{-1}\big)^*(F)
\end{equation}
belongs to $C^m_b(TG)$. One easily checks that $C^m_b(TG)$ is a Banach spaces and that the map
\begin{equation}\label{eq:tau^*}
\tau^* : C^m_b(TG)\to C^m_b(G,\g),\quad X\mapsto F_X:=\tau^*(X),
\end{equation}
is a bijective linear map. The map \eqref{eq:tau^*} trivially becomes an isomorphism of Banach spaces if 
we equip $C^m_b(G,\g)$ with the norm $|F|_m:=\|X_F\|_m$ induced from the norm
\eqref{eq:C^m-norm_of_X} on $C^m_b(TG)$ via the map \eqref{eq:tau^*}.
Following von Neumann, for a given $F\in C^m_b(G,\g)$ consider the set 
\[
\mathcal{S}_F:=\big\{F_c\,\big|\,F_c(x):=F(xc)\big\}_{c\in G}
\] 
and define the space of {\em $C^m$-almost periodic $\g$-valued functions} on $G$,
\begin{equation}\label{eq:K_alpha-general11}
C^m_{\ap}(G,\g):=\big\{F\in C^m_b(G,\g)\,\big|\,\mathcal{S}_F\,
\text{is precompact in}\,\,\,C^m_b(G,\g)\big\}.
\end{equation}
Similarly, for a given vector field $X$ on $G$ consider the set
\[
\mathcal{S}_X:=\big\{X_c\,\big|\,X_c:=\big[(dR_c)(X)\big]\circ R_{c^{-1}}\big\}_{c\in G}
\]
and define the space of {\em $C^m$-almost periodic vector fields} on $G$,
\begin{equation}\label{eq:5}
C^m_{\ap}(TG):=\big\{X\in C^m_b(TG)\,\big|\,\mathcal{S}_X\,
\text{is precompact in}\,\,\,C^m_b(TG)\big\}.
\end{equation}

We have
\begin{Lemma}\label{lem:equivalent_definitions_ap}
Let $X\in C^m_b(TG)$ and $F_X$ be the $\g$-valued function on $G$ associated to
$X$ by \eqref{eq:tau^*}. Then $F_X\in C^m_{\ap}(G,\g)$ if and
only if $X\in C^m_{\ap}(TG)$.
\end{Lemma}
\begin{proof}[Proof of Lemma \ref{lem:equivalent_definitions_ap}]
First, we will prove that for any $c\in G$ we have the following commutative diagram
\begin{equation}\label{eq:diagram}
\begin{tikzcd}
C^m_b(TG)\arrow{r}{\tau^*}&C^m_b(G,\g)\\
C^m_b(TG)\arrow{u}{c}\arrow{r}{\tau^*}&C^m_b(G,\g)\arrow{u}{c^{-1}}
\end{tikzcd}
\end{equation}
where the horizontal arrows represent the map \eqref{eq:tau^*},
$C^m_b(TG)\stackrel{c^{-1}}{\to}C^m_b(TG)$, $X\mapsto X_{c^{-1}}$, and
$C^m_b(TG)\stackrel{c}{\to}C^m_b(TG)$, $X\mapsto X_c$.
In fact, by the definition of $X_c$ and $F_c$, we have that for any $x\in G$, 
\begin{eqnarray*}
F_{X_c}(xc)&=&\big(d_e R_{xc}\big)^{-1}\big(X_c(xc)\big)\\
&=&\big(d_e R_{xc}\big)^{-1}\Big(\big(d_x R_c\big)\big(X(x)\big)\Big)\\
&=&\big(d_{xc} R_{(xc)^{-1}}\big)\big(d_x R_c\big)\big(X(x)\big)\\
&=&\big(d_x R_{x^{-1}}\big)\big(X(x)\big)=F(x).
\end{eqnarray*}
Hence, for any $c, x\in G$, $F_{X_c}(x)=F_{c^{-1}}(x)$, or equivalently,
\begin{equation}\label{eq:commutativity}
\tau^*(X_c)=(\tau^*(X))_{c^{-1}}.
\end{equation}
This proves the commutativity of the diagram.
In particular, \eqref{eq:commutativity} shows that
\begin{equation}\label{eq:vectors<->functions}
\tau^*(\mathcal{S}_X)=\mathcal{S}_F.
\end{equation}
Since $\tau^* : C^m_b(TG)\to C^m_b(G,\g)$ is an isomorphism of Banach spaces,
the Lemma follows from \eqref{eq:vectors<->functions} and the definitions 
\eqref{eq:K_alpha-general11} and \eqref{eq:5}.
\end{proof}

\begin{remark}
Lemma \ref{lem:equivalent_definitions_ap} shows that von Neumann's definition of almost periodic functions 
on abstract groups naturally extends to vector fields on Lie groups.
\end{remark}

Our next goal is to define the group of $C^m_b$-diffeomorphisms of  the Lie group $G$.
To this end, we fix a coordinate chart $U\subseteq\R^n$, $0\in U$,
\[
\varkappa : U\to\mathcal{U},
\] 
where $\mathcal{U}$ is an open neighborhood of identity in $G$, $U$ and $\mathcal{U}$ are precompact in 
$\R^n$ and $G$ respectively, and $\varkappa(0)=e$. 
Take a $C^m$-diffeomorphism $\varphi : G\to G$ and for any given $g\in G$ consider the map
\begin{equation}\label{eq:phi_g}
{\widetilde\varphi}_g:=(R_{\varphi(g)}\circ\varkappa)^{-1}\circ\varphi\circ(R_g\circ\varkappa).
\end{equation}
One easily sees that \eqref{eq:phi_g} defines a map that takes an open neighborhood of zero in $\R^n$ onto an
open neighborhood of zero $\R^n$ so that ${\widetilde\varphi}_g(0)=0$.

\begin{definition}\label{def:C^m_b-diffeomorphisms_groups}
A $C^m$-diffeomorphism $\varphi : G\to G$ belongs to the class $\diff^m_b(G)$ if there exists a positive constant
$C>0$ such that
\[
\big|\big(\partial^\beta{\widetilde\varphi}_g\big)(0)\big|<C
\]
for any multi-index $\beta\in\Z_{\ge 0}$, $0\le|\beta|\le m$, and for any $g\in G$.
\end{definition}
This definition is independent of the choice of the coordinate chart. It follows directly from \eqref{eq:phi_g}
that 
\[
\widetilde{(\varphi\circ\psi)}_g={\widetilde\varphi}_{\psi(g)}\circ{\widetilde\psi}_g
\] 
for any $\varphi,\psi\in\diff^m_b(G)$ and for any $g\in G$. This together with the chain rule then implies that
$\diff^m_b(G)$ is a group. Note also that for any given $x,g\in G$, we have $\widetilde{(R_x)}_g=\id$ and hence
$R_x\in\diff^m_b(G)$. Moreover, one can prove that $\diff^m_b(G)$ is a Banach manifold modeled on
$C^m_b(G,\g)$ (or equivalently on $C^m_b(TG)$). We have

\begin{theorem}\label{th:diff^m_b-topological_group}
$\diff^m_b(G)$ is a topological group under the composition of diffeomorphisms.
\end{theorem}

The proof of the continuity of the composition and the inversion map in $\diff^m_b(G)$ is straightforward and thus omitted.

Now, we are ready to define the group of $C^m$-{\em almost periodic diffeomorphisms} on $G$.
For any $c\in G$ and for any $\varphi\in\diff^m_b(G)$ define
\begin{equation}\label{eq:phi_c}
\varphi_c:=R_{c^{-1}}\circ\varphi\circ R_c.
\end{equation}
Since $R_c\in\diff^m_b(G)$ we conclude that $\varphi_c\in\diff^m_b(G)$
for any $c\in G$.

\begin{remark}
Note that definition \eqref{eq:phi_c} is consistent with the definition of the corresponding quantity in
the case of $\R^n$ (see Lemma \ref{lem:criteria}).
\end{remark}

\begin{definition}
A diffeomorphism $\varphi\in\diff^m_b(G)$ is {\em almost periodic}, $\varphi\in\Diff^m_\ap(G)$, if the set 
\[
\mathcal{S}_\varphi:=\big\{\varphi_c\,\big|\,c\in G\big\}
\]
is precompact in $\diff^m_b(G)$.
\end{definition}

We equip $\Diff^m_\ap(G)$ with the topology coming from $\diff^m_b(G)$. Arguing as in Section \ref{sec:spaces} and
Section \ref{sec:groups_almost_periodic}, one proves that $\Diff^m_\ap(G)$ is a Banach manifold modeled on 
$C^m_\ap(G,\g)$ (or equivalently on $C^m_\ap(TG)$).

\begin{theorem}\label{th:diff^m_ap-topological_group}
$\Diff^m_{\ap}(G)$ is a topological subgroup of $\diff^m_b(G)$.
\end{theorem}

\begin{proof}[Proof of Theorem \ref{th:diff^m_ap-topological_group}]
In view of Theorem \ref{th:diff^m_b-topological_group}, we have
only to show that $\Diff^m_{\ap}(G)$ is closed under the composition and  the inversion map. 
To this end, take $\varphi,\psi\in\Diff^m_\ap(G)$. It follows directly from \eqref{eq:phi_c} that
for any $c\in G$,
\begin{equation}\label{eq:c-shifts_groups}
(\varphi\circ\psi)_c=\varphi_c\circ\psi_c\quad\text{and}\quad\big(\varphi^{-1}\big)_c=\big(\varphi_c\big)^{-1}.
\end{equation}
Now, we argue as in the proof of Theorem \ref{th:group_regularity}.
In view of \eqref{eq:c-shifts_groups}, we have
\begin{equation}\label{eq:c-shifts_groups'}
\mathcal{S}_{\varphi\circ\psi}\subseteq\mathcal{S}_\varphi\circ\mathcal{S}_\psi\quad\text{and}
\quad\mathcal{S}_{\varphi^{-1}}=\imath(\mathcal{S}_\varphi).
\end{equation}
Since, $\mathcal{S}_\varphi\times\mathcal{S}_\psi$ is precompact in $\diff^m_b(G)\times\diff^m_b(G)$ and since,
by Theorem \ref{th:diff^m_b-topological_group}, the composition 
$\circ : \diff^m_b(G)\times\diff^m_b(G)\to\diff^m_b(G)$ is continuous, we conclude that 
$\mathcal{S}_\varphi\circ\mathcal{S}_\psi=\circ\big(\mathcal{S}_\varphi\times\mathcal{S}_\psi\big)$
is precompact in $\diff^m_b(G)$. This together with the first relation in \eqref{eq:c-shifts_groups'} then
implies that $\mathcal{S}_{\varphi\circ\psi}$ is precompact in $\diff^m_b(G)$. Hence, $\varphi\circ\psi\in\Diff^m_\ap(G)$.
The proof that $\varphi^{-1}\in\Diff^m_\ap(G)$ follows from the second relation in \eqref{eq:c-shifts_groups'}
and by similar arguments.
\end{proof}

\section{Appendix A: Well-posedness of the geodesic equation}\label{sec:local-well-posedness}
In this Appendix we summarize several basic results on the well-posedness of the geodesic equations
\eqref{eq:ch} and \eqref{eq:ch'} studied in Section \ref{sec:exponential_map}. 
Some of the statements are part of the first author's Ph.D. thesis.

\medskip\medskip

\noindent{\em The inverse of $A_{\alpha}$}:
A result of Muhamadiev \cite[Theorem 1]{Muham1} implies that for any H{\"o}lderian exponent 
$0<\gamma<1$ the operator 
\begin{equation}\label{eq:A_on_holder}
A_\alpha\equiv (I-\alpha^2\Delta) : C^{2+\gamma}_b\to C^{\gamma}_b
\end{equation}
is an isomorphism of Banach spaces. Denote by 
\begin{equation}\label{eq:A-inverse_on_holder}
A_\alpha^{-1} : C^{\gamma}_b\to C^{2+\gamma}_b
\end{equation}
its inverse. Since the operator $A_\alpha$ is {\em translation invariant}, i.e. for any
$f\in C^{2+\gamma}_b$ and for any $c\in\R^n$ we have that $A_\alpha(f_c)=(A_\alpha f)_c$
where $f_c(x)=f(x+c)$, and since \eqref{eq:A_on_holder} is an isomorphism, we conclude that its inverse 
\eqref{eq:A-inverse_on_holder} is translation invariant. This implies that \eqref{eq:A-inverse_on_holder} commutes 
with the partial differentiations which, together with Remark \ref{rem:banach_algebra_C^n_b} and 
Theorem \ref{th:almost_periodic_functions_equivalence_m}, gives

\begin{Lemma}\label{lem:A_isomorphism}
For any integer $k\ge 0$ and for any H{\"o}lderian exponent $0<\gamma<1$ the operator
\begin{equation*}
A_\alpha : c^{k+2+\gamma}_b\to c^{k+\gamma}_b
\end{equation*}
is an isomorphism of Banach spaces.
\end{Lemma}

As a consequence we obtain

\begin{coro}\label{coro:A_on_ap_isomorphism}
For any integer $k\ge 0$ and for any H{\"o}lderian exponent $0<\gamma<1$ the operator
\begin{equation}\label{eq:A_on_ap}
A_\alpha : C^{k+2+\gamma}_\ap\to C^{k+\gamma}_\ap
\end{equation}
is an isomorphism of Banach spaces.
\end{coro}

\begin{proof}[Proof of Corollary \ref{coro:A_on_ap_isomorphism}]
In view of the translation invariance of the operator
\[
A_\alpha^{-1} : c^{k+\gamma}_b\to c^{k+2+\gamma}_b
\]
we conclude that for any $f\in c^{k+\gamma}_b$ we have
\begin{equation}\label{eq:AS-invariance}
A_\alpha^{-1}(\mathcal{S}_f)=\mathcal{S}_{A_\alpha^{-1}f},
\end{equation}
where $\mathcal{S}_f=\{f_c\}_{c\in\R^n}$ and $\mathcal{S}_f\subseteq c^{k+\gamma}_b$.
Since $\mathcal{S}_f$ is precompact in $c^{k+\gamma}_b$ and since the map 
$A_\alpha^{-1} : c^{k+\gamma}_b\to c^{k+2+\gamma}_b$ is continuous we obtain from \eqref{eq:AS-invariance} 
that $\mathcal{S}_{A_\alpha^{-1}f}$ is precompact in $c^{k+2+\gamma}_b$. 
By the definition of the class of almost periodic functions $C^{k+2+\gamma}_\ap$ we then
conclude that $A_\alpha^{-1}f\in C^{k+2+\gamma}_\ap$. Hence, the map \eqref{eq:A_on_ap} is a bijection onto
its image. By the open mapping theorem we then see that \eqref{eq:A_on_ap} is an isomorphism of Banach spaces.
\end{proof}

\begin{remark}
The statements on the operator $A_\alpha$ above can be also proved directly by using the Littlewood-Paley dyadic 
characterization of the H{\"o}lder space $C^{k+\gamma}_b$ $($see e.g. \cite{AlinhacGerard}$)$.
\end{remark}

\medskip\medskip

\noindent{\em The geodesic equation}:
In what follows we discuss the well-posedness of the geodesic equation \eqref{eq:ch} (and \eqref{eq:ch'}). 
To this end, we first prove

\begin{theorem}\label{th:ode}
Let $m\ge 2$ and $T>0$ be real. Then, for any $u\in C\big([0,T],C^m_{\ap}\big)$ there
exists a unique solution $\varphi\in C^1\big([0,T],\Diff^m_\ap(\mathbb{R}^n)\big)$ of
\begin{equation}\label{eq:ode}
\left\{
\begin{array}{l}
\dot{\varphi}=u\circ\varphi,\\
\varphi|_{t=0}=\id.
\end{array}
\right. 
\end{equation}
\end{theorem}

\begin{remark}\label{rem:ode_dependence_on_u}
It will follow from the proof of Theorem \ref{th:ode} together with Theorem \ref{th:group_regularity} and the theorems on 
the dependence on parameters of the solutions of ODEs in Banach spaces, that in the case when $u\in C^m_\ap$ is 
independent of $t$, the solution $\varphi$ of \eqref{eq:ode}, considered as a map $(t,u)\mapsto\varphi(t,u)$ belongs to
$C\big([0,\infty)\times C^m_\ap,\Diff^m_\ap(\R^n)\big)\cap 
C^1\big([0,\infty)\times C^m_\ap,\Diff^{m-1}_\ap(\R^n)\big)$.
Moreover, in this case, for any given $T>0$ the solution $\varphi$ belongs to $C^1\big([0,T],\Diff^m_\ap(\mathbb{R}^n)\big)$ 
and depends continuously on $u\in C^m_\ap$ in the sense that the map
\[
C^m_\ap\to C^1\big([0,T],\Diff^m_\ap(\mathbb{R}^n)\big),\quad u\mapsto\varphi,
\]
is continuous.
\end{remark}

\begin{remark}
The same results hold with $\Diff^m_\ap(\R^n)$ replaced by $\diff^m_b(\R^n)$.
\end{remark}

\begin{proof}[Proof of Theorem \ref{th:ode}]
Take $u\in C\big([0,T],C^m_{\ap}\big)$ where $m\ge 2$ and $T>0$ and consider the map
\begin{equation}\label{eq:F_rhs}
F : [0,T]\times\Diff^{m-1}_\ap(\R^n)\to C^{m-1}_\ap,\quad F : (t,\varphi)\mapsto u(t)\circ\varphi.
\end{equation}
It follows from Theorem \ref{th:group_regularity} that \eqref{eq:F_rhs} is continuous and locally Lipschitz 
continuous in the second argument. By the existence and uniqueness theorem for ODEs in Banach spaces 
(see e.g. \cite{Lang,Dieudonne}) and the compactness of the interval $[0,T]$ one obtains that there exists $\varepsilon>0$ such that
for any $t_0\in[0,T]$ there exists a unique solution 
\[
\varphi\in C^1\big((t_0-\varepsilon,t_0+\varepsilon)\cap[0,T],\Diff^{m-1}_\ap(\R^n)\big),\quad
\varphi|_{t=t_0}=\id,
\]
of the ordinary differential equation $\dot\varphi=F(t,\varphi)$ in $\Diff^{m-1}_\ap(\R^n)$.
Note that for any given $\psi\in\Diff^{m-1}_\ap(\R^n)$ the right-translation by $\psi$,
\[
R_\psi : \Diff^{m-1}_\ap\to\Diff^{m-1}_\ap(\R^n),\quad\varphi\mapsto\varphi\circ\psi,
\]
is an affine linear map in the coordinate chart of $\Diff^{m-1}_\ap(\R^n)$ and hence, 
by Theorem \ref{th:group_regularity} and Remark \ref{rem:composition},
a $C^\infty$-smooth map. This together with the symmetry $F(t,\varphi\circ\psi)=F(t,\varphi)\circ\psi$ implies that
if for some $t_0\in[0,T]$ the curve $\varphi\in C^1\big((t_0-\varepsilon,t_0+\varepsilon)\cap[0,T],\Diff^{m-1}_\ap(\R^n)\big)$ 
is a solution of $\dot\varphi=F(t,\varphi)$ with initial data $\varphi|_{t=t_0}=\id$ then for any given $\psi\in\Diff^{m-1}_\ap(\R^n)$
the right translated curve $\varphi\circ\psi\in C^1\big((t_0-\varepsilon,t_0+\varepsilon)\cap[0,T],\Diff^{m-1}_\ap(\R^n)\big)$
is a solution of this equation with initial data $\varphi\circ\psi|_{t=t_0}=\psi$.
Since $\varepsilon>0$ is independent of the choice of $t_0\in[0,T]$ and $\psi\in\Diff^{m-1}_\ap(\R^n)$, and since
the solutions of $\dot\varphi=F(t,\varphi)$ are unique we obtain that \eqref{eq:ode} has a unique solution
\begin{equation}\label{eq:varphi_m-1_smooth}
\varphi\in C^1\big([0,T],\Diff^{m-1}_\ap(\R^n)\big).
\end{equation}
This together with \eqref{eq:ode} implies that 
\begin{equation}\label{eq:linearized_ode}
[d\varphi(t)]^{\LargerCdot}=[du(t)]\circ\varphi(t)\cdot[d\varphi(t)],\quad[d\varphi]|_{t=0}=I,
\end{equation}
where $[d\varphi(t)]$ is the Jacobian matrix of $\varphi(t,\cdot) : \R^n\to\R^n$ at time $t\in[0,T]$, $[du(t)]$ is the Jacobian 
matrix of $u(t,\cdot) : \R^n\to\R^n$ at time $t\in[0,T]$, and $(\cdot)^{\LargerCdot}$ is the (pointwise) partial 
derivative with respect to $t$. It follows from Theorem \ref{th:group_regularity} that the map
\[
J : [0,T]\to C^{m-1}_\ap,\quad t\mapsto[du(t)]\circ\varphi(t),
\]
is continuous. The Banach algebra property in $C^{m-1}_\ap$ then implies that for any
given $t\in[0,T]$ the $n\times n$ matrix $J(t)$ with elements in $C^{m-1}_\ap$ defines,
by pointwise multiplication, a bounded linear map $J(t) : C^{m-1}_\ap\to C^{m-1}_\ap$ 
such that $J : [0,T]\to\mathcal{L}\big(C^{m-1}_\ap,C^{m-1}_\ap\big)$ is continuous.
This shows that the linear ODE,
\[
{\dot Y}=J(t) Y,\quad Y|_{t-0}=I,
\]
in the Banach space $C^{m-1}_\ap$ has a unique solution $Y\in C^1\big([0,T],C^{m-1}_\ap\big)$. Now, it follows from
\eqref{eq:linearized_ode} and the uniqueness of solutions of the corresponding ODEs that 
\begin{equation}\label{eq:dvarphi_m-1_smooth}
[d\varphi]\in C^1\big([0,T],C^{m-1}_\ap\big).
\end{equation}
The characterization (ii) of the function space $C^m_\ap$ in Theorem \ref{th:almost_periodic_functions_equivalence_m}
together with \eqref{eq:varphi_m-1_smooth} and \eqref{eq:dvarphi_m-1_smooth} then implies that
$\varphi\in C^1\big([0,T],\Diff^m_\ap(\R^n)\big)$.
\end{proof}

\begin{theorem}\label{th:ch_equivalence}
Let $m\ge 3$, $m\notin\Z$, and $T>0$. Then, there exists a one-to-one correspondence between the
solutions 
$(\varphi, v)\in C^1\big([0,T],\Diff^m_{\ap}(\mathbb{R}^n\big)\times C^m_{\ap}\big)$ of 
the initial value problem
\begin{equation}\label{eq:CH'}
(\dot{\varphi},\dot{v})=\big(v,R_{\varphi}\circ A_\alpha^{-1}\circ
B_\alpha\circ R_{\varphi^{-1}}(v)\big),\quad
\big(\varphi,v\big)|_{t=0}=(\id,u_0),
\end{equation}
where $B_\alpha(u)=A_\alpha\big[(\nabla u)\cdot u\big]-
\nabla[A_\alpha u]\cdot u-\big(\Div u\cdot I+(\nabla u)^{T}\big) A_\alpha u$
for $u\in C^m_{\ap}$, and the solutions $u\in C\big([0,T],C^m_\ap\big)\cap C^1\big([0,T],C^{m-1}_{\ap}\big)$ of
the initial value problem
\begin{equation}\label{eq:CH}
\left\{
\begin{array}{l}
\m_t+(\nabla\m)\cdot u+\big(\Div u\cdot I+(\nabla u)^T\big)\cdot\m=0,\quad\m=A_\alpha u\\
u|_{t=0}=u_0
\end{array}
\right.
\end{equation}
given by the relation $u=v\circ\varphi^{-1}$.
\end{theorem}

\begin{proof}[Proof of Theorem \ref{th:ch_equivalence}]
The proof follows the lines of the proof of Proposition 2.1 in \cite{McOT1} (cf. Proposition 7.1 in \cite{McOT3}). 
\end{proof}

\begin{Prop}\label{prop:smooth_maps}
Let $m\ge 3$, $m\notin\Z$. Then the map $\mathcal{A}_{\alpha}$ defined by
\begin{equation}\label{eq:K_alpha-general25}
\mathcal{A}_{\alpha}:{\Diff}^m_{\ap}(\R^n)\times
C^m_{\ap}\to{\Diff}^m_{\ap}(\R^n)\times
C^{m-2}_{\ap},\,\,(\varphi,v)\mapsto\big(\varphi,R_{\varphi}\circ A_\alpha\circ
R_{\varphi^{-1}}(v)\big),
\end{equation}
is a $C^1$-diffeomorphism with inverse given by
\begin{equation}\label{eq:K_alpha-general24}
\mathcal{A}^{-1}_{\alpha}:{\Diff}^m_{\ap}(\R^n)\times
C^{m-2}_{\ap}\to{\Diff}^m_{\ap}(\R^n)\times
C^m_{\ap},\,\,(\varphi,v)\mapsto\big(\varphi,R_{\varphi}\circ A_\alpha^{-1}\circ
R_{\varphi^{-1}}(v)\big).
\end{equation}
\end{Prop}

The proof of Proposition \ref{prop:smooth_maps} is based on the following Lemma (cf. Lemma 3.3 in \cite{KLT}
for an analog in the Sobolev setting on the 2d torus).

\begin{Lemma}\label{lem:derivatives}
Assume that $w\in C^m_\ap(\R^n,\R)$ and $\varphi\in\Diff_\ap^m(\R^n)$ with $m\ge 2$. Then
\[
\big(R_\varphi\circ\nabla\circ R_{\varphi^{-1}}\big)(w)=\big([d\varphi]^{-1}\big)^T\cdot(\nabla w)
\]
and
\[
\Big(R_\varphi\circ\big(\nabla\nabla^T\big)\circ R_{\varphi^{-1}}\Big)(w)=
\big([d\varphi]^{-1}\big)^T\cdot\nabla\Big(\big(\nabla^T w\big)\cdot[d\varphi]^{-1}\Big)
\]
where $\big(\nabla\nabla^T\big)=\big(\partial^2/\partial x_k\partial x_l\big)_{1\le k,l\le n}$ and 
$\nabla^T=(\partial_{x_1},...,\partial_{x_n})$. The expressions appearing on the right hand sides of these formulas 
have components in $C^{m-2}_\ap(\R^n,\R)$.
\end{Lemma}

\begin{proof}[Proof of Lemma \ref{lem:derivatives}]
The first formula follows directly from the chain rule applied to $w\circ\varphi^{-1}$ and
the fact that $\big[d(\varphi^{-1})\big]=[d\varphi]^{-1}\circ\varphi$.
The second formula follows from the first one and its transpose. The last statement of the Lemma follows from 
Proposition \ref{prop:banach_algebra}, Lemma \ref{lem:division}, and the fact that 
$[d\varphi]^{-1}=\mathop{\rm Adj}[d\varphi]/\det[d\varphi]$ where $\mathop{\rm Adj} M$ denotes 
the cofactor matrix of an $n\times n$ matrix $M$.
\end{proof}

\begin{proof}[Proof of Proposition \ref{prop:smooth_maps}]
It follows from Lemma \ref{lem:derivatives} that $R_{\varphi}\circ A_\alpha\circ R_{\varphi^{-1}}(v)$ is a polynomial in $v$, $\varphi$, 
their derivatives up to order two, and $1/\det[d\varphi]$. Then, in view of Proposition \ref{prop:banach_algebra} and Lemma \ref{lem:division}, 
this polynomial expression defines a real analytic map
\[
\Diff_\ap^m(\R^n)\times C^m_\ap\to C^{m-2}_\ap,\quad(\varphi,v)\mapsto R_{\varphi}\circ A_\alpha\circ R_{\varphi^{-1}}(v).
\]
In particular, the map \eqref{eq:K_alpha-general25} is $C^1$-smooth (cf. Proposition 3.1 in \cite{KLT} for a similar proof in 
the Sobolev setting on the 2d torus).
One easily sees by inspection that the mapping $\mathcal{A}^{-1}_{\alpha}$ defined in 
\eqref{eq:K_alpha-general24} is the inverse of $\mathcal{A}_{\alpha}$. 
For any $(\varphi_0, v_0)\in\Diff^m_{\ap}\times C^m_{ap}$, the differential of $\mathcal{A}_{\alpha}$
at $(\varphi_0,v_0)$ is
\begin{equation}\label{eq:K_alpha-general27}
d_{(\varphi_0,v_0)}\mathcal{A}_{\alpha}(\delta\varphi, \delta v)=
\begin{pmatrix}
\delta\varphi & 0\\
\ast & R_{\varphi_0}\circ A_\alpha\circ
R_{\varphi_0}^{-1}\delta v
\end{pmatrix}
\end{equation}
where $\delta\varphi$ and $\delta v$ are in $C^m_\ap$.
It follows from Corollary \ref{coro:A_on_ap_isomorphism} and Theorem \ref{th:group_regularity} that 
$R_{\varphi_0}\circ A_\alpha\circ R_{\varphi_0}^{-1} : C^m_{\ap}\to C^{m-2}_{\ap}$ 
is bounded and has $R_{\varphi_0}\circ A_\alpha^{-1}\circ R_{\varphi_0}^{-1}$ as its (bounded) inverse. 
Hence, by \eqref{eq:K_alpha-general27}, the differential $d_{(\varphi_0,v_0)}\mathcal{A}_{\alpha}$ is a linear isomorphism. 
By the inverse function theorem, $\mathcal{A}_{\alpha}$ is therefore a local $C^1$-diffeomorphism. Since it is also bijective, 
we obtain that $\mathcal{A}_{\alpha}$ is a $C^1$-diffeomorphism, with inverse given by (\ref{eq:K_alpha-general24}).
\end{proof}

\begin{Prop}\label{prop:smooth_vector_field}
Let $m\ge 3$, $m\notin\Z$.  Then, the map $\mathcal{B}_\alpha$ defined by
\begin{equation}\label{eq:B_alpha}
\mathcal{B}_\alpha :{\Diff}^m_{\ap}(\R^n)\times
C^m_{\ap}\to{\Diff}^m_{\ap}(\R^n)\times
C^{m-2}_{\ap},\,\,(\varphi,v)\mapsto\big(\varphi,R_{\varphi}\circ B_{\alpha}\circ
R_{\varphi^{-1}}(v)\big),
\end{equation}
is $C^1$-smooth. 
\end{Prop}

\begin{proof}[Proof of Proposition \ref{prop:smooth_vector_field}]
First, note that $A_\alpha\big[(\nabla u)\cdot u\big]=\nabla[A_\alpha u]\cdot u+\cdots$ where $\cdots$ stands for 
a polynomial expression of $u$ and its derivatives up to order two. This implies that $B_\alpha(u)$ is a polynomial expression of
$u$ and its derivatives up to order two. Then, arguing as in the proof of Proposition \ref{prop:smooth_maps} we see from Lemma \ref{lem:derivatives} 
that $R_{\varphi}\circ B_\alpha\circ R_{\varphi^{-1}}(v)$ is a polynomial in $v$, $\varphi$, their derivatives up to order two, and $1/\det[d\varphi]$.
Proposition \ref{prop:banach_algebra} and Lemma \ref{lem:division} then imply that the map \eqref{eq:B_alpha} is real analytic and hence $C^1$-smooth.
\end{proof}

As a consequence of Proposition \ref{prop:smooth_maps} and 
Proposition \ref{prop:smooth_vector_field} we obtain

\begin{Prop}\label{prop:smooth_ode}
Let $m\ge 3$, $m\notin\Z$. Then the vector field on the right hand side of \eqref{eq:CH'},
\[
\Diff^m_\ap(\R^n)\times C^m_\ap\to C^m_\ap\times C^m_\ap,\quad 
(\varphi,v)\mapsto\big(v,R_{\varphi}\circ A_\alpha^{-1}\circ
B_\alpha\circ R_{\varphi^{-1}}(v)\big),
\]
is $C^1$-smooth.
\end{Prop}

\begin{remark}\label{rem:real_analytic}
The arguments in the proof of Lemma 3.2 and Proposition 4.3 in \cite{McOT1} show that the vector field in
Proposition \ref{prop:smooth_ode} is actually real analytic $($cf. \cite{McOT3}$)$.
Note also that the condition $m\notin\Z$ is not needed in the case when $n=1$.
The reason is that in this case the operator $
A_\alpha\equiv(1-\alpha^2\partial_x^2) : C^{m+2}_b\to C^m_b$ is
an isomorphism of Banach spaces for any real $m\ge 0$.
\end{remark}

\begin{proof}[Proof of Theorem \ref{th:ch_m}]
The Theorem follows from Proposition \ref{prop:smooth_ode}, Theorem \ref{th:ch_equivalence},
the theorems on the existence, uniqueness and dependence on parameters of the solutions of ODE's in Banach spaces 
(see e.g. \cite{Lang,Dieudonne}), and Lemma \ref{lem:general} below.
\end{proof}

\begin{Lemma}\label{lem:general}
Let $X$ and $Y$ be Banach spaces, $U\subseteq X$ an non-empty open set, and let 
$\Gamma : [0,T]\times U\to Y$ be a continuous map. Then the map
\[
U\to C\big([0,T],Y\big),\quad u\mapsto\Gamma(\cdot,u),
\]
is continuous.
\end{Lemma}

The proof of this Lemma follows easily from the continuity of $\Gamma$ and the compactness
of the interval $[0,T]$. The Lemma is obvious if $X$ is finite dimensional.

\section{Appendix B: Auxiliary results}\label{sec:auxiliary_results}
In this Appendix we collect several technical results and proofs used in the main body of the paper.
We first prove the following characterization of the {\em little H{\"o}lder space} $c^m_b(\R^n,\R)$ with real exponent $m\ge 0$. 
Let $\gamma:=m-[m]$ where $[m]$ is the largest integer that is less or equal than $m$. 
Recall that by definition, the space $c^m_b(\R^n,\R)$ is the closure of $C^\infty_b(\R^n,\R)$ 
in $C^m_b(\R^n,\R)$, where $C^m_b(\R^n,\R)$ is the standard  H{\"o}lder space with exponent $m\ge 0$ -- see
Section \ref{sec:spaces}. Let $\chi\in C^\infty_c(\R^n,\R)$ be a cut-off function such that $\chi\ge 0$, 
$\mathop{\rm supp}\chi\subseteq\big\{x\in\R^n\,\big|\,|x|\le 1\big\}$, and $\int_{\R^n}\chi(x)\,dx=1$. 
For any $\varepsilon>0$ denote $\chi_\varepsilon(x):=\chi(x/\varepsilon)/\varepsilon^n$.
We have

\begin{Lemma}\label{lem:charachterization_little_holder}
For any given real exponent $m\ge 0$, $m\notin\Z$, one has that $f\in c^m_b(\R^n,\R)$ if and only if $f\in C^{[m]}_b(\R^n,\R)$ and 
$\partial^\beta f\in c^\gamma_b(\R^n,\R)$ for any multi-index $\beta\in\Z^n_{\ge 0}$ with $|\beta|=[m]$.
In this case, $f*\chi_\varepsilon\in C^\infty_b(\R^n,\R)$ and $f*\chi_\varepsilon\stackrel{C^m_b}{\to} f$ as $\varepsilon\to0+$.
\end{Lemma}

\begin{proof}[Proof of Lemma \ref{lem:charachterization_little_holder}]
The direct statement of the Lemma follows easily from the definition of the norm in $C^m_b(\R^n,\R)$.
Let us prove the converse statement. 
Assume that 
\begin{equation}\label{eq:two_formulas}
f\in C^{[m]}_b(\R^n,\R)\quad\text{and}\quad\partial^\beta f\in c^\gamma_b(\R^n,\R)
\end{equation}
for any multi-index $\beta\in\Z^n_{\ge 0}$ with $|\beta|=[m]$.
By writing $m=[m]+\gamma$ with $0<\gamma<1$, 
we obtain from \eqref{eq:two_formulas}, the properties of the mollifiers, and the fact that $\partial^\beta f$
is uniformly continuous for $|\beta|\le[m]$ (since $\gamma>0$) that $f*\chi_\varepsilon\in C^\infty_b(\R^n,\R)$,
\begin{equation}\label{eq:approximation_integer_part}
f*\chi_\varepsilon\stackrel{C^{[m]}_b}{\to}f\quad\text{as}\quad\varepsilon\to 0+.
\end{equation}
Assume that for any given multi-index with $|\beta|=[m]$ there exists a sequence $(f_k)_{k\ge 1}$ in 
$C^\infty_b(\R^n,\R)$ such that 
\begin{equation}\label{eq:f^beta_approximation}
f_k\stackrel{C^\gamma_b}{\to}\partial^\beta f\quad\text{as}\quad k\to\infty.
\end{equation}
For any given $k\ge 1$ and $\varepsilon>0$ we have
\begin{eqnarray}
|\partial^\beta f-\partial^\beta f*\chi_\varepsilon|_\gamma&\le&|\partial^\beta f-f_k|_\gamma+|f_k-f_k*\chi_\varepsilon|_\gamma+
|(f_k-\partial^\beta f)*\chi_\varepsilon|_\gamma\nonumber\\
&\le&2|\partial^\beta f-f_k|_\gamma+|f_k-f_k*\chi_\varepsilon|_1\label{eq:f^beta_triangle},
\end{eqnarray}
where we used that for any $g\in C^\gamma_b(\R^n,\R)$, $|g*\chi_\varepsilon|_\infty\le|g|_\infty$ and
\[
[g*\chi_\varepsilon]_\gamma\le
\sup_{h\ne 0,x\in\R^n}\int_{\R^n}\frac{\big|g(x-y+h)-g(x-y)\big|}{|h|^\gamma}\,\chi_\varepsilon(y)\,dy
\le[g]_\gamma.
\]
Since $f_k\in C^2_b(\R^n,\R)$ for any $k\ge 1$, we obtain that $|f_k-f_k*\chi_\varepsilon|_1\to 0$ as $\varepsilon\to 0+$.
This, together with \eqref{eq:f^beta_approximation}, \eqref{eq:f^beta_triangle}, then implies that
\[
\partial^\beta f*\chi_\varepsilon\stackrel{C^\gamma_b}{\to}\partial^\beta f\quad\text{as}\quad\varepsilon\to 0+.
\]
By combining it with \eqref{eq:approximation_integer_part} and 
$\partial^\beta\big(f*\chi_\varepsilon\big)=(\partial^\beta f)*\chi_\varepsilon$
we then conclude that $f*\chi_\varepsilon\stackrel{C^m_b}{\to}f$ as $\varepsilon\to 0+$.
This proves that $f\in c^m_b(\R^n,\R)$ and hence, completes the proof of the equivalence statement in the Lemma. 
It also proves the second statement of the Lemma.
\end{proof}

\begin{remark}\label{rem:integer_exponents}
If  the exponent $m\ge 0$ is integer then $c^m_b(\R^n,\R)=C^m_b(\R^n,\R)$.
One can easily see this by using the Littlewood-Paley dyadic partition of unity on $\R^n$ and the fact that mollifiers (uniformly)
approximate $C^m_b(\R^n,\R)$ functions on compact sets of $\R^n$ (cf. \cite{AlinhacGerard}).
\end{remark}

One has the following complementary characterization of the space $c^\gamma_b(\R^n,\R)$ where $0<\gamma<1$
(cf. e.g. \cite[\S 2]{MisYon}). 

\begin{Lemma}\label{lem:characterization_little_holder*}
For any given $0<\gamma<1$ one has that $f\in C^\gamma_b(\R^b,\R)$ belongs to $c^\gamma_b(\R^n,\R)$ if and only if
\begin{equation}\label{eq:characterization_little_holder*}
\lim_{\delta\to 0+}\sup_{0<|h|<\delta,x\in\R^n}\frac{\big|f(x+h)-f(x)\big|}{|h|^\gamma}\,=\,0.
\end{equation}
\end{Lemma}
For the convenience of the reader we include the proof of this Lemma.

\begin{proof}[Proof of Lemma \ref{lem:characterization_little_holder*}]
For $f\in C^\gamma_b(\R^b,\R)$ denote 
\[
[f]_{\gamma,<\delta}:=\sup_{0<|h|<\delta,x\in\R^n}\frac{\big|f(x+h)-f(x)\big|}{|h|^\gamma}<\infty,
\]
\[
[f]_{\gamma,\ge\delta}:=\sup_{|h|\ge\delta,x\in\R^n}\frac{\big|f(x+h)-f(x)\big|}{|h|^\gamma}<\infty.
\]
These quantities obviously satisfy the triangle inequality and $[f]_\gamma\le [f]_{\gamma,<\delta}+[f]_{\gamma,\ge\delta}$.
Assume that $f\in C^\gamma_b(\R^n,\R)$ and that
\begin{equation}\label{eq:limit_exists1}
f_k\stackrel{C^\gamma_b}{\to} f\quad\text{as}\quad k\to\infty
\end{equation}
for some sequence $(f_k)_{k\ge 1}$ in $C^\infty_b(\R^n,\R)$.
We have 
\begin{eqnarray}
[f]_{\gamma,<\delta}&\le&[(f-f_k)+f_k]_{\gamma,<\delta}\le [f-f_k]_{\gamma,<\delta}+[f_k]_{\gamma,<\delta}\nonumber\\
&\le&[f-f_k]_{\gamma}+|f_k|_1\delta^{1-\gamma}\label{eq:small1}
\end{eqnarray}
where we used the mean value theorem to conclude $[f_k]_{\gamma,<\delta}\le |f_k|_1\delta^{1-\gamma}$.
Now, take $\varepsilon>0$ and by \eqref{eq:limit_exists1} choose $k\ge 1$ such that $[f-f_k]_{\gamma}\le\varepsilon/2$.
Then, take $\delta_0>0$ such that $|f_k|_1\delta_0^{1-\gamma}\le\varepsilon/2$. 
In view of \eqref{eq:small1} we then conclude that $[f]_{\gamma,<\delta}\le\varepsilon$ for any $0<\delta<\delta_0$. 
This proves the direct statement of the Lemma.

Now, assume that $f\in C^\gamma_b(\R^n,\R)$ and \eqref{eq:characterization_little_holder*} holds.
Let $\chi_\lambda$, $\lambda>0$, be the cut-off function defined in the beginning of the Appendix.
Since $f\in C^\gamma_b(\R^n,\R)$ the map $f : \R^n\to\R$ is uniformly continuous and hence, by the properties of the mollifiers,
\begin{equation}\label{eq:limit_exists2}
f*\chi_\lambda\stackrel{C_b}{\to} f\quad\text{as}\quad\lambda\to 0+.
\end{equation}
In addition, we have
\begin{eqnarray}
[f*\chi_\lambda-f]_\gamma&\le&[f*\chi_\lambda-f]_{\gamma,<\delta}+[f*\chi_\lambda-f]_{\gamma,\ge\delta}\nonumber\\
&\le&[f*\chi_\lambda]_{\gamma,<\delta}+[f]_{\gamma,<\delta}+\frac{2}{\delta^\gamma}|f*\chi_\lambda-f|_\infty\nonumber\\
&\le&2[f]_{\gamma,<\delta}+\frac{2}{\delta^\gamma}|f*\chi_\lambda-f|_\infty\label{eq:small2}
\end{eqnarray}
where we used that $[f*\chi_\lambda]_{\gamma,<\delta}\le[f]_{\gamma,<\delta}$ and
$[g]_{\gamma,\ge\delta}\le\frac{2}{\delta^\gamma}|g|_\infty$ for any $g\in C_b$.
Arguing as above, we conclude from \eqref{eq:limit_exists2} and \eqref{eq:small2} that $f*\chi_\lambda\stackrel{C_b^\gamma}{\to} f$
as $\lambda\to 0+$. Since, $f*\chi_\lambda\in C^\infty_b(\R^n,\R)$ we see that $f\in c_b^\gamma(\R^n,\R)$.
\end{proof}

\begin{remark}\label{rem:counterexamples}
As a consequence from Lemma \ref{lem:characterization_little_holder*} one sees that the function
\[
f(x):=
\left\{
\begin{array}{cl}
\sqrt{|x|},&x\in[-1,1],\\
1,&\text{\rm else}.
\end{array}
\right.
\]
belongs to $C^{1/2}_b(\R,\R)$ but {\em not} to $c^{1/2}_b(\R,\R)$. 
This example can be easily modified so that $f$ is $1$-periodic, lies in $C^{1/2}_b(\R,\R)$ but not in $c^{1/2}_b(\R,\R)$.
In particular, by Proposition \ref{prop:approximation}, $f$ is $1$-periodic in $C^{1/2}_b(\R,\R)$ but {\em not} 
$C^{1/2}$-almost periodic.
\end{remark}

\medskip

Further, we prove Theorem \ref{th:diff^m-topological_group} formulated in Section \ref{sec:groups_almost_periodic}.
\begin{proof}[Proof of Theorem \ref{th:diff^m-topological_group}]
Assume that $m$ is real and $m\ge 1$. We will first prove that $\diff^m_b(\R^n)$ is a topological group.
Note that for $0<m<1$ this follows from Lemma 2.1 and Lemma 2.2 in \cite{MisYon}.
Here we give the proof for any $m\ge 1$.

\begin{remark}
Note that Lemma \ref{lem:almost_lipschitz} and Lemma \ref{lem:almost_bilinear} below can be also deduced from 
\cite[Theorem 2.7]{NenRainer}.
\end{remark}

By definition,  the norm in $c^m_b(\R^n,\R^n)$ is given by $|f|_m:=\sup_{1\le k\le n}|f_k|_m$ where $f_k\in c^m_b(\R^n,\R)$ are 
the components of the map $f\in c^m_b(\R^n,\R^n)$ and the norm in $c^m_b(\R^n,\R)$ is defined in \eqref{eq:C^m-norm}.


\begin{Lemma}\label{lem:almost_lipschitz}
For any $\varphi_0\in\diff^m_b(\R^n)$ with $m\ge 1$ real there exist an open neighborhood $U(\varphi_0)$ of $\varphi_0$ in
$\diff^m_b(\R^n)$ and a constant $C>0$ such that for any $\varphi\in U(\varphi_0)$ and for any $g\in c^m_b$
one has $g\circ\varphi\in c^m_b$ and
\[
|g\circ\varphi|_m\le C\,|g|_m.
\]
\end{Lemma}

\begin{proof}[Proof of Lemma \ref{lem:almost_lipschitz}]
We will prove this Lemma by induction in the regularity exponent $m\ge 1$.
Note that  for any $\varphi\in\diff^m_b(\R^n)$, $g\in c^m_b$, and for any two indices $1\le j, k\le n$,
\begin{equation}\label{eq:product_rule}
\partial_j(g_k\circ\varphi)=\sum_{l=1}^n
(\partial_l g_k)\circ\varphi\cdot\partial_j\varphi_l.
\end{equation}
Assume that $m=1+\gamma$ for some $0\le\gamma<1$. 
Then \eqref{eq:product_rule} together with the inequality $|g\circ\varphi|_\infty\le|g|_\infty$ implies that 
\begin{equation}\label{eq:m=1}
|g\circ\varphi|_1\le n |g|_1\big(1+|f|_1\big)
\end{equation}
where $\varphi=\id+f$ and $f\in c^m_b$. In particular, this proves the statement of the Lemma when $m=1$ 
(and $\gamma=0$).
If $0<\gamma<1$ then we obtain from \eqref{eq:product_rule} and Lemma \ref{lem:shift_simple} below that
\begin{eqnarray*}
\big[\partial_j (g_k\circ\varphi)\big]_\gamma&\le&\sum_{l=1}^n\Big(
\big[(\partial_l g_k)\circ\varphi\big]_\gamma|\partial_j\varphi_l|_\infty+
\big|(\partial_l g_k)\circ\varphi\big|_\infty[\partial_j\varphi_l]_\gamma\Big)\\
&\le&\sum_{l=1}^n\Big(
[(\partial_l g_k)]_\gamma\big(1+|f|_1\big)^{\gamma+1}+
|(\partial_l g_k)|_\infty[\partial_j\varphi_l]_\gamma\Big)\\
&\le&n |g|_{1+\gamma}\big(1+|f|_{1+\gamma}\big)^{1+\gamma}.
\end{eqnarray*}
This together with \eqref{eq:m=1} then shows that
\[
|g\circ\varphi|_{1+\gamma}\le 2n |g|_{1+\gamma}\big(1+|f|_{1+\gamma}\big)^{1+\gamma}.
\]
Hence, the statement of the Lemma holds for $m=1+\gamma$ and $0\le\gamma<1$.
Further, assume that $m\ge 2$ and let the statement of the Lemma hold
with $m$ replaced by $m-1$. Then, we obtain from the induction hypothesis, formula \eqref{eq:product_rule}, 
and the Banach algebra property of $c^{m-1}_b$, that for any $1\le j,k\le n$ there exist an open neighborhood 
$U(\varphi_0)$ of $\varphi_0$ in $\diff^m_b(\R^n)$ and positive constants $C$, $C_1$, and $C_2>0$ such that 
for any $\varphi\in U(\varphi_0)$ and for any $g\in c^m_b$ one has
\begin{eqnarray*}
\big|\partial_j (g_k\circ\varphi)\big|_{m-1}&\le&\!\!\!C_1 \sum_{l=1}^n
\big|(\partial_l g_k)\circ\varphi\big|_{m-1}|\partial_j\varphi_l|_{m-1}\\
&\le&C_1 \sum_{l=1}^n C\,|\partial_l g_k|_{m-1}\big(1+|f|_m\big)\nonumber\\
&\le&C_2\,|g|_m.
\end{eqnarray*}
This together with the inequality $|g\circ\varphi|_\infty\le|g|_\infty$ completes the proof of the Lemma.
\end{proof}

The proof of the following Lemma follows directly from the definition of the semi-norm $[\cdot]_\gamma$ and
the mean value theorem in integral form (see e.g. \cite[Lemma 4.1 and 4.3]{MajdaBertozzi}).
\begin{Lemma}\label{lem:shift_simple}
\begin{itemize}
\item[(i)] For any $f,g\in c^\gamma_b(\R^n,\R)$ with $0<\gamma<1$ one has
\[
[fg]_\gamma\le[f]_\gamma|g|_\infty+|f|_\infty[g]_\gamma.
\]
\item[(ii)] For any $\varphi\in\diff^1_b(\R^n)$ and for any $g\in c^\gamma_b(\R^n,\R)$ with $0<\gamma<1$ one has 
\[
[g\circ\varphi]_\gamma\le [g]_\gamma|d\varphi|_\infty^\gamma
\]
where $|d\varphi|_\infty\equiv\max\limits_{1\le j,k\le n}\big|\frac{\partial\varphi_k}{\partial x_j}\big|_\infty$.
\end{itemize}
\end{Lemma}

We will also need the following

\begin{Lemma}\label{lem:almost_bilinear}
For any $\varphi_0\in\diff^m_b(\R^n)$ with $m\ge 1$ real there exist an open neighborhood $U(\varphi_0)$ of $\varphi_0$ in
$\diff^m_b(\R^n)$ and a constant $C>0$ so that for any $g\in c^{m+1}_b$ and for any $\varphi\in U(\varphi_0)$ 
one has
\[
\big|g\circ\varphi-g\circ\varphi_0\big|_m\le C |g|_{m+1}|\varphi-\varphi_0|_m.
\]
\end{Lemma}

\begin{proof}[Proof of Lemma \ref{lem:almost_bilinear}]
Take $\varphi_0\in\diff^m_b(\R^n)$ and chose an open ball  $B(f_0)$ of $f_0\equiv\varphi_0-\id$ in 
$c^m_b$ such that $\varphi_0+B(f_0)\subseteq\diff^m_b(\R^n)$.
Then, for any given $x\in\R^n$,
\begin{eqnarray}
g\big(\varphi(x)\big)-g\big(\varphi_0(x)\big)&=&\int_0^1[d_yg]|_{y=\varphi_0(x)+t\delta\varphi(x)}\cdot\delta\varphi(x)\,dt\nonumber\\
&=&\Big(\int_0^1[d_yg]|_{y=\varphi_0(x)+t\delta\varphi(x)}\,dt\Big)\cdot\delta\varphi(x)\label{eq:mean_value}
\end{eqnarray}
where $\delta\varphi:=\varphi-\varphi_0\in c^m_b$.
Since, $g\in c^{m+1}_b$ and $\varphi,\varphi_0\in c^m_b$ we conclude that 
\[
(t,x)\mapsto[d_yg]|_{y=\varphi_0(x)+t\delta\varphi(x)},\quad[0,1]\times\R^n\to\text{\rm Mat}_{n\times n}(\R),
\]
is a $C^{[m]}$-map. This allows us to differentiate with respect to the $x$-variables under the integral sign and 
to conclude from Lemma \ref{lem:almost_lipschitz} that the open ball $B(f_0)$ in $c^m_b$ can be shrank 
if necessary so that for any $\varphi\in\varphi_0+B(f_0)\subseteq\diff^m_b(\R^n)$ we have
\[
\Big|\int_0^1[d_yg]|_{y=\varphi_0(x)+t\delta\varphi(x)}\,dt\Big|_m\le
\int_0^1\big|[d_yg]|_{y=\varphi_0(x)+t\delta\varphi(x)}\big|_m\,dt\le C\,|g|_{m+1}.
\]
This together with \eqref{eq:mean_value} and the Banach algebra property of $c^m_b$ then implies that
there exists $C>0$ such that for any $\varphi\in\varphi_0+B(f_0)$  and for any
$g\in c^m_b$ one has
\begin{equation}
\big|g\circ\varphi-g\circ\varphi_0\big|_m\le C |g|_{m+1}|\delta\varphi|_m.
\end{equation}
This completes the proof of the Lemma.
\end{proof}

Now, the proof of the continuity of the composition in $\diff^m_b(\R^n)$ is straightforward.
In fact, take $\varphi_0\in\diff^m_b(\R^n)$ and $g_0\in c^m_b$. In view of Lemma \ref{lem:almost_lipschitz}
and Lemma \ref{lem:almost_bilinear} there exists an open neighborhood $U(\varphi_0)$ of $\varphi_0$ in
$\diff^m_b(\R^n)$ and a constant $C>0$ such that inequalities in these Lemmas hold  for any $\varphi\in U(\varphi_0)$
and for any choice of $g\in c^m_b$ respectively $g\in c^{m+1}_b$. Take $\tilde{g}\in C^\infty_b$. Then for any
$\varphi\in U(\varphi_0)$ and for any $g\in c^m_b$ we have
\begin{eqnarray*}
\big|g\circ\varphi-g_0\circ\varphi_0\big|_m&\le&\big|g\circ\varphi-\tilde{g}\circ\varphi\big|_m+
\big|\tilde{g}\circ\varphi-\tilde{g}\circ\varphi_0\big|_m+\big|\tilde{g}\circ\varphi_0-g_0\circ\varphi_0\big|_m\\
&\le&C\,\big|g-\tilde{g}\big|_m+C\,|\tilde{g}|_{m+1}\big|\varphi-\varphi_0\big|_m+C\,\big|\tilde{g}-g_0\big|_m\\
&\le&C\,\big|g-g_0\big|_m+C\,|\tilde{g}|_{m+1}\big|\varphi-\varphi_0\big|_m+2C\,\big|\tilde{g}-g_0\big|_m.
\end{eqnarray*}
Now, take $\varepsilon>0$ and choose a non-zero $\tilde{g}\in C^\infty_b$ such that $|\tilde{g}-g_0|_m\le\varepsilon/(6C)$.
Then, for any $\varphi\in U(\varphi_0)$ and for any $g\in c^m_b$ such that
$|\varphi-\varphi_0|_m\le\varepsilon/(3C|\tilde{g}|_{m+1})$ and $|g-g_0|_m\le\varepsilon/(3C)$ we have
\[
\big|g\circ\varphi-g_0\circ\varphi_0\big|_m\le\varepsilon.
\]
This shows that the map
\[
c^m_b(\R^n,\R)\times\diff^m_b(\R^n)\to c^m_b(\R^n,\R),\quad(\varphi,g)\mapsto g\circ\varphi,
\]
is continuous. Moreover, it follows from the definition of $\diff^m_b(\R^n)$ and the chain rule that for any 
$\varphi,\psi\in\diff^m_b(\R^n)$ there exists $\varepsilon>0$ so that $\det\big[d_x(\varphi\circ\psi)\big]>\varepsilon$ uniformly 
in $x\in\R^n$. Hence, the composition in $\diff^m_b(\R^n)$,
\begin{equation}\label{eq:composition_continuous}
\diff^m_b(\R^n)\times\diff^m_b(\R^n)\to\diff^m_b(\R^n),\quad(\varphi,\psi)\mapsto\varphi\circ\psi,
\end{equation}
is well-defined and continuous.

\begin{remark}
Note that since $\diff^m_b(\R^n)$ is not separable we {\em cannot} apply the result of Montgomery $($see Theorem 2 in \cite{Mont}$)$
to conclude that the continuity of the inversion map $\diff^m_b(\R^n)\to\diff^m_b(\R^n)$, $\varphi\mapsto\varphi^{-1}$,
follows from the continuity of the composition.
\end{remark}

We will prove the continuity of the inversion map in $\diff^m_b(\R^n)$ directly, by using argument similar to the ones
given above. To this end, we will need the following analog of Lemma \ref{lem:almost_lipschitz}.

\begin{Lemma}\label{lem:almost_lipschitz_inverse}
For any $\varphi_0\in\diff^m_b(\R^n)$ with $m\ge 1$ real there exist an open neighborhood $U(\varphi_0)$ of $\varphi_0$ in
$\diff^m_b(\R^n)$ and a constant $C>0$ such that for any $\varphi\in U(\varphi_0)$ and for any $g\in c^m_b$
one has $g\circ\varphi^{-1}\in c^m_b$ and
\[
\big|g\circ\varphi^{-1}\big|_m\le C\,|g|_m.
\]
\end{Lemma}

The proof of this Lemma is similar to the proof of Lemma \ref{lem:almost_lipschitz} and is based on
the following analog of Lemma \ref{lem:shift_simple} (ii).

\begin{Lemma}\label{lem:almost_lipschitz_inverse'}
For any $\varphi_0\in\diff^1_b(\R^n)$, $0<\gamma<1$, there exist an open neighborhood 
$U(\varphi_0)$ in $\diff^1_b(\R^n)$ and $C>0$ such that for any $\varphi\in U(\varphi_0)$ and
for any $g\in c^\gamma_b(\R^n,\R)$ one has that $g\circ\varphi^{-1}\in c^\gamma_b(\R^n,\R)$ and
\[
\big[g\circ\varphi^{-1}\big]_\gamma\le C\,[g]_\gamma.
\]
\end{Lemma}

\begin{proof}[Proof of Lemma \ref{lem:almost_lipschitz_inverse'}]
In view of Lemma \ref{lem:shift_simple} (ii), for any $\varphi\in\diff^1_b(\R^n)$ and for any $g\in c^\gamma_b(\R^n,\R)$,
\begin{equation}\label{eq:almost_done}
\big[g\circ\varphi^{-1}\big]_\gamma\le [g]_\gamma\big|d(\varphi^{-1})\big|_\infty^\gamma.
\end{equation}
We set $\varphi_0=\id+f_0$ and  $\varphi=\id+f$ where $f_0 , f\in c^1_b$.
Since $\varphi_0\in\diff^1_b(\R^n)$ there exists $\varepsilon>0$ such that $\det\big(I+[d_xf_0]\big)>\varepsilon$
uniformly in $x\in\R^n$. This implies that there exists an open neighborhood $U(\varphi_0)$ in $\diff^1_b(\R^n)$ such that
for any $\varphi\in U(\varphi_0)$,
\[
\det\big(I+[d_xf]\big)>\varepsilon,
\]
uniformly in $x\in\R^n$. Then, we have $\big[d(\varphi^{-1})\big]=\big[(d\varphi)\circ\varphi^{-1}\big]^{-1}$, and
\[
\big|d(\varphi^{-1})\big|_\infty=\left|\frac{\mathop{\rm Adj}\big(I+[df]\big)}{\det\big(I+[df]\big)}\circ\varphi^{-1}\right|_\infty
\le (n-1)!\big(1+|df|_\infty\big)^{n-1}/\varepsilon
\]
where $\mathop{\rm Adj} M$ denotes the cofactor matrix of an $n\times n$ matrix $M$ and 
$|M|_\infty:=\max_{1\le k,j\le n}|M_{kj}|_\infty$ when the elements of $M$ are functions on $\R^n$.
This implies that there exists $C>0$ such that $\big|d(\varphi^{-1})\big|_\infty\le C$ uniformly in $\varphi\in U(\varphi_0)$.
Then, the Lemma follows from \eqref{eq:almost_done}.
\end{proof}

\medskip

Now, take $\varphi\in\diff^m_b(\R^n)$ and set $\varphi=\id+f$, $f\in c^m_b$, and
\[
\varphi^{-1}=\id+g.
\] 
Then, we obtain from the relation $\varphi\circ\varphi^{-1}=\id$ that $g=-f\circ\varphi^{-1}$.
Since $f\in c^m_b$, we then conclude from Lemma \ref{lem:almost_lipschitz_inverse} that
\begin{equation}\label{eq:g_in_c^m_b}
g\in c^m_b.
\end{equation}
Moreover, one easily sees from the formula
\[
\det\big[d_y(\varphi^{-1})\big]=1/\det\big(I+[d_xf]\big)\big|_{x=\varphi^{-1}(y)},\quad y\in\R^n,
\]
and the fact that $f\in c^m_b$ that there exists $\varepsilon>0$ such that
$\det[d_y(\varphi^{-1})]>\varepsilon$ uniformly in $y\in\R^n$. This together with \eqref{eq:g_in_c^m_b}
then implies that
\[
\varphi^{-1}\in\diff^m_b(\R^n).
\]
Hence, $\diff^m_b(\R^n)$ is a group such that the composition \eqref{eq:composition_continuous} is continuous. 
Now, we will prove that the inversion map
\begin{equation}\label{eq:inversion_map_continuous}
\diff^m_b(\R^n)\to\diff^m_b(\R^n),\quad\varphi\mapsto\varphi^{-1},
\end{equation}
is continuous. Take $\varphi_0\in\diff^m_b(\R^n)$ and $\widetilde\psi\in\diff^\infty_b(\R^n)$, $\widetilde\psi=\id+\widetilde{g}$. 
Then it follows from Lemma \ref{lem:almost_lipschitz_inverse} and Lemma \ref{lem:almost_lipschitz} that there exists
and open neighborhood $U(\varphi_0)$ of $\varphi_0$ in $\diff^m_b(\R^n)$ and positive constants $C$ and $C_1>0$
such that for any $\widetilde\psi\in\Diff^\infty_b(\R^n)$ and for any $\varphi\in U(\varphi_0)$,
\[
\begin{array}{l}
\big|\varphi_0^{-1}-\varphi^{-1}\big|_m=\big|\big(\varphi_0^{-1}\circ\varphi-\id\big)\circ\varphi^{-1}\big|_m
\le C\big|\varphi_0^{-1}\circ\varphi-\id\big|_m\\
\le C\big|\varphi_0^{-1}\circ\varphi-\widetilde\psi\circ\varphi\big|_m+
C\big|\widetilde\psi\circ\varphi-\widetilde\psi\circ\varphi_0\big|_m+
C\big|\widetilde\psi\circ\varphi_0-\varphi_0^{-1}\circ\varphi_0\big|_m\\
\le C_1\big|\widetilde\psi-\varphi_0^{-1}\big|_m+
C_1|\widetilde{g}|_{m+1}\big|\varphi-\varphi_0\big|_m.
\end{array}
\]
Now, take $\varepsilon>0$ and choose a non-zero $\widetilde{g}\in C^\infty_b$ such that $\widetilde\psi=\id+\widetilde{g}$
lies in $\Diff^\infty_b(\R^n)$ and satisfies the inequality $|\widetilde\psi-\varphi_0^{-1}|_m\le\varepsilon/(2 C_1)$. 
Then, for any $\varphi\in U(\varphi_0)$ such that 
$|\varphi-\varphi_0|_m\le\varepsilon/\big(2 C_1|\widetilde{g}|_{m+1}\big)$
we have
\[
\big|\varphi^{-1}-\varphi_0^{-1}\big|_m\le\varepsilon.
\]
This proves the continuity of the inversion map \eqref{eq:inversion_map_continuous}.
Hence, $\diff^m_b(\R^n)$ is a topological group.

\medskip

Further, we prove the $C^r$-regularity of the composition
\begin{equation}\label{eq:composition_regularity}
\circ : \diff^{m+r}_b(\R^n)\times\diff^m_b(\R^n)\to\diff^m_b(\R^n),\quad(\varphi,\psi)\mapsto\varphi\circ\psi,
\end{equation}
for any $r\ge 0$. We will follow the arguments in the proof of \cite[Proposition 2.9]{IKT}.
Assume that $r\ge 1$ and take $g,\delta g\in c^{m+r}_b$, $\varphi\in\diff^m_b(\R^n)$, and $\delta\varphi\in c^m_b$ so that 
$\varphi+\delta\varphi\in\id+B(\widetilde\varphi)$ where $B(\widetilde\varphi)$ is an open ball centered at 
$\widetilde\varphi\equiv\varphi-\id$ in $c^m_b$ and such that $\id+B(\widetilde\varphi)\subseteq\diff^m_b(\R^n)$. 
One sees from Taylor's formula with remainder in integral form that for any $x\in\R^n$,
\begin{equation}\label{eq:taylor_expansion1}
g\big(\varphi(x)+\delta\varphi(x)\big)=\sum_{|\beta|\le r}\frac{1}{\beta!}
\big(\partial^\beta g\big)\big(\varphi(x)\big){\big(\delta\varphi(x)\big)^\beta}+
\mathcal{R}_1\big(g,\varphi,\delta\varphi\big)(x)
\end{equation}
where the remainder $\mathcal{R}_1\big(g,\varphi,\delta\varphi\big)(x)$ is given by
\begin{equation}\label{eq:remainder1}
\sum_{|\beta|=r}\frac{r}{\beta!}\int_0^1(1-t)^{r-1}
\Big(\big(\partial^\beta g\big)\big(\varphi(x)+t\delta\varphi(x)\big)-\big(\partial^\beta g\big)\big(\varphi(x)\big)\Big)\,dt
\cdot\big(\delta\varphi(x)\big)^\beta.
\end{equation}
Similarly, for any $\varphi$, $\delta\varphi$, and $\delta g$ as above and for any $x\in\R^n$ we have
\begin{equation}\label{eq:taylor_expansion2}
\delta g\big(\varphi(x)+\delta\varphi(x)\big)=\sum_{|\beta|\le r-1}\frac{1}{\beta!}
\big(\partial^\beta \delta g\big)\big(\varphi(x)\big){\big(\delta\varphi(x)\big)^\beta}+
\mathcal{R}_2\big(\varphi,\delta g,\delta\varphi\big)(x)
\end{equation}
where $\mathcal{R}_2\big(\varphi,\delta g,\delta\varphi\big)(x)$ is given by
\begin{equation}\label{eq:remainder2}
\sum_{|\beta|=r}\frac{r}{\beta!}\int_0^1(1-t)^{r-1}
\big(\partial^\beta\delta g\big)\big(\varphi(x)+t\delta\varphi(x)\big)\,dt
\cdot\big(\delta\varphi(x)\big)^\beta.
\end{equation}
It follows from the continuity of the composition and the Banach algebra property of $c^m_b$ that 
the integrals in \eqref{eq:remainder1} and \eqref{eq:remainder2} have convergent Riemann sums in $c^m_b$
and that the equalities \eqref{eq:taylor_expansion1} and \eqref{eq:taylor_expansion2} hold in $c^m_b$.
Summing up these two equalities, we see that for any $g,\delta g\in c^{m+r}_b$, $\varphi\in\diff^m_b(\R^n)$, and for any
$\delta\varphi\in c^m_b$ as above,
\begin{equation}\label{eq:expansion3}
(g+\delta g)\circ(\varphi+\delta\varphi)=g\circ\varphi+\sum_{k=1}^r \frac{1}{k!}P_k\big(g,\varphi\big)(\delta g,\delta\varphi)+
\mathcal{R}\big(g,\varphi,\delta g,\delta\varphi\big)
\end{equation}
where, for $1\le k\le r$,
\[
\frac{1}{k!}\,P_k\big(g,\varphi\big)(\delta g,\delta\varphi):=
\sum_{|\beta|=k}\frac{1}{\beta!}\big(\partial^\beta g\big)\circ\varphi\cdot(\delta\varphi)^\beta+
\sum_{|\beta|=k-1}\frac{1}{\beta!}\big(\partial^\beta \delta g\big)\circ\varphi\cdot(\delta\varphi)^\beta,
\]
and
\begin{eqnarray}
\mathcal{R}\big(g,\varphi,\delta g,\delta\varphi\big)&:=&\mathcal{R}_1(g,\varphi,\delta\varphi\big)+
\mathcal{R}_2(\varphi,\delta\varphi,\delta g\big)\nonumber\\
&=&\sum_{|\beta|=r}\rho_{\beta}(g,\varphi,\delta g,\delta\varphi)\cdot(\delta\varphi)^\beta\label{eq:remainder3}
\end{eqnarray}
with
\begin{eqnarray}
\rho_{\beta}(g,\varphi,\delta g,\delta\varphi)&:=&\frac{r}{\beta!}\int_0^1(1-t)^{r-1}
\Big(\big(\partial^\beta g\big)\circ\big(\varphi+t\delta\varphi\big)-\big(\partial^\beta g\big)\circ\varphi(x)\Big)\,dt\nonumber\\
&+&\frac{r}{\beta!}\int_0^1(1-t)^{r-1}
\big(\partial^\beta\delta g\big)\circ\big(\varphi+t\delta\varphi\big)\,dt\label{eq:remainder3'}
\end{eqnarray}
for any multi-index $\beta\in\Z_{\ge 0}^n$ with $|\beta|=r$.
In view of the Banach algebra property of $c^m_b$ for any $g\in c^{m+r}_b$ and 
$\varphi\in\diff^m_b(\R^n)$,
\[
P_k(g,\varphi)\in\mathcal{P}^k\big(c^{m+r}_b\times c^m_b,c^m_b\big),\quad 1\le k\le r,
\]
where $\mathcal{P}^k(X,Y)$ denotes the Banach space of polynomial maps of degree $k$ from a normed
space $X$ to a Banach space $Y$ supplied with the uniform norm.
Moreover, using again the Banach algebra property of $c^m_b$ and Lemma \ref{lem:almost_bilinear} one 
easily sees that for any $1\le k\le r$ the map
\[
P_k : c^{m+r}_b\times\diff^m_b(\R^n)\to\mathcal{P}^k\big(c^{m+r}_b\times c^m_b,c^m_b\big)
\]
is continuous. For any $\varphi\in\diff^m_b(\R^n)$ denote by $B_\bullet(\widetilde\varphi)$ the ball centered at 
$\widetilde\varphi\equiv\varphi-\id_{R^n}$ in $c^m_b$ of maximal radius so that $
\id+B_\bullet(\widetilde\varphi)\subseteq\diff^m_b(\R^n)$.
Consider the open set $V$ of elements $(g,\varphi,\delta g,\delta\varphi)$ in 
$c^{m+r}_b\times\diff^m_b(\R^n)\times c^{m+r}_b\times c^m_b$,
\[
V:=\Big\{(g,\delta g)\in c^{m+r}_b\times c^{m+r}_b,\varphi\in \diff^m_b(\R^n),\delta\varphi\in c^m_b\,\Big|\,
\varphi+\delta\varphi\in\id+B_\bullet(\widetilde\varphi)\Big\}.
\]
It follows from \eqref{eq:remainder3'} and the continuity of the composition that 
for any $\beta$ with $|\beta|=r$ the map $\rho_\beta : V\to c^m_b$ is continuous. 
This together with \eqref{eq:remainder3} and the Banach algebra property of $c^m_b$ then implies that
\[
\mathcal{R} : V\to \mathcal{P}^k\big(c^{m+r}_b\times c^m_b,c^m_b\big).
\]
Since $\mathcal{R}\big(g,\varphi,0,0\big)=0$ we then conclude from \eqref{eq:expansion3} and the converse to 
Taylor's theorem (see e.g. \cite[Theorem 2.4.15]{abraham2012manifolds}) that \eqref{eq:composition_regularity} is a $C^r$-map.

\medskip

The $C^r$-regularity of the inverse,
\[
\diff^{m+r}_b(\R^n)\to\diff^m_b(\R^n),\quad\varphi\mapsto\varphi^{-1},
\]
follows from an inverse function theorem argument as in the proof of \cite[Proposition 2.13]{IKT}.
In fact, assume that $r\ge 1$ and consider the $C^r$-smooth composition map
\[
\Phi: {\diff}_b^{m+r}(\R^n)\times{\diff}_b^m(\R^n)\to{\diff}_b^m(\R^n),\quad
(\varphi,\psi)\mapsto\varphi\circ\psi.
\]
For any given $\varphi\in\diff^{m+r}_b(\R^n)$ we have
\begin{equation*}
\Phi\big(\varphi,\varphi^{-1}\big)=\id.
\end{equation*}
Denote by $D_2\Phi\big|_{(\varphi,\varphi^{-1})} : c^m_b\to c^m_b$ the partial derivative of $\Phi$
with respect to its second argument at the point $(\varphi,\varphi^{-1})$. A direct computation shows that
\begin{equation}\label{eq:differential1}
D_2\Phi\big|_{(\varphi,\varphi^{-1})} (\delta\psi)=\big[d\varphi\circ\varphi^{-1}\big](\delta\psi)
\end{equation}
for any variation $\delta\psi\in c^m_b$. Since $\varphi\in\diff^{m+r}_b(\R^n)$ we conclude that
$[d\varphi]\in c^m_b$ and there exists $\varepsilon>0$ such that $\det[d_x\varphi]>\varepsilon$
for any $x\in\R^n$. This together with Lemma \ref{lem:division} and the Banach algebra property of $c^m_b$ 
then implies that the inverse of the Jacobian matrix $[d\varphi\circ\varphi^{-1}]$ belongs to $c^m_b$.
Formula \eqref{eq:differential1} and the Banach algebra property of $c^m_b$ then show that the map
\[
D_2\Phi\big|_{(\varphi,\varphi^{-1})} : c^m_b\to c^m_b
\]
is a linear isomorphism. By the implicit function theorem in Banach spaces, there exists an open neighborhood 
$U(\varphi)$ of $\varphi$ in $\diff^{m+r}_b(\R^n)$ and a $C^r$-smooth map
\[
\Gamma : U(\varphi)\to\diff^m_b(\R^n)
\] 
such that 
\[
\Phi\big(\psi,\Gamma(\psi)\big)=\id
\]
for any $\psi\in U(\varphi)$. By the uniqueness of the inverse diffeomorphism we then conclude that 
$\Gamma(\psi)=\psi^{-1}$ for any $\psi\in U(\varphi)$. This shows that the inversion map
\[
\imath: \diff^{m+r}_b(\R^n)\to\diff^m_b(\R^n),\quad\varphi\mapsto\varphi^{-1},
\]
is $C^r$-smooth.
\end{proof}

\section{Appendix C: Calculus in Fr\'{e}chet spaces}\label{sec:frechet_calculus}
In this Appendix we recall some of the main concepts of the calculus in Fr\'{e}chet spaces. 
For a detailed treatment, we refer the reader to \cite{Hamilton} and Appendix A of \cite{KLT}.  

\begin{definition}
A pair $\big(X,\big\{\|\cdot\|_m\big\}_{m\ge 0}\big)$ consisting of a topological
vector space $X$ and a countable system of seminorms
$\big\{\|\cdot\|_m\big\}_{m\ge 0}$ is called a {\em Fr\'{e}chet space} if the topology of $X$ is induced by
$\big\{\|\cdot\|_m\big\}_{m\ge 0}$ and $X$ is Hausdorff and complete.
\end{definition}

\begin{definition} 
Let $f: U\subseteq X\to Y$ be a map from an open subset $U$ of a Fr\'{e}chet space to 
a Fr\'{e}chet space $Y$. For given $x\in U$ and $h\in X$ we say that $f$ is {\em differentiable at $x$ in the direction $h$} if the limit
\[
d_xf(h):=\mathop{\lim}_{\varepsilon\to 0}\Big(f(x+\varepsilon h)-f(x)\Big)/\varepsilon
\]
exists in $Y$. In this case, $d_xf(h)$ is called the {\em directional derivative} of $f$ at $x$ in the direction of $h$.
\end{definition}

\begin{definition}
A map $f:U\to Y$ where $U$ is an open set in $X$ is called {\em $C^1_F$-smooth} if $d_xf(h)$ exists for any 
$x\in U$ and $h\in X$ and the map
\[
df : U\times X\to Y,\quad (x, h)\mapsto d_xf(h),
\]
is continuous. 
\end{definition}

The class of $C^1_F$-smooth maps $f : U\to V$ is denoted by $C^1_F(U,V)$.

\begin{definition}
A map $f:U\to V$ where $U$ and $V$ are open sets in $X$ and $Y$ respectively is called a {\em $C^1_F$-diffeomorphism} if 
it is a homeomorphism and $f$ and its inverse $f^{-1} : V\to U$ are $C^1_F$-smooth.
\end{definition}

\begin{remark}\label{rem:non-degeneracy}
One easily sees from the chain rule that if $f : U\to V$ is a $C^1_F$-diffeomorphism then
for any $x\in U$ the map $d_xf : X\to Y$ is a linear isomorphism.
\end{remark}

The space $C^k_F(U,Y)$ for $k\ge 2$ is defined inductively: If $f\in C^{k-1}_F(U,Y)$ then for
given $(h_1,...,h_k)\in X^k$ one defines the {\em $k$-th (directional) derivative} by
\[
d^k_xf(h_1,...h_k):=\lim_{\varepsilon\to 0}\Big(   
d^{k-1}_{x+\varepsilon h_1}f(h_2,...,h_k)-d^{k-1}_xf(h_2,...,h_k)
\Big)/\varepsilon
\]
provided that the limit exists in $Y$. Then by definition, $f\in C^k_F(U,Y)$ if the directional derivative 
$d^k_xf(h_1,...,h_k)$ exists for any $x\in U$ and $(h_1,...,h_k)\in X^k$ and the map
\[
d^k f : U\times X^k\to Y,\quad
(x,h_1,...,h_k)\mapsto d_x^kf(h_1,...,h_k),
\]
is continuous. We refer to \cite{Hamilton} for details. Here we only note that, as in the classical calculus in Banach spaces, 
$f\in C_F^k(U,Y)$ implies that for any $x\in U$ the $k$'th derivative $d^k_xf : X^k\to Y$ is
a $k$-linear and symmetric. 

\begin{definition}
A sequence of Banach spaces $\big\{\big(X_m, \|\cdot\|_m\big)\big\}_{m\ge 0}$ is called a {\em Banach approximation} of 
a given Fr\'{e}chet space $X$ if
\[
X=\bigcap_{m=0}^\infty X_m,\quad X_0\supseteq X_1\supseteq X_2\supseteq\cdots\supseteq X,
\]
$\|x\|_m\le\|x\|_{m+1}$ for any $x\in X$ and $m\ge 0$, and the sequence of norms $\big\{\|\cdot\|_m\big\}_{m\ge 0}$ 
induces the topology on $X$. 
\end{definition}

For Fr\'{e}chet spaces that admit Banach approximation we have the following version of the inverse function theorem.

\begin{theorem}\label{th:inverse_function_theorem}
Let $X$ and $Y$ be Fr\'{e}chet spaces that admit Banach approximations $\big\{\big(X_m,\|\cdot\|_m\big)\big\}_{m\ge 0}$ 
respectively $\big\{\big(Y_m,|\cdot|_m\big)\big\}_{m\ge 0}$. Let $f : V_0\to U_0$ be a map where $V_0$ and 
$U_0$ are open sets in $X_0$ respectively $Y_0$. For any $m\ge 0$ set 
\[
V_m:= V_0\cap X_m,\quad U_m:= U_0\cap Y_m,
\]
and $V:=V_0\cap X$, $U:=U_0\cap Y$. Finally, assume that for any $m\ge 0$ the following properties hold:
\begin{enumerate}
\item $f: V_0\to U_0$ is a bijective $C^1$-map and for any $x\in V$ the map
$d_xf: X_0\to Y_0$ is a linear isomorphism;
\item $f(V_m)\subseteq Y_m$ and the restriction $f|_{V_m}: V_m\to Y_m$ is a $C^1$-map;
\item the map $f|_{V_m}: V_m\to U_m$ is onto;
\item for any $x\in V$, $d_xf(X_m\setminus X_{m+1})\subseteq Y_m\setminus Y_{m+1}$.
\end{enumerate}
Then $f(V)\subseteq U$ and the map $f_{\infty}:= f|_{V}: V\to U$ is a $C^1_F$-diffeomorphism.
\end{theorem}

\begin{proof}[Proof of Theorem \ref{th:inverse_function_theorem}]
We refer the reader to the proof of Theorem A.5 in Appendix A of \cite{KLT} (cf. \cite{CKKT}).
\end{proof}


\end{document}